\documentclass[11pt]{article}
\pdfoutput=1 

\usepackage[utf8]{inputenc}
\usepackage[T1]{fontenc}
\usepackage[american]{babel}

\usepackage{amsfonts}
\usepackage{amsmath}
\usepackage{amssymb}
\usepackage{dsfont}
\usepackage{bm}
\usepackage{mleftright}

\usepackage{mathdots}
\usepackage{siunitx}
\usepackage{booktabs}
\usepackage{enumitem}
\usepackage{subfigure}

\usepackage{authblk}
\usepackage{geometry}
\usepackage{xcolor}
\usepackage{graphicx}

\usepackage{tikz}
\usetikzlibrary{positioning}
\usetikzlibrary {shapes.geometric}

\usepackage{amsthm}
\usepackage{hyperref}
\usepackage{cleveref}
\usepackage{csquotes} 
\usepackage[backend=biber,
            style=numeric-comp,
            sorting=none,
            giveninits=true,
            maxbibnames=99,
            date=year,      
		hyperref]{biblatex}

\AtEveryBibitem{
  \clearlist{language}
  \clearfield{urlyear}
  \clearfield{urlmonth}
  \clearfield{month}
  \clearfield{day}
  \clearfield{eprintclass}
}

\renewbibmacro*{doi+eprint+url}{%
  \iftoggle{bbx:doi}
  {%
    \iffieldundef{doi}
      {\usebibmacro{url+urldate}}%
      {\printfield{doi}}%
  }
  {%
    \iftoggle{bbx:url}
    {\usebibmacro{url+urldate}}
    {}%
  }%
  \newunit\newblock
  \iftoggle{bbx:eprint}
  {\usebibmacro{eprint}}
  {}%
 }

\hypersetup{
    hidelinks,
    colorlinks
}

\newtheorem{theorem}{Theorem}
\newtheorem{rmk}{Remark}
\newtheorem{lemma}{Lemma}
\newtheorem{corollary}{Corollary}

\setlist[enumerate]{leftmargin=.5in}
\setlist[itemize]{leftmargin=.5in}

\newcommand{\bydef}{:=}

\providecommand{\transp}{}
\renewcommand{\transp}{\mathrm{T}}

\newcommand{\Exp}{\mathbb{E}}
\newcommand{\Var}{\mathbb{V}}
\newcommand{\Cov}{\mathbb{C}}

\newcommand{\R}{\mathbb{R}}
\newcommand{\N}{\mathbb{N}}

\DeclareMathOperator{\mse}{MSE}

\DeclareMathOperator*{\argmin}{arg\,min}
\let\vect\relax
\DeclareMathOperator{\vect}{vec}
\DeclareMathOperator{\tr}{tr}
\DeclareMathOperator{\Span}{span}
\DeclareMathOperator{\diag}{diag}
\DeclareMathOperator{\Diag}{Diag}
\DeclareMathOperator{\Circ}{Circ}

\newcommand{\id}{I}

\newcommand{\bX}{\mathbf{X}}
\newcommand{\bx}{\mathbf{x}}
\newcommand{\bY}{\mathbf{Y}}
\newcommand{\by}{\mathbf{y}}
\newcommand{\bZ}{\mathbf{Z}}
\newcommand{\bz}{\mathbf{z}}
\newcommand{\bR}{\mathbf{R}}
\newcommand{\bP}{\mathbf{P}}
\newcommand{\bQ}{\mathbf{Q}}
\newcommand{\bC}{\mathbf{C}}
\newcommand{\bS}{\mathbf{S}}
\newcommand{\bA}{\mathbf{A}}

\newcommand{\bM}{\mathbf{M}}

\newcommand{\bI}{\mathbf{I}}
\newcommand{\bK}{\mathbf{K}}
\newcommand{\bF}{\mathbf{F}}
\newcommand{\bG}{\mathbf{G}}
\newcommand{\bW}{\mathbf{W}}
\newcommand{\bH}{\mathbf{H}}
\newcommand{\bh}{\mathbf{h}}
\newcommand{\bB}{\mathbf{B}}
\newcommand{\bL}{\mathbf{L}}
\newcommand{\bu}{\mathbf{u}}

\newcommand{\bV}{\mathbf{V}}

\newcommand{\be}{\mathbf{e}}
\newcommand{\ba}{\mathbf{a}}

\newcommand{\bmu}{\bm{\mu}}
\newcommand{\bnu}{\bm{\nu}}

\newcommand{\btheta}{\bm{\theta}}

\newcommand{\bSigma}{\bm{\Sigma}}
\newcommand{\bGamma}{\bm{\Gamma}}
\newcommand{\bDelta}{\bm{\Delta}}
\newcommand{\bLambda}{\bm{\Lambda}}
\newcommand{\bPi}{\bm{\Pi}}
\newcommand{\bzero}{\bm{0}}
\newcommand{\bone}{\bm{1}}

\newcommand*{\indic}{\text{\usefont{U}{bbold}{m}{n}1}}

\newcommand{\pre}{\textnormal{pre}}
\newcommand{\post}{\textnormal{post}}
\newcommand{\transf}{\textnormal{transf}}
\newcommand{\nucum}{\nu^{\textnormal{cml}}}
\newcommand{\bnucum}{\bnu^{\textnormal{cml}}}
\newcommand{\muMC}{\hat{\bmu}^{\textnormal{MC}}}
\newcommand{\muMLMChilb}{\hat{\mu}^{\textnormal{MLMC}}}

\newcommand{\muMLMC}{\hat{\bmu}^{\textnormal{MLMC}}}
\newcommand{\muFMLMC}{\hat{\bmu}^{\textnormal{F-MLMC}}}

\title{%
A filtered multilevel Monte Carlo method for estimating the expectation of cell-centered discretized random fields%
\footnote{%
    \includegraphics[height=2ex,trim=0 .4ex 0 0]{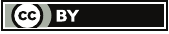} Distributed under a \href{https://creativecommons.org/licenses/by/4.0/}{CC-BY 4.0 licence}.
    This project has received financial support from the CNRS (Centre National de la Recherche Scientifique) through the 80|Prime program and the French national program LEFE (Les Enveloppes Fluides et l'Environnement).
}%
}

\date{}

\author[1,3]{Jérémy Briant}
\author[2,3,4]{Paul Mycek}
\author[5]{Mayeul Destouches}
\author[2,3]{Olivier Goux}
\author[1,6]{Serge Gratton}
\author[2,3]{Selime Gürol}
\author[1]{Ehouarn Simon}
\author[2,3]{Anthony T. Weaver}

\affil[1]{IRIT, UMR~5505, Université de Toulouse, CNRS, Toulouse INP, Toulouse, France (\url{jeremy.briant@toulouse-inp.fr}, \url{serge.gratton@toulouse-inp.fr}, \url{ehouarn.simon@toulouse-inp.fr}).}
\affil[2]{Cerfacs, Toulouse, France (\url{mycek@cerfacs.fr}, \url{goux@cerfacs.fr},
\url{gurol@cerfacs.fr},  \url{weaver@cerfacs.fr}).}
\affil[3]{CECI, UMR~5318, Université de Toulouse, Cerfacs, CNRS, IRD, Toulouse, France.}
\affil[4]{Concace, Airbus CR\&T, Cerfacs, Inria, France.}
\affil[5]{CNRM, UMR~3589, Université de Toulouse, Météo-France, CNRS, Toulouse, France (\url{mayeul.destouches@meteo.fr}).}
\affil[6]{ANITI, Toulouse, France.}

\begin{document}

\maketitle

\begin{abstract}
In this paper, we investigate the use of multilevel Monte Carlo (MLMC) methods for estimating the expectation of discretized random fields.
Specifically, we consider a setting in which the input and output vectors of numerical simulators have inconsistent dimensions across the multilevel hierarchy.
This motivates the introduction of grid transfer operators borrowed from multigrid methods.
By adapting mathematical tools from multigrid methods, we perform a theoretical spectral analysis of the MLMC estimator of the expectation of discretized random fields, in the specific case of linear, symmetric and circulant simulators.
We then propose filtered MLMC (F-MLMC) estimators based on a filtering mechanism similar to the smoothing process of multigrid methods, and we show that the filtering operators improve the estimation of both the small- and large-scale components of the variance, resulting in a reduction of the total variance of the estimator.
Next, the conclusions of the spectral analysis are experimentally verified with a one-dimensional illustration.
Finally, the proposed F-MLMC estimator is applied to the problem of estimating the discretized variance field of a diffusion-based covariance operator, which amounts to estimating the expectation of a discretized random field.
The numerical experiments support the conclusions of the theoretical analysis even with non-linear simulators, and demonstrate the improvements brought by the F-MLMC estimator compared to both a crude MC and an unfiltered MLMC estimator.
\end{abstract}


\section{Introduction}

In recent years, multilevel and multifidelity Monte Carlo (MC) methods (see, e.g., \cite{giles_multilevel_2015,peherstorfer_survey_2018} for surveys on these methods) have grown in popularity as a means to accelerate traditional MC methods by combining simulators of different fidelities to leverage the lower computational cost of lower-fidelity simulators while preserving the accuracy of the high-fidelity simulator.
Such lower fidelity numerical simulators can be obtained, e.g., from a coarser spatial and/or temporal discretization as is the case in multigrid methods \cite{Brandt1982_GuideMultigridDevelopment,Hackbusch1985_MultiGridMethodsApplications,trottenberg_multigrid_2000,briggs_multigrid_2000}.
In such instances, the different fidelities are referred to as levels, and the multifidelity method is then referred to as a multilevel method.
Among multilevel methods, the multilevel Monte Carlo (MLMC) method~\cite{heinrich_multilevel_2001, giles_multilevel_2008, giles_multilevel_2015}, which is the focus of the present paper, combines samples (or ensembles) from different levels in such a way that, under certain assumptions, the variance of the resulting multilevel estimator is reduced, while leaving the bias unchanged.
Originally designed for the estimation of the expected value of scalar, real-valued random variables, the MLMC methodology has since been extended to the estimation of higher-order statistical moments~\cite{bierig_convergence_2015, bierig_estimation_2016} and variance-based global sensitivity measures~\cite{mycek_multilevel_2019}.
The analysis of MLMC estimators has also been extended to the estimation of statistics of random variables with values in separable Hilbert spaces \cite{bierig_convergence_2015, bierig_estimation_2016}.

The estimation of covariance matrices is a prominent example where multilevel approaches have started to emerge, especially in the field of data assimilation \cite{hoel_multilevel_2016, destouches_mlblue_2023,Maurais2023_MultiFidelityCovarianceEstimation,Maurais2025_MultifidelityCovarianceEstimation,Destouches2024_MultilevelMonteCarlo}.
Another example, which will constitute the motivating example of this paper, is embedded in the general approach of using a discretized linear differential operator to represent the application of a parametric form of a covariance matrix.
In particular, we focus on an approach that uses a discretized diffusion operator to represent a covariance operator with a parametric kernel from the Mat\'ern family \cite{Stein_1999}. Diffusion operators are commonly used for modelling spatial covariances in ocean data assimilation \cite{weaver_correlation_2001} and are closely related to other techniques for modelling spatial covariances in atmospheric data assimilation \cite{Purser_2003}, geostatistical modelling \cite{lindgren_explicit_2011}, inverse problems \cite{buithanh13} and uncertainty quantification \cite{Drzisga2017_SchedulingMassivelyParallel}. Central to the approach is the need to extract the diagonal elements ({\it intrinsic} variances) of the diffusion-modelled covariance matrix so that they can be used to normalize the matrix, to transform it into a (unit-diagonal) correlation matrix. 
Once the covariance matrix is properly normalized, a desired variance field, different from the intrinsic variance field, can be imposed.
The standard method for estimating the intrinsic variances of the diffusion-based covariance matrix is the randomization method \cite{weaver_correlation_2001,weaver_evaluation_2021}. 
This method relies on the MC estimation of the expectation of a discretized random field, which, in a more abstract formulation, may be viewed as the output of a numerical simulator, whose input is also a discretized random field.

In this paper, we are interested in applying the MLMC methodology to improve the efficiency of estimating the expectation of cell-centered discretized random fields in the abstract setting described above.
The main specificities of this setting are that, first, the considered simulators are based on a hierarchy of cell-centered discretizations on grids of different resolution; second, the input and output of the numerical simulators are discretized fields whose dimensions depend on the level; and third, we focus on the discrete representation (as opposed to a functional representation) of these discretized fields, arising from the discretization of the strong form of the differential operators at hand (e.g., in our example, diffusion operators).
These specificities motivate the introduction of grid transfer operators, namely restriction and prolongation operators, which may be borrowed from multigrid methods \cite{Brandt1982_GuideMultigridDevelopment, Hackbusch1985_MultiGridMethodsApplications, trottenberg_multigrid_2000, briggs_multigrid_2000}.

Using MLMC with simulators whose inputs and/or outputs are discretized fields has been studied before
\cite{cliffe_multilevel_2011,
  Barth2011_MultilevelMonteCarlo,
  Mishra2012_SparseTensorMultilevel,
  Mishra2012_MultilevelMonteCarlo,
  Abdulle2013_MultilevelMonteCarlo,
  Gittelson2013_MultilevelMonteCarlo,
  Teckentrup2013_FurtherAnalysisMultilevel,
  Charrier2013_FiniteElementError,
  Istratuca2025_SmoothedCirculantEmbedding,
  croci_efficient_2018}, typically in a functional framework based on the finite element (FE) or finite volume (FV) method, which prompts the use of natural grid transfer operators, especially prolongation operators, which are then usually only implicitly defined.
Therefore, their impact on the quality of the resulting multilevel estimator is generally not studied.
Nonetheless, the introduction of grid transfer operators raises crucial questions about their effect on the different spatial scales (or, equivalently, frequencies) of the discretized fields (or signals) that are transferred between grids of different resolution.
We note that the MLMC estimator proposed in \cite{croci_efficient_2018} circumvents the need to use restriction operators by explicitly exploiting the joint distribution of the pair of input fields defined on successive levels within each MLMC correction term, which can be exhibited for the specific problem addressed in \cite{croci_efficient_2018}.

The question of the representation of small scales on coarse grids within the MLMC framework has been explored previously by \cite{Teckentrup2013_FurtherAnalysisMultilevel,Istratuca2025_SmoothedCirculantEmbedding} in an FE setting where the input of the considered simulator is a discretized random field.
The authors proposed the idea of smoothing highly oscillatory input signals based on a spectral truncation of their high-frequency components, which is controlled for each MLMC level so that the truncation error matches the discretization error.
In this paper, we investigate these important questions in the abstract setting described above.
Our objective is to gain a better understanding of the impact of grid transfer operators on the quality of the MLMC estimator in order to propose remedies for their negative effects.
This leads us to propose a novel, filtered MLMC (F-MLMC) estimator that mitigates the negative effects of the grid transfer operators on the quality of the multilevel estimator.
In multigrid methods~\cite{Brandt1982_GuideMultigridDevelopment, Hackbusch1985_MultiGridMethodsApplications,trottenberg_multigrid_2000,briggs_multigrid_2000,wienands_practical_2004}, studying the effects of grid transfer operators on the different frequencies of a discrete signal, which is typically achieved through a spectral analysis, is core to determining the effectiveness of the method.
Along the same lines, we present a spectral analysis of the MLMC and F-MLMC estimators in a simplified setting where the numerical simulators are assumed to be linear, symmetric and circulant operators.

The remainder of this paper is organized as follows.
\Cref{sec:mlmc_for_discretized_fields} introduces the MLMC and F-MLMC estimators of the expectation of random fields as adaptations of the framework in \cite{bierig_convergence_2015} to our specific setting, relying on (possibly filtered) grid transfer operators.
A spectral analysis of the MLMC and F-MLMC estimators is then carried out in \cref{sec:spectral_analysis} to investigate the effects of grid transfer on their variance, in the specific case of linear, symmetric and circulant numerical simulators.
In \cref{sec:experiments}, we apply the MLMC and F-MLMC estimators, first, to a one-dimensional illustration that comply with the assumptions of the spectral analysis, and, second, to our target problem of estimating the intrinsic variances of a two-dimensional, heterogeneous, diffusion-based covariance operator.
General conclusions are drawn in \cref{sec:conclusion}, along with prospective avenues for future work.

\section{MLMC estimation of the expectation of discretized random fields}
\label{sec:mlmc_for_discretized_fields}

The MLMC method aims to improve the accuracy of MC estimators by combining samples of different fidelities. 
In favorable cases, large low-fidelity samples (i.e., with many members that are cheap to generate on coarse levels) are used to reduce the sampling error (variance), while smaller samples are required at finer and more expensive levels to correct the bias. 
Under certain assumptions, \cite[Theorem~1]{cliffe_multilevel_2011} ensures that there exists a sample allocation on a finite number of levels such that the computational cost of the MLMC estimator decreases at a faster rate as a function of the mean square error (MSE), than that of the crude MC estimator. 
In practice, MLMC is typically implemented as a sequential algorithm, whose main idea is to start with a limited number of coarse fidelity levels, and add as many finer levels as needed to reach a target MSE, with a prescribed variance/bias balance. 
In the present work, however, we adopt a multilevel approach that is closer to multifidelity methods. 
Specifically, a fine, high-fidelity level is fixed (and thus so is the bias), while coarser, low-fidelity levels are considered to reduce the variance, for a prescribed computational budget. 

In this paper, we focus on the multilevel estimation of the expectation of cell-centered discretized random fields.
We assume that such discretized random fields can be represented as random variables defined on some probability space $(\Omega, \Sigma, \mathbb{P})$, with values in separable Hilbert spaces and having finite second-order moment.
We thus rely on the mathematical framework proposed in \cite{bierig_convergence_2015} for the design and analysis of MC and MLMC estimators in such a setting. 
Specifically, given a separable Hilbert space $H$ equipped with the inner product $\langle \cdot, \cdot \rangle_H$ and induced norm $\|\cdot\|_H$, the space of second-order $H$-valued random variables on  $(\Omega, \Sigma, \mathbb{P})$,
\begin{equation}
    L^2(\Omega, H) \bydef \{ \xi \colon \Omega \to H \mid \textstyle\int_{\Omega} \| \xi(\omega) \|^2_H \, \mathrm{d} \mathbb{P}(\omega) < +\infty \},
\end{equation}
is a Hilbert space when equipped with the inner product $\langle \cdot, \cdot \rangle_{L^2(\Omega,H)} $ defined as
\begin{equation}
    \forall \xi,\eta \in L^2(\Omega, H),
    \quad
    \langle \xi, \eta \rangle_{L^2(\Omega,H)} 
    \bydef \int_{\Omega} \langle \xi(\omega), \eta(\omega) \rangle_H \, \mathrm{d} \mathbb{P}(\omega),
\end{equation}
whose induced norm is denoted by $\| \cdot \|_{L^2(\Omega,H)}$.
More specifically, we are interested in discretized random fields that correspond to the output of a numerical simulator whose input is also a discretized random field.
Formally, we consider an abstract, deterministic numerical simulator as $f \colon H_{\textnormal{in}} \to H_{\textnormal{out}}$, where
$H_{\textnormal{in}}$ and $H_{\textnormal{out}}$ are two separable Hilbert spaces, and we assume that $f$ is such that $f(X) \in L^2(\Omega, H_{\textnormal{out}})$ for any $X \in L^2(\Omega, H_{\textnormal{in}})$.
Then,
given an input $X \in L^2(\Omega, H_{\textnormal{in}})$, the Bochner integral $\Exp[f(X)] \bydef \int_{\Omega} f(X(\omega)) \mathrm{d} \mathbb{P}(\omega)$, which we shall refer to as the expectation of the output $f(X)$, is well-defined and we have $\Exp[f(X)] \in H_{\textnormal{out}}$~\cite{DaPrato2014_StochasticEquationsInfinite, Hytonen2016_AnalysisBanachSpaces}.
From here on, for the sake of exposition and without loss of generality, we restrict ourselves to the case where $H_{\textnormal{in}} = H_{\textnormal{out}} = H$.

We now turn to the multilevel estimation of $\Exp[f(X)]$, which relies on a hierarchy (sequence) of $L+1$ numerical simulators $(f_\ell)_{\ell=0}^{L}$.
First, in \cref{sec:MLMC_consistent}, we describe the setting where the input and ouput Hilbert space $H$ is the same across the simulators and briefly recall the tools introduced in \cite{bierig_convergence_2015} for the analysis of the MLMC expectation estimator in this setting.
Next, in \cref{sec:MLMC_inconsistent}, we examine the more general case where the input and output spaces $H_\ell$ of the simulators vary across the levels of the hierarchy, and introduce transfer operators between these spaces, eventually leading to a multilevel estimator for which the tools of \cite{bierig_convergence_2015} remain relevant.
We then describe in \cref{sec:discrete_representation} the particular choice of spaces $H_\ell$ that will be considered in the remainder of the paper, arising from the discrete approximation of differential operators.
\Cref{sec:fmlmc} presents the proposed F-MLMC estimator, based on the introduction of filtering operators.
Finally, we outline some practical considerations for the remainder of this paper in \cref{sec:pract-cons}.

\subsection{MLMC with consistent input and output space}\label{sec:MLMC_consistent}

Let $X \in L^2(\Omega, H)$ be the $H$-valued random variable corresponding to a discretized random field.
We consider a multilevel hierarchy of $L+1$ abstract numerical models 
$(f_{\ell} \colon H \to H)_{\ell=0}^{L}$
of increasing fidelity.
We are interested in estimating the expectation of the output of the finest (highest-fidelity) model, 
$\Exp[f_L (X)]$.
The MLMC estimator $\muMLMChilb_L$ of $\Exp[f_L (X)]$, using a sequence $( \{X^{(\ell,i)}\}^{M_{\ell}}_{i=1} )_{\ell=0}^L$ of $L+1$ independent samples of $X$, is defined as
\begin{equation}\label{eq:mlmc_estimator} 
    \muMLMChilb_L
    = 
    \frac{1}{M_0} \sum_{i=1}^{M_0} f_{0} ( X^{(0,i)} ) 
    +
    \sum_{\ell=1}^L \frac{1}{M_{\ell}} \sum_{i=1}^{M_{\ell}} 
    \left[ 
        f_{\ell} ( X^{(\ell,i)} ) - f_{\ell-1} ( X^{(\ell,i)} )
    \right]
    \in L^2(\Omega, H).
\end{equation}
The accuracy of this MLMC estimator may be quantified by its normwise MSE with respect to some $\mu \in H$, defined by (see, e.g., \cite{bierig_convergence_2015})
\begin{equation}\label{eq:def_mse_vector}
    \mse( \muMLMChilb_L, \mu)
    \bydef
    \| \muMLMChilb_L - \mu \|^2_{L^2(\Omega, H)}.
\end{equation}
As shown in~\cite[Theorem~3.1]{bierig_convergence_2015}, the MSE admits the decomposition
\begin{equation}\label{eq:mse_vector_decomposition}
    \mse( \muMLMChilb_L, \mu ) 
    = 
    \mathcal{V}(\muMLMChilb_L)
    +
    \| \mathbb{E}[f_L(X)] - \mu \|^2_{H},
\end{equation}
where
\begin{equation}\label{eq:var_decomposition}
    \forall Y \in L^2(\Omega, H),
    \quad
    \mathcal{V}(Y)
    \bydef \| Y - \Exp[Y] \|^2_{L^2(\Omega, H)}
    = \Exp [\|Y\|^2_{H} ] - \| \Exp[Y] \|^2_{H},
\end{equation}
with the shorthand notation $\Exp[\| \cdot \|^2_H]^{1/2} = \| \cdot \|_{L^2(\Omega,H)}$.
The first term in \cref{eq:mse_vector_decomposition} is referred to as the variance of the multilevel estimator, while the second term corresponds to the squared bias.
Furthermore, \cite[Theorem~3.1]{bierig_convergence_2015} shows that the variance $\mathcal{V}(\muMLMChilb_L)$ can be further decomposed level-wise into
\begin{equation}\label{eq:var_EMLMC_vector}
    \mathcal{V}(\muMLMChilb_L) =
    \frac{1}{M_0} \mathcal{V}(f_0(X)) 
    +
    \sum^L_{\ell=1} \frac{1}{M_{\ell}} \mathcal{V}(f_{\ell}(X) - f_{\ell-1}(X)).
\end{equation}%
We note that the multilevel estimator is unbiased with respect to the expectation of the output at the finest level $\mu_L \bydef \Exp[f_L(X)]$, so that $\mse( \muMLMChilb_L, \mu_L) = \mathcal{V}(\muMLMChilb_L)$.%

\subsection{MLMC with inconsistent input and output spaces}\label{sec:MLMC_inconsistent}

In typical applications, especially when the simulators are related to discretized partial differential equations (PDEs), and that their fidelity is dictated by the quality of the discretization, the input and output spaces $H_\ell$ are typically finite-dimensional and generally not the same across the simulators of the hierarchy.
The multilevel hierarchy of numerical simulators then becomes $(\tilde{f}_{\ell} \colon H_\ell \to H_\ell)_{\ell=0}^L$,
and we are interested in the expectation of the highest-fidelity output, $\Exp[\tilde{f}_L(X_L)] \in H_L$, for some input $X_L \in L^2(\Omega, H_L)$.
As a consequence, the MLMC estimator defined in \cref{eq:mlmc_estimator} cannot be used directly, but transfer operators may be introduced to transfer inputs and outputs between the highest-fidelity space $H_L$ and lower-fidelity spaces $H_\ell$ with $\ell = 0,\ldots, L-1$.
Operators that transfer a discretized field from a level $\ell$ to a higher-fidelity level $\ell' > \ell$ will be referred to as prolongation operators.
We denote by 
$P_{\ell}^{\ell'} \colon H_{\ell} \to H_{\ell'}$
the prolongation operator from level $\ell$ to a higher-fidelity level $\ell' > \ell$.
In addition, restriction operators $R_{\ell'}^{\ell} \colon H_{\ell'} \to H_{\ell}$ are used to perform the ``reverse'' operation of transferring  discretized fields defined on level $\ell'$ to a lower-fidelity level $\ell < \ell'$.
We may now define $f_\ell$ from $\tilde{f}_\ell$ as 
\begin{equation}\label{eq:f_simu_restr_prolong_fine}
    f_{\ell} = P_{\ell}^L \circ \tilde{f}_{\ell} \circ R_L^{\ell} \colon H_L \to H_L,
\end{equation}
for $\ell=0, \ldots, L$.

The numerical simulators in \cref{eq:f_simu_restr_prolong_fine} involve transfer operators between the highest-fidelity level $L$ and lower-fidelity levels.
In practice, grid transfer operators between arbitrary levels can be defined as the composition of transfer operators between successive levels, as
\begin{equation}\label{eq:composed_restr_prolong_general}
\begin{split}
    \forall \ell = 0, \ldots, L-1,
    \quad
    \forall \ell' = \ell+1, \ldots, L,
    \quad
    P_{\ell}^{\ell'} & 
    \bydef 
    P^{\ell'}_{\ell'-1} \circ P^{\ell'-1}_{\ell'-2} 
    \circ \cdots \circ 
    P^{\ell+2}_{\ell+1} \circ P^{\ell+1}_{\ell}, \\
    R^{\ell}_{\ell'} &
    \bydef 
    R^{\ell}_{\ell+1} \circ R^{\ell+1}_{\ell+2} 
    \circ \cdots \circ
    R^{\ell'-2}_{\ell'-1} \circ R^{\ell'-1}_{\ell'},
\end{split}
\end{equation}
and, for $\ell=0,\ldots,L$, 
$R_\ell^\ell = P_\ell^\ell = \id_\ell$, 
where $\id_\ell \colon H_\ell \to H_\ell$
is the identity operator on $H_\ell$.
We further assume that, for any $X \in L^2(\Omega, H_\ell)$,
\begin{itemize}
    \item $\tilde{f}_\ell(X) \in L^2(\Omega, H_\ell)$, for $\ell=0,\ldots,L$;
    \item $P_{\ell}^{\ell+1}(X) \in L^2(\Omega, H_{\ell+1})$, for $\ell=0,\ldots,L-1$;
    \item $R_{\ell}^{\ell-1}(X) \in L^2(\Omega, H_{\ell-1})$, for $\ell=1,\ldots,L$.
\end{itemize}
It follows that, for any $X_L \in L^2(\Omega, H_L)$ and for $\ell=0,\ldots,L$, $f_\ell(X_L) \in L^2(\Omega, H_L)$.
Consequently, the MLMC expression \cref{eq:mlmc_estimator} can now be used with $H = H_L$ and $X=X_L \in L^2(\Omega, H_L)$.
The resulting MLMC formulation, i.e., based on transferred versions $f_\ell$ of the simulators $\tilde{f}_\ell$ as defined in \cref{eq:f_simu_restr_prolong_fine} and composed transfer operators as in \cref{eq:composed_restr_prolong_general}, constitutes the core of the proposed methodology, and will be the focus of the remainder of this paper, in a discrete, finite-dimensional setting described in the next section.

\subsection{Discrete finite-dimensional spaces}\label{sec:discrete_representation}

In this paper, we focus on numerical simulators that arise from the discrete approximation of the strong form of elliptic PDEs, i.e., through the direct discretization of the differential operators involved in such PDEs.
This is the case, for instance, in the finite difference (FD) or finite volume (FV) discretization of a PDE.
This choice is motivated by the considered application, presented in \cref{sec:problem_presentation}, and more importantly by the target ocean numerical model~\cite{Madec2023_NEMOOceanEngine}, which relies on an FD-like discretization.
Consequently, input and ouput discretized fields of such simulators may only be evaluated at a finite number of discrete locations, and are thus represented as finite-dimensional real vectors, whose size depends on the number of discretization points in the spatial domain of interest $\mathcal{D}$.
Formally, we consider a polytopal tessellation $\mathcal{T}$ of a polytopal domain $\mathcal{D}\subset \R^d$, with $d \in \{1,2,3\}$.
We refer to elements of $\mathcal{T}$ as cells, and we denote by $\lvert \mathcal{T} \vert$ the cardinality, or size, of $\mathcal{T}$, i.e., its number of cells.

We now describe the MLMC hierarchy of simulators (and corresponding Hilbert spaces) for this discretization setting.
Let $(\mathcal{T}_{\ell})^L_{\ell=0}$ be a sequence of $L+1$ tessellations of $\mathcal{D}$ of increasing size $n_\ell \bydef \lvert \mathcal{T}_{\ell} \rvert$, for $\ell=0,\ldots, L$.
These tessellations of different sizes define the geometric hierarchy of levels of the MLMC approach.
Specifically, $\mathcal{T}_0$ is the coarsest tessellation corresponding to the lowest-fidelity level, while $\mathcal{T}_L$ is the finest tessellation corresponding to the highest-fidelity level. 
To each level $\ell=0,\ldots,L$, we associate the symmetric positive definite (SPD) Gram matrix 
$\bW_{\ell} 
\in \R^{n_\ell \times n_\ell}$, which encodes the structural information related to the discrete approximation of scalar fields as vectors of $\mathbb{R}^{n_{\ell}}$ on the corresponding cell-centered discretization on $\mathcal{T}_\ell$.
We then define
the weighted inner product $\langle \cdot , \cdot \rangle_{\bW_{\ell}}$ between elements of $\R^{n_\ell}$ as
\begin{equation}\label{eq:def_inner_prod}
    \forall \mathbf{u}, \mathbf{v} \in \mathbb{R}^{n_{\ell}},
    \quad
    \langle \mathbf{u}, \mathbf{v} \rangle_{\bW_{\ell}}
    = \mathbf{u}^{\transp} \bW_{\ell} \mathbf{v} 
    = \langle \bV_{\ell}^{\transp}\mathbf{u}, \bV_{\ell}^{\transp}\mathbf{v} \rangle_{\bI_{n_\ell}},
\end{equation}
where 
$\bV_{\ell} \in \R^{n_\ell \times n_\ell}$ is an invertible matrix such that $\bW_{\ell} = \bV_{\ell} \bV_{\ell}^{\transp}$, which exists since $\bW_{\ell}$ is SPD, and where
$\langle \cdot , \cdot \rangle_{\bI_{n_\ell}}$ denotes the canonical (Euclidean) dot product between vectors of $\R^{n_\ell}$.
The inner product space $H_\ell \bydef (\R^{n_\ell}, \langle \cdot , \cdot \rangle_{\bW_{\ell}})$ is a separable Hilbert space, so that the multilevel hierarchy of simulators $(\tilde{f}_\ell \colon H_\ell \to H_\ell)_{\ell=0}^L$ fits in the framework described in \cref{sec:MLMC_inconsistent}.
For clarity, we shall from here on simply write $\R^{n_\ell}$ for $(\R^{n_\ell}, \langle \cdot , \cdot \rangle_{\bW_{\ell}})$, and, unless explicitly stated otherwise, use $\langle \cdot , \cdot \rangle_{\bW_{\ell}}$ as the default inner product between vectors of $\R^{n_\ell}$.
Furthermore, we shall denote by $\| \cdot \|_{\bW_{\ell}}$ the norm induced by $\langle \cdot , \cdot \rangle_{\bW_{\ell}}$.
In this discrete setting, the MLMC estimator defined in \cref{eq:mlmc_estimator} becomes
\begin{equation}\label{eq:mlmc_estimator_discrete} 
    \muMLMC_L
    = 
    \frac{1}{M_0} \sum_{i=1}^{M_0} f_{0} ( \bX_L^{(0,i)} ) 
    +
    \sum_{\ell=1}^L 
    \frac{1}{M_{\ell}} 
    \sum_{i=1}^{M_{\ell}} 
    \left[ f_{\ell} ( \bX_L^{(\ell,i)} ) - f_{\ell-1} ( \bX_L^{(\ell,i)} ) \right],
\end{equation}
with the grid transferred version of the simulators defined as in \cref{eq:f_simu_restr_prolong_fine}.
It should be remarked that the $L+1$ independent samples $(\{\bX_L^{(\ell,i)}\}_{i=1}^{M_\ell})_{\ell=0}^L$ of $\bX_L$ are all generated on the finest grid.

In the functional framework that is typically employed for simulators whose inputs and/or outputs are discretized fields
  \cite{cliffe_multilevel_2011,
  Barth2011_MultilevelMonteCarlo,
  Mishra2012_SparseTensorMultilevel,
  Mishra2012_MultilevelMonteCarlo,
  Abdulle2013_MultilevelMonteCarlo,
  Gittelson2013_MultilevelMonteCarlo,
  Teckentrup2013_FurtherAnalysisMultilevel,
  Charrier2013_FiniteElementError,
  Istratuca2025_SmoothedCirculantEmbedding},
grid transfer operators are implicitly employed, but usually not explicitly defined or even mentioned.
This is especially the case for the prolongation operator, since discretized fields are defined everywhere in the spatial domain through their functional representation. 
Nevertheless, the practical computation of the multilevel correction terms in \cref{eq:mlmc_estimator_discrete} implicitly requires the introduction of prolongation operators.
By focusing on the discrete framework, we explicitly exhibit the grid transfer operators and study their effect on the quality of the multilevel estimator.
We shall see that the operators that seem the most natural in the functional framework, such as the canonical injection for the prolongation operator in settings with nested approximation spaces, actually yield multilevel estimators with poor properties.
This is particularly true for cell-centered discretizations, which are the focus of this paper (see \cref{sec:pract-cons}), and which are also known in the multigrid community for causing similar difficulties~\cite{Wesseling1988_CellcenteredMultigridInterface, Hemker1990_OrderProlongationsRestrictions, khalil_vertex-centered_1992, Mohr2004_CellcentredMultigridRevisited}.

\subsection{The filtered MLMC method}\label{sec:fmlmc}

In this paper, we are particularly concerned with the impact of the grid transfer operators on the representation of the different scales (or frequencies) of the multilevel correction fields.
While it is clear that, owing to the telescoping correction mechanism, the bias of the multilevel estimator \eqref{eq:mlmc_estimator_discrete} is unaffected, the impact on its variance, and hence on the MSE, remains to be investigated.
As will be illustrated in \cref{sec:spectral_analysis,sec:experiments}, the introduction of grid transfer operators induces a deterioration of the multilevel estimator's variance in the high frequencies. 
Inspired by multigrid methods~\cite{trottenberg_multigrid_2000, briggs_multigrid_2000, wienands_practical_2004}, we propose an improvement of the MLMC estimator for discretized random fields by adding pre- and post-filtering (or, in multigrid terminology, smoothing). 
The objective is to filter out the smaller scales, which cannot be represented on the coarse grids, before using a restriction operator and after using a prolongation operator.
The idea of filtering highly oscillatory input fields in an MLMC framework has been proposed in \cite{Teckentrup2013_FurtherAnalysisMultilevel, Istratuca2025_SmoothedCirculantEmbedding}, for simulators with scalar outputs.
This filtering is achieved through spectral truncation in the Karhunen-Lo\`eve~\cite{Kac1947_ExplicitRepresentationStationary,LeMaitre2010_SpectralMethodsUncertainty,Betz2014_NumericalMethodsDiscretization} or circulant embedding~\cite{Graham2011_QuasiMonteCarloMethods,Graham2018_AnalysisCirculantEmbedding} framework.

In this paper, we propose a general filtering formalism for the abstract, grid-transferred MLMC framework \eqref{eq:mlmc_estimator_discrete}.
Because the outputs of the simulators considered here are discretized fields, we also apply the filtering process to the outputs of coarser levels after they have been prolongated to the fine grid.
In the proposed formalism, filtering the small-scale components out of a signal defined on level $\ell$ is achieved using a low-pass pre- or post-filtering operator $S^{\pre/\post}_{\ell} \colon \mathbb{R}^{n_{\ell}} \to \mathbb{R}^{n_{\ell}}$.
These filters can then be composed with the grid transfer operators between successive levels to obtain pre-filtered restriction operators and post-filtered prolongation operators,
\begin{equation}
  \label{eq:filtered_restr_prolong}
  \forall \ell=1,\ldots,L,
  \qquad
  \bar{R}_{\ell}^{\ell-1} \bydef R_{\ell}^{\ell-1} \circ S^{\pre}_{\ell},
  \mbox{ and }
  \bar{P}_{\ell-1}^{\ell} \bydef S^{\post}_{\ell} \circ P_{\ell-1}^{\ell}.
\end{equation}
Similar to \cref{eq:composed_restr_prolong_general}, filtered grid transfer operators between arbitrary levels can be defined as the composition of transfer operators between successive levels, as
\begin{equation}\label{eq:composed_filt_restr_prolong_general}
\begin{split}
    \forall \ell = 0, \ldots, L-1,
    \quad
    \forall \ell' = \ell+1, \ldots, L,
    \quad
    \bar{P}_{\ell}^{\ell'} & 
    \bydef 
    \bar{P}^{\ell'}_{\ell'-1} \circ \bar{P}^{\ell'-1}_{\ell'-2} 
    \circ \cdots \circ 
    \bar{P}^{\ell+2}_{\ell+1} \circ \bar{P}^{\ell+1}_{\ell}, \\
    \bar{R}^{\ell}_{\ell'} &
    \bydef 
    \bar{R}^{\ell}_{\ell+1} \circ \bar{R}^{\ell+1}_{\ell+2} 
    \circ \cdots \circ
    \bar{R}^{\ell'-2}_{\ell'-1} \circ \bar{R}^{\ell'-1}_{\ell'},
\end{split}
\end{equation}
and, for $\ell=0,\ldots,L$, 
$\bar{R}_\ell^\ell = \bar{P}_\ell^\ell = \id_\ell$.
A new (filtered) simulator $\bar{f}_{\ell}$ is then defined on each level $\ell$ as
\begin{equation}
  \label{eq:f_simu_filt_trans}
  \bar{f}_{\ell} = \bar{P}_{\ell}^L \circ \tilde{f}_{\ell} \circ \bar{R}_L^{\ell},
\end{equation}
so that the filtered MLMC (F-MLMC) estimator is obtained by replacing $f_\ell$ with $\bar{f}_\ell$ in  \cref{eq:mlmc_estimator_discrete}:
\begin{equation}\label{eq:fmlmc_estimator} 
    \muFMLMC_L
    = 
    \frac{1}{M_0} \sum_{i=1}^{M_0} \bar{f}_{0} ( \bX_L^{(0,i)} ) 
    +
    \sum_{\ell=1}^L 
    \frac{1}{M_{\ell}} 
    \sum_{i=1}^{M_{\ell}} 
    \left[ \bar{f}_{\ell} ( \bX_L^{(\ell,i)} ) - \bar{f}_{\ell-1} ( \bX_L^{(\ell,i)} ) \right].
\end{equation}

\subsection{Practical considerations}\label{sec:pract-cons}

While the proposed methodology solely relies on \cref{eq:f_simu_restr_prolong_fine,eq:mlmc_estimator_discrete} for the unfiltered MLMC estimator, and on \cref{eq:filtered_restr_prolong,eq:composed_filt_restr_prolong_general,eq:f_simu_filt_trans,eq:fmlmc_estimator} for the F-MLMC estimator, we describe in this section further practical considerations that will be assumed for the remainder of this paper.

\subsubsection{Discrete setting and MLMC hierarchy}

In what follows, we consider cell-centered discretizations, meaning that the discretization points are located at the center of the cells of a tessellation $\mathcal{T}$.
It should be noted that functional representations can be deduced \emph{a posteriori} from such discretized fields.
In particular, if $\bu \in \R^n$ is a cell-centered discretized field on a tessellation $\mathcal{T} \bydef \{T_i\}_{i=1}^n$, a natural functional representation is a piecewise constant field $u \in \mathbb{P}^0$, where $\mathbb{P}^0 \bydef \Span(\{\indic_{T_i}\}_{i=1}^n)$ and $\indic_T \colon \mathcal{D} \to \{0,1\}$ denotes the indicator function associated with the subset $T \subseteq \mathcal{D}$. 
The piecewise constant representation can then be defined through $\bu \in \R^n$ as $u=\bu^{\transp} \bm{\phi}$, where $\bm{\phi} \bydef [\indic_{T_1}, \ldots, \indic_{T_n}]^{\transp}$.

In the remainder of this paper, we prescribe a fixed finest level $L$.
This is consistent with the fact that, in complex operational settings, the high-fidelity level is typically dictated by the application.
We further assume that the tessellations of the grid hierarchy $(\mathcal{T}_\ell)_{\ell=0}^L$ are nested, i.e., for $\ell=1,\ldots,L$, for any $T \in \mathcal{T}_\ell$, there is a unique $T' \in \mathcal{T}_{\ell-1}$ such that $T \subset T'$.
We note, however, that the sets of discretization points, i.e., the sets of cell centers, are not nested.
It should be stressed that the nestedness of the tessellations is not strictly required, but facilitates the design of transfer operators.
Non-nested grid hierarchies would call for the design of more complex, and possibly more expensive, transfer operators.

\subsubsection{Optimal sample allocation}
The sample sizes $M_\ell$ are optimally allocated across levels such that, for a given computational budget $\mathcal{C}$, the variance of the  multilevel estimator is minimal~\cite{mycek_multilevel_2019},
\begin{equation}\label{eq:optimal_sample_alloc}
    M_{\ell} 
    =
    \left\lfloor 
        \frac{\mathcal{C}}{\mathcal{S}_L} 
        \sqrt{\frac{\mathcal{V}_{\ell}}{\mathcal{C}_{\ell} + \mathcal{C}_{\ell-1}}} 
    \right\rfloor^+,
    \qquad
    \mathcal{V}_{\ell} 
    \bydef 
    \mathcal{V}( \bar{f}_{\ell}(\bX_L) - \bar{f}_{\ell-1}(\bX_L) ),
    \qquad
    \mathcal{S}_L 
    \bydef 
    \sum_{\ell=0}^L \sqrt{\mathcal{V}_{\ell} (\mathcal{C}_{\ell} + \mathcal{C}_{\ell-1})},
\end{equation}
where $\mathcal{C}_{\ell}$ denotes the mean computational cost of evaluating $\bar{f}_\ell(\bX_L)$ 
and $\lfloor \cdot \rfloor^+ \bydef \max(1, \lfloor \cdot \rfloor)$, where $\lfloor \cdot \rfloor$ denotes the floor function.
By convention, the quantities indexed by $\ell=-1$ vanish, so that $\mathcal{C}_{-1} = 0$ and $\mathcal{V}_{0} = \mathcal{V}(\bar{f}_0(\bX_L))$.
Furthermore, for conciseness, we have used the same notation $\bar{f}_\ell$ to refer either to the filtered version of $f_\ell$ (i.e., $\bar{f}_\ell$ as defined in \cref{sec:fmlmc}) when considering F-MLMC, or to its unfiltered counterpart (i.e., simply $f_\ell$) when dealing with unfiltered MLMC.
As a matter of fact, the latter is a special case of the former, with identity operators as filters.
For a fixed finest level $L$, it is easy to see from the optimal sample allocation \eqref{eq:optimal_sample_alloc} (ignoring the rounding of $M_\ell$) that the variance of the resulting multilevel estimator is $\mathcal{S}_L^2 / \mathcal{C}$, while the variance of the single-level, high-fidelity MC estimator at equivalent cost is $\mathcal{C}_L\mathcal{V}(f_L(\bX_L)) / \mathcal{C}$.
This shows that, while the variance of the multilevel estimator is reduced (regardless of the budget) provided that $\mathcal{S}_L^2 < \mathcal{C}_L\mathcal{V}(f_L(\bX_L))$, the convergence rate remains unchanged compared to a standard single-level MC estimator, which is typical of multifidelity methods where $L$ is fixed (see, e.g., \cite{ElAmri2024_MultilevelSurrogatebasedControl}).
It should be noted that, because of the introduction of grid transfer operations, the level-wise cost $\mathcal{C}_\ell$ can be decomposed into two contributions, specifically, the cost $\tilde{\mathcal{C}}_\ell$ of evaluating the simulator $\tilde{f}_\ell$
on the one hand, and the cost $\mathcal{C}_\ell^{\transf}$ of applying the (possibly filtered) grid transfer operators on the other hand.

\subsubsection{Unfiltered grid transfer operators}

In practical applications, the transfer operators introduced in \cref{sec:MLMC_inconsistent} are typically chosen to be linear.
A linear prolongation operator $P_{\ell}^{\ell'}$ may be identified with the matrix $\bP_{\ell}^{\ell'} \in \mathbb{R}^{n_{\ell'} \times n_{\ell}}$ such that $P_{\ell}^{\ell'}(\mathbf{x}_{\ell}) = \bP_{\ell}^{\ell'} \mathbf{x}_{\ell}$ for any $\mathbf{x}_{\ell} \in \mathbb{R}^{n_{\ell}}$,
and, likewise, a linear restriction operator $R_{\ell'}^{\ell}$ may be identified with the appropriate matrix $\bR_{\ell'}^{\ell} \in \mathbb{R}^{n_{\ell} \times n_{\ell'}}$.
In the remainder of this paper, we shall only consider linear transfer operators, and hence employ the associated matrix notation.
Then, \cref{eq:composed_restr_prolong_general} implies
\begin{equation}\label{eq:composed_restr_prolong}
    \forall \ell = 0, \ldots, L-1,
    \quad
    \bP_{\ell}^L \bydef \bP^L_{L-1} \cdots \bP^{\ell+2}_{\ell+1} \bP^{\ell+1}_{\ell}
    \mbox{ and }
    \bR^{\ell}_L \bydef \bR^{\ell}_{\ell+1} \cdots \bR^{L-2}_{L-1} \bR^{L-1}_L,
  \end{equation}
and $\bR_L^L = \bP_L^L = \bI_{n_L}$.
Consequently, the MLMC estimator \eqref{eq:mlmc_estimator_discrete} can be written as
\begin{align}\label{eq:mlmc_estimator_transfer}
    \muMLMC_L
    & = 
    \frac{1}{M_0} \sum_{i=1}^{M_0} 
    \bP_0^L \tilde{f}_{0} ( \bR_L^0 \bX_L^{(0,i)} ) 
    +
    \sum_{\ell=1}^L \frac{1}{M_{\ell}} \sum_{i=1}^{M_{\ell}} 
    \left[ 
        \bP_\ell^L \tilde{f}_{\ell} ( \bR_L^{\ell} \bX_L^{(\ell,i)} )
        - 
        \bP_{\ell-1}^L \tilde{f}_{\ell-1} ( \bR_L^{\ell-1} \bX_L^{(\ell,i)} )
    \right],\\ \label{eq:mlmc_estimator_transfer_v2}
    & = 
    \frac{1}{M_0} \bP_0^L 
    \sum_{i=1}^{M_0} \tilde{f}_{0} ( \bR_L^0 \bX_L^{(0,i)} ) 
    +
    \sum_{\ell=1}^L \frac{1}{M_{\ell}} \bP_{\ell}^L
    \sum_{i=1}^{M_{\ell}} \left[ 
        \tilde{f}_{\ell} ( \bR_L^{\ell} \bX_L^{(\ell,i)} )
        - 
        \bP_{\ell-1}^{\ell} \tilde{f}_{\ell-1} ( \bR_L^{\ell-1} \bX_L^{(\ell,i)} )
    \right],
\end{align}
where the second identity is obtained by exploiting the linearity of the prolongation operators and the composition definition in \cref{eq:composed_restr_prolong}.
The latter is to be preferred in practical implementations, as it involves fewer applications of the prolongation operators.
Finally, in the remainder of this paper, we shall consider prolongation operators satisfying
\begin{equation}\label{eq:prop_restr_prolong}
    \forall \ell = 0, \ldots, L-1,
    \quad
    (\bP_{\ell}^{\ell+1})^{\transp} \bW_{\ell+1}  \bP_{\ell}^{\ell+1}  
    = 
    \bW_{\ell}.
\end{equation}
A major consequence of \cref{eq:prop_restr_prolong} is that
\begin{equation}
    \forall \ell=0, \ldots, L-1,
    \;
    \forall \bx_{\ell} \in \R^{n_{\ell}},
    \quad
    \|\bP_{\ell}^{\ell+1} \bx_{\ell}\|_{\bW_{\ell+1}}
    =
    \bx_{\ell}^{\transp} (\bP_{\ell}^{\ell+1})^{\transp} \bW_{\ell+1} \bP_{\ell}^{\ell+1} \bx_{\ell} 
    =
    \bx_{\ell}^{\transp} \bW_{\ell} \bx_{\ell}
    = 
    \|\bx_{\ell}\|_{\bW_{\ell}},
\end{equation}
indicating that the prolongation operator is norm-preserving.
We note that, by \cref{eq:composed_restr_prolong}, similar properties then hold for transfer operators between arbitrary levels.

\subsubsection{Filtering operators and filtered grid transfer operators}

Without loss of generality, we shall use the same operator $S_{\ell} = S^{\pre}_\ell=S^{\post}_\ell$ for both pre- and post-filtering on each level.
Moreover, we shall use linear filtering operators $S_\ell$, which can thus be identified with matrices $\bS_\ell$, so that the filtered grid transfer operators remain linear.
The resulting filtered grid transfer operators $\bar{P}_{\ell}^L$ and $\bar{R}_L^{\ell}$ are then defined through their corresponding matrices $\bar{\bP}_{\ell}^L$ and $\bar{\bR}^{\ell}_L$ by
\begin{equation}\label{eq:composed_restr_prolong_fmlmc}
    \forall \ell = 0, \ldots, L-1,
    \quad
    \bar{\bP}_{\ell}^L \bydef \bar{\bP}^L_{L-1} \cdots \bar{\bP}^{\ell+2}_{\ell+1} \bar{\bP}^{\ell+1}_{\ell}
    \mbox{ and }
    \bar{\bR}^{\ell}_L \bydef \bar{\bR}^{\ell}_{\ell+1} \cdots \bar{\bR}^{L-2}_{L-1} \bar{\bR}^{L-1}_L,
\end{equation}
where
\begin{equation}
  \forall \ell = 1, \ldots, L,
  \quad
  \bar{\bR}_{\ell}^{\ell-1} 
  \bydef \bR_{\ell}^{\ell-1} \bS_{\ell},
  \quad \text{and} \quad
  \bar{\bP}_{\ell-1}^{\ell}
  \bydef \bS_{\ell} \bP_{\ell-1}^{\ell},
\end{equation}
and $\bar{\bR}_L^L = \bar{\bP}_L^L = \bI_{n_L}$.
The F-MLMC estimator then reads
\begin{align}\label{eq:fmlmc_estimator_transfer}
    \muFMLMC_L
    & = 
    \frac{1}{M_0} \sum_{i=1}^{M_0} 
    \bar{\bP}_0^L \tilde{f}_{0} ( \bar{\bR}_L^0 \bX_L^{(0,i)} ) 
    +
    \sum_{\ell=1}^L \frac{1}{M_{\ell}} \sum_{i=1}^{M_{\ell}} 
    \left[ 
        \bar{\bP}_\ell^L \tilde{f}_{\ell} ( \bar{\bR}_L^{\ell} \bX_L^{(\ell,i)} )
        - 
        \bar{\bP}_{\ell-1}^L \tilde{f}_{\ell-1} ( \bar{\bR}_L^{\ell-1} \bX_L^{(\ell,i)} )
    \right],\\ \label{eq:fmlmc_estimator_transfer_v2}
    & = 
    \frac{1}{M_0} \bar{\bP}_0^L 
    \sum_{i=1}^{M_0} \tilde{f}_{0} ( \bar{\bR}_L^0 \bX_L^{(0,i)} ) 
    +
    \sum_{\ell=1}^L \frac{1}{M_{\ell}} \bar{\bP}_{\ell}^L
    \sum_{i=1}^{M_{\ell}} \left[ 
        \tilde{f}_{\ell} ( \bar{\bR}_L^{\ell} \bX_L^{(\ell,i)} )
        - 
        \bar{\bP}_{\ell-1}^{\ell} \tilde{f}_{\ell-1} ( \bar{\bR}_L^{\ell-1} \bX_L^{(\ell,i)} )
    \right].
\end{align}
Again, we note that the unfiltered MLMC estimator corresponds to a special case of the filtered MLMC estimator with identity filtering operators.
However, we note that, in general, the filtered transfer operators defined in \cref{sec:fmlmc} no longer satisfy the norm-preserving property of \cref{eq:prop_restr_prolong}.
It should be stressed that the proposed methodology remains valid with non-linear (filtered or unfiltered) transfer operators; in particular, the telescoping sum property still guarantees the unbiasedness of the multilevel estimators in \cref{eq:mlmc_estimator_discrete,eq:fmlmc_estimator}.

\section{Spectral analysis}\label{sec:spectral_analysis}

In this section, we conduct a spectral analysis of the MLMC estimator \eqref{eq:mlmc_estimator_transfer} and of the F-MLMC estimator \eqref{eq:fmlmc_estimator_transfer} to study more closely the effects of grid transfer operators on their variance at different spatial scales.
The analysis is performed in a one-dimensional spatial domain, with cell-centered discretized input and output fields on uniform, nested grids.
The setting is further simplified to the case where the numerical simulators are linear, and that their associated matrices are symmetric and circulant.

\subsection{Linear simulators}\label{sec:var_exp_mlmc}

We start by considering linear simulators of the form $\tilde{f}_\ell \colon \bx_\ell \mapsto \tilde{\bF}_\ell \bx_\ell$, for $\ell=0,\ldots,L$, where $\tilde{\bF}_{\ell} \in \mathbb{R}^{n_{\ell} \times n_{\ell}}$, and, consistently with \cref{eq:f_simu_restr_prolong_fine,eq:mlmc_estimator_transfer}, we define $\bF_\ell \bydef \bP^L_{\ell} \tilde{\bF}_{\ell} \bR^{\ell}_L \in \R^{n_L \times n_L}$,
with $\bP^L_L = \bR^L_L = \bI_{n_L}$.
For now, we make no further assumption on $\tilde{\bF}_{\ell}$.
Then, by linearity of the expectation operator,
the variance $\mathcal{V}( \muMLMC_L)$ of the multilevel estimator \eqref{eq:mlmc_estimator_transfer}, given by \cref{eq:var_EMLMC_vector}, becomes
\begin{equation}\label{eq:variance_mlmc_linear_models}
    \mathcal{V}( \muMLMC_L)
    = 
    \frac{1}{M_0} \Exp[ \| \bF_0 \dot{\bX}_L \|^2_{\bW_L} ] 
    + 
    \sum^L_{\ell=1} \frac{1}{M_{\ell}} 
    \Exp[ \| 
        ( \bF_{\ell} - \bF_{\ell-1} ) \dot{\bX}_L 
    \|^2_{\bW_L} ],
\end{equation}
where $\dot{\bX}_L \bydef \bX_L - \Exp[\bX_L]$.
Now, for any matrix $\bF\in\R^{n_L \times n_L}$,
\begin{equation}
\begin{split}
    \Exp[ \| \bF\dot{\bX}_L \|^2_{\bW_L}]
    & = \Exp[ (\bF\dot{\bX}_L)^{\transp} \bW_L (\bF\dot{\bX}_L) ]
    = \tr \Exp[ \dot{\bX}_L^{\transp} \bF^{\transp} \bW_L \bF\dot{\bX}_L ]
    = \Exp[ \tr( \dot{\bX}_L^{\transp} \bF^{\transp} \bW_L \bF\dot{\bX}_L ) ] \\
    & = \Exp[ \tr( \bF^{\transp} \bW_L \bF \dot{\bX}_L \dot{\bX}_L^{\transp} ) ]
    = \tr( \bF^{\transp} \bW_L \bF \Exp[\dot{\bX}_L \dot{\bX}_L^{\transp}] )
    = \tr( \bF^{\transp} \bW_L \bF \bG) \\
    & = \tr( (\bF \bG^{1/2})^{\transp} \bW_L (\bF \bG^{1/2}) )
    = \| \bF \bG^{1/2} \|^2_{F,\bW_L}, 
\end{split}
\end{equation}
where $\tr(\cdot)$ denotes the matrix trace operator, $\| \cdot \|_{F,\bW} \colon \bF \mapsto \operatorname{tr}(\bF^{\transp} \bW \bF)^{1/2}$ denotes the $\bW$-weighted Frobenius norm for any SPD weighting matrix $\bW$, and $\bG \bydef \Exp[ \dot{\bX}_L \dot{\bX}_L^{\transp}] = \bG^{1/2}(\bG^{1/2})^{\transp}$ denotes the covariance matrix of $\bX_L$.
Thus, \cref{eq:variance_mlmc_linear_models} reduces to
\begin{equation}\label{eq:variance_mlmc_linear_models_frob}
    \mathcal{V}( \muMLMC_L)
    = 
    \frac{1}{M_0} \| \bF_0 \bG^{1/2} \|^2_{F,\bW_L}  
    + 
    \sum^L_{\ell=1} \frac{1}{M_{\ell}} 
    \| 
        ( \bF_{\ell} - \bF_{\ell-1} ) \bG^{1/2} 
    \|^2_{F,\bW_L},
\end{equation}
which emphasizes that the variance reduction is closely related to the similarity of successive fidelity models.
Furthermore, for any orthogonal matrix $\bQ\in\R^{n_L \times n_L}$ with respect to the weighted inner product $\langle \cdot, \cdot\rangle_{\bW_{L}}$, i.e., such that $\bQ \bW_{L} \bQ^{\transp}= \bQ^{\transp} \bW_{L} \bQ =\bI_{n_L}$, we have, for any $\bx\in\R^{n_L}$,
\begin{equation}\label{eq:prop_norms_orthog_proj}
    \| \bQ^{\transp} \bV^{\transp} \bx \|_{\bW_L} = \| \bQ \bV^{\transp} \bx \|_{\bW_L} = \| \bx \|_{\bW_L},
    \quad \mbox{and }
    \| \bQ^{\transp} \bx \|_{\bW_L} = \| \bQ \bx \|_{\bW_L} = \| \bx \|_{\bI_{n_L}},
\end{equation}
and, similarly, for any $\bF\in\R^{n_L \times n_L}$,
\begin{equation}\label{eq:prop_norms_orthog_proj_frob}
    \| \bQ^{\transp} \bV^{\transp} \bF \|_{F,\bW_L} = \| \bQ \bV^{\transp} \bF \|_{F,\bW_L} = \| \bF \|_{F,\bW_L},
    \quad \mbox{and }
    \| \bQ^{\transp} \bF \|_{F, \bW_L} = \| \bQ \bF \|_{F, \bW_L} = \| \bF \|_{F, \bI_{n_L}}.
\end{equation}

\subsection{One-dimensional spatial setting}\label{sec:spectral_analysis_1d_setting}

The analysis is conducted in a one-dimensional (1D) setting with periodic, cell-centered discretized input and output random fields defined on a unit spatial domain $\mathcal{D} \bydef [0,1) \subset \R$. 
Specifically, we consider a hierarchy of uniform grids $(\mathcal{T}_\ell)_{\ell=0}^L$, where $\mathcal{T}_{\ell} \bydef \{T_{\ell, i}\}_{i=1}^{n_\ell}$ and $T_{\ell,i} \bydef [(i-1)/n_\ell, i/n_\ell)$ on level $\ell$,
along with the associated Gram matrices $\bW_{\ell} = n_{\ell}^{-1} \bI_{n_{\ell}}$.
It follows immediately that $\bV_{\ell} = n_{\ell}^{-1/2} \bI_{n_{\ell}} = \bW_{\ell}^{1/2}$. 
For $\ell = 0, \dots, L-1$, we set a constant refinement factor $n_{\ell+1}/n_{\ell} = 2$, so that $n_{\ell} = 2^{\ell-L} n_L$.
The inter-level transfer operators are defined as
\begin{equation}\label{eq:def_prolongation_restriction}
    \forall \ell = 0, \ldots, L-1,
    \qquad
    \bP_{\ell}^{\ell+1} \bydef
    \begin{bmatrix}
        1 & & \\
        1 & & \\
        & 1 & \\
        & 1 & \\
        & & \ddots
    \end{bmatrix} \in \mathbb{R}^{n_{\ell+1} \times n_{\ell}},
    \mbox{ and }
    \bR_{\ell+1}^{\ell}= (\bP_{\ell}^{\ell+1})^{\transp}
    \in \mathbb{R}^{n_{\ell} \times n_{\ell+1}}.
\end{equation}
Note that the prolongation operator in \cref{eq:def_prolongation_restriction} induces the canonical injection for the functional, piecewise constant representation mentioned in \cref{sec:discrete_representation}.
Indeed, letting $\mathbb{P}^0_{\ell} \bydef \Span(\{\indic_{T_{\ell,i}}\}_{i=1}^{n_\ell})$ denote the space of piecewise constant functions on $\mathcal{T}_\ell$ and letting $\imath_{\ell}^{\ell+1} \colon \mathbb{P}^0_{\ell}  \hookrightarrow \mathbb{P}^0_{\ell+1}$ denote the canonical injection from $\mathbb{P}^0_{\ell}$ to $\mathbb{P}^0_{\ell+1}$, then, for any $u_\ell = \bx^{\transp} \bm{\phi}_{\ell} \in \mathbb{P}^0_{\ell}$ with $\bm{\phi}_{\ell} \bydef [\indic_{T_{\ell, 1}}, \ldots, \indic_{T_{\ell, n_\ell}}]^{\transp}$, $\imath_{\ell}^{\ell+1}(u_\ell) = \bx_{\ell+1}^{\transp} \bm{\phi}_{\ell+1}$ with $\bx_{\ell+1} = \bP_{\ell}^{\ell+1}\bx_\ell$.
This prolongation operator satisfies \cref{eq:prop_restr_prolong}, which, in this setting, amounts to 
$(\bP_{\ell}^{\ell+1})^{\transp} \bP_{\ell}^{\ell+1} 
= 2 \bI_{n_{\ell}}$,
for $\ell=0,\ldots,L-1$.
Then, grid transfer operators between arbitrary levels are defined by composition of successive inter-level operators as in \cref{eq:composed_restr_prolong}.

\subsection{The Hartley basis}\label{sec:hartley-basis}

The spectral analysis conducted below relies on the discrete Hartley basis~\cite{Hartley1942_MoreSymmetricalFourier, Bracewell1983_DiscreteHartleyTransform, bini_matrix_1993}, adapted here to cell-centered discretized fields.
The Hartley basis is a Fourier-like basis, commonly used in circulant embedding techniques for generating stationary Gaussian random fields with prescribed covariance structure \cite{Graham2011_QuasiMonteCarloMethods,Graham2018_AnalysisCirculantEmbedding,Istratuca2025_SmoothedCirculantEmbedding}.
Its main advantage is that it consists of purely real basis vectors, as opposed to the Fourier basis, whose basis vectors are complex, thus easing interpretation and visualization.
On level $\ell$, the $n_\ell$ cell-centered Hartley basis vectors  
$\{\bh^{\ell}_k\}_{k=0}^{n_\ell-1}$
correspond to the columns of the Hartley matrix $\bH_\ell$ with entries
\begin{equation}\label{eq:Hartley_basis}
    (\bH_{\ell})_{j,k} 
    \bydef 
    \cos \mleft( \frac{2(j+\frac{1}{2})k\pi}{n_{\ell}} \mright)
    + 
    \sin \mleft( \frac{2(j+\frac{1}{2})k\pi}{n_{\ell}} \mright), 
    \quad \forall j,k = 0,\dots,n_{\ell} -1.
\end{equation}%
The matrices $\bH_{\ell}$ are orthogonal with respect to their associated inner product $\langle \cdot, \cdot\rangle_{\bW_{\ell}}$, i.e., $\bH_{\ell}^{\transp} \bH_{\ell} = \bH_{\ell} \bH_{\ell}^{\transp} = n_{\ell} \bI_{n_{\ell}}$, for $\ell=0,\dots,L$ (see \cref{app:ortho_hartley}).
\Cref{fig:hartley_basis} depicts the basis vectors $\bh^{1}_k$ and $\bh^{0}_k$ of fine- and coarse-grid Hartley basis vectors, discretized on grids with $n_{1}=16$ and $n_{0}=8$ cells, respectively.
These plots highlight that, because of aliasing, the basis vectors exhibit a discrete frequency that is different from their continuous counterpart.
Specifically, for $\ell\in \{0,1\}$, the vectors indexed by $k$ close to 0 or $n_\ell-1$ are discrete signals with low frequency, while their frequency increases as $k$ tends to $n_\ell/2$.

\begin{figure}[!ht]
    \centering
    \includegraphics[width=\linewidth]{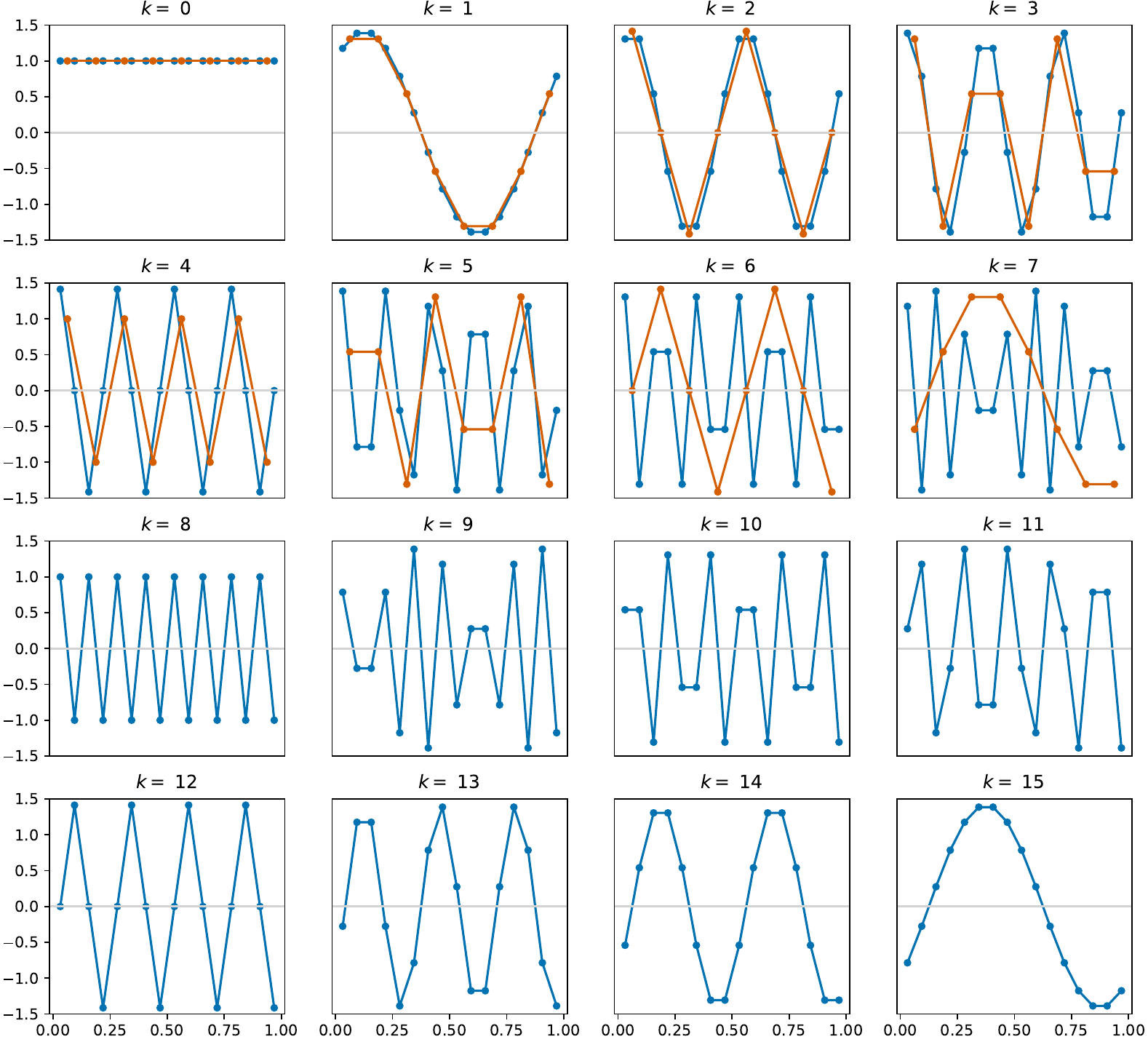}
    \caption{The vectors $\bh^{0}_k$ (orange) of a coarse-grid Hartley basis on level $\ell=0$, with $n_{0}=8$, and the vectors $\bh^{1}_k$ (blue) of a fine-grid Hartley basis on level $\ell=1$, with $n_{1}=16$.
    For $\ell\in \{0,1\}$ and for each vector $k=0,\ldots,n_{\ell}-1$, the entries of the vector $[\bh^{\ell}_k]_j$ are plotted against $x_j \bydef (2j+1)/(2 n_{\ell}) \in (0,1)$, for $j=0,\ldots,n_{\ell}-1$.}
    \label{fig:hartley_basis}
\end{figure}

We now examine the effects of the inter-level grid transfer operators between two successive levels on Hartley basis vectors.
Such effects have been studied extensively for multigrid methods~\cite{briggs_multigrid_2000, wienands_practical_2004, trottenberg_multigrid_2000} using different bases.
We succinctly present here results for the specific, cell-centered Hartley basis \cref{eq:Hartley_basis}.
For the sake of using lighter notations, we denote by $(0, 1)$ the considered pair of successive levels, but the analysis and its conclusions naturally hold for any pair $(\ell, \ell+1)$.
With the prolongation operator $\bP \bydef \bP_0^1$ defined in \cref{eq:def_prolongation_restriction}, we have (see \cref{app:proof_prolong_hartley})
\begin{equation}\label{eq:prolong_hartley}
    \bP \bh^0_k
    =
    c_k \bh^1_k - c_{n_0+k} \bh^1_{n_0+k},
    \quad \forall k=0,\dots,n_0-1,
\end{equation}
where the coefficients $c_{k} \bydef \cos(k\pi/n_1) = \cos(k\pi/(2n_0))$ are strictly decreasing with $k=0,\dots,2n_0-1$, i.e., 
\begin{equation}
   1 = c_0 > c_1 > \dots > (c_{n_0} = 0) > \dots > c_{2n_0-1} > -1.
\end{equation}
Prolongating a coarse-grid basis vector $\bh^0_k$, with $k=0,\dots,n_0-1$, produces a (fine-grid) vector consisting of a linear combination of two fine-grid basis vectors $\bh^1_k$ and $\bh^1_{n_0+k}$.
\Cref{fig:hartley_basis} shows that, for $k \leq n_0/2$, the fine-grid signal $\bh^1_k$ has the same frequency as the original, coarse-grid signal $\bh^0_k$, while $\bh^1_{n_0+k}$ has higher frequency.
Conversely, for $k > n_0/2$, the fine-grid signal $\bh^1_{n_0+k}$ has the same frequency as the coarse-grid signal $\bh^0_k$, while $\bh^1_k$ has higher frequency.
In both cases, the prolongation of a coarse-grid signal introduces spurious high-frequency (i.e., small-scale) components to the prolongated signal.
Fortunately, both $\bh^1_k$ and $\bh^1_{n_0+k}$ are damped by a factor which is closer to zero for the spurious, high-frequency signals than for the consistent, low-frequency signals.
Specifically, $c_{k}$ tends to $1$ and $c_{n_0+k}$ tends to $0$ as $k$ tends to $0$, thus damping  more severely the spurious, high-frequency signals $\bh^1_{n_0+k}$ than the consistent, low-frequency signals $\bh^1_{k}$.
Conversely, $c_{k}$ tends to $0$ and $c_{n_0+k}$ tends to $1$ as $k$ tends to $n_0$, thus damping more severely the spurious, high-frequency signals $\bh^1_{k}$ than the consistent, low-frequency signals $\bh^1_{n_0+k}$.

For the restriction operator $\bR \bydef \bR_1^0$ defined in \cref{eq:def_prolongation_restriction}, we have (see \cref{app:proof_restrict_hartley})
\begin{equation}\label{eq:restrict_hartley}
    \bR \bh^1_k
    = 2 c_k \bh^0_k,
    \quad 
    \mbox{and }
    \bR \bh^1_{n_0+k}
    = -2 c_{n_0+k} \bh^0_k,
    \quad \forall k=0,\dots,n_0-1.
\end{equation}
Restricting a fine-grid basis vector $\bh^1_k$, with $k=0,\dots,n_0-1$, produces a (coarse-grid) vector proportional to the corresponding coarse-grid basis vector $\bh^0_k$, specifically by a factor $2c_k \leq 2$.
For $n_0/2 < k < n_0$, the restricted signal has lower frequency than the original, fine-grid signal, as illustrates \cref{fig:hartley_basis}.
Similarly, restricting a fine-grid basis vector $\bh^1_{k}$, with $k=n_0,\dots,2n_0-1$, produces a (coarse-grid) vector proportional to the complementary coarse-grid basis vector $\bh^0_{k-n_0}$, specifically by a factor $-2c_{k} < 2$.
Again, high-frequency fine-grid signals $\bh^1_{k}$ corresponding to $n_0 \leq k < 3n_0/2$ are restricted to a signal with lower frequency.
In conclusion, high-frequency fine-grid basis vectors that cannot be represented on the coarse grid are thus restricted to lower-frequency signals.
Fortunately, such signals are the most damped, since they correspond to ranges of $k$ where $c_k$ is closer to $0$.

\begin{rmk}\label{rmk:mat_ident_restr_prol_hartley}
Denoting
$\bC \bydef 
    \begin{bmatrix}
    \Diag(\{c_k\}_{k=0}^{n_0-1}) & \Diag(\{-c_{n_0+k}\}_{k=0}^{n_0-1})
    \end{bmatrix}
    \in \R^{n_0 \times 2n_0}
$,
the identities in \cref{eq:prolong_hartley,eq:restrict_hartley} may be compactly recast as
\begin{align}\label{eq:mat_ident_restr_prol_hartley}
    \bP \bH_0 & = \bH_1 \bC^{\transp},
    & \bR\bH_1 & = 2\bH_0 \bC,
    &\bH_1^{\transp} \bP & = 2\bC^{\transp} \bH_0^{\transp}
    & \bH_0^{\transp} \bR &= \bC \bH_1^{\transp},
\end{align}
where the last two identities follow from the first two by exploiting the orthogonality of $\bH_0$ and $\bH_1$.
\end{rmk}

\subsection{Two-level MLMC with linear, symmetric, circulant simulators}

We now assume that the operators $\tilde{\bF}_{\ell}$ are symmetric, circulant matrices. 
Such matrices can be diagonalized in the Hartley basis~\cite{bini_matrix_1993} (see \cref{app:symcirc_diag_hartley} for the proof with the cell-centered basis), i.e.,
\begin{equation}\label{eq:F_diag_Hartley}
    \tilde{\bF}_{\ell} = \bH_{\ell} \bLambda_{\ell} \bH_{\ell}^{\transp},
    \quad
    \bLambda_{\ell} \bydef \Diag(\{\lambda^{\ell}_k\}_{k=0}^{n_\ell-1}),
    \quad
    \mbox{for } \ell \in \{0,1\}.
\end{equation} 
This property, along with the identities in
\cref{eq:mat_ident_restr_prol_hartley},
allows 
$\bF_0 \bydef \bP \tilde{\bF}_0 \bR 
= \bP \bH_0 \bLambda_0 \bH_0^{\transp} \bR$
to be decomposed as ${\bF_0 = \bH_1 \bM \bH_1^{\transp}}$,
where 
$\bM \bydef 
\bC^{\transp} \bLambda_0 \bC
= 
\begin{bmatrix}\label{eq:matrix_M}
    \bM_{11} & \bM_{12} \\
    \bM_{21} & \bM_{22}
\end{bmatrix}$,
with
\begin{align}
    \bM_{11} &\bydef \Diag(\{c_k^2 \lambda_k^0\}_{k=0}^{n_0-1} ),
    \label{eq:MLMC_M11}\\
    \bM_{22} &\bydef \Diag(\{c_{n_0+k}^2 \lambda_k^0\}_{k=0}^{n_0-1} ),
    \label{eq:MLMC_M22}\\
    \bM_{12} &\bydef \Diag(\{-c_k c_{n_0+k} \lambda_k^0\}_{k=0}^{n_0-1} )  = \bM_{21}.
    \label{eq:MLMC_M12_M21}
\end{align}
As a consequence, \cref{eq:variance_mlmc_linear_models_frob} becomes
\begin{equation}
    \mathcal{V}( \muMLMC_1)
    = \frac{1}{M_0} \| \bH_1 \bM \bH_1^{\transp} \bG^{1/2} \|^2_{F, \bW_{L}} 
    + \frac{1}{M_1} \| \bH_1 (\bLambda_1 - \bM) \bH_1^{\transp} \bG^{1/2} \|^2_{F, \bW_{L}},
\end{equation}
which, owing to the orthogonality of the Hartley matrix, and more particularly \cref{eq:prop_norms_orthog_proj_frob}, can be recast as
\begin{equation}\label{eq:variance_mlmc_circulant_models}
    \mathcal{V}( \muMLMC_1)
    = \frac{1}{M_0} \| \bM \bH_1^{\transp} \bG^{1/2} \|^2_{F, \bI_{n_L}} 
    + \frac{1}{M_1} \| (\bLambda_1 - \bM) \bH_1^{\transp} \bG^{1/2} \|^2_{F, \bI_{n_L}}.
\end{equation}
To reduce the variance of the correction term in \cref{eq:variance_mlmc_circulant_models}, the difference between $\bM$ and $\bLambda_1$ needs to be as small as possible. 
First, we note that the two off-diagonal blocks $\bM_{12}$ and $\bM_{21}$, which are themselves diagonal matrices, contribute to increasing this difference. 
On the main diagonal, i.e., in the diagonal blocks $\bM_{11}$ and $\bM_{22}$, scaled eigenvalues of $\tilde{\bF}_0$ appear twice.
To compare these diagonal blocks to the eigenvalues of $\tilde{\bF}_1$ in $\bLambda_1$, further assumptions on $\tilde{\bF}_0$ are required.
We thus introduce the Galerkin coarse-grid operator, which is an algebraic way of constructing the coarse-grid operator $\tilde{\bF}_0$ from the fine-grid operator $\tilde{\bF}_1$, and which is widely used in multigrid methods and their analysis~\cite{trottenberg_multigrid_2000}.
Specifically, the Galerkin operator is defined as
\begin{equation}\label{eq:def_galerkin}
    \tilde{\bF}_0 
    \bydef 
    \frac{1}{4} \bR \tilde{\bF}_1 \bP \in \mathbb{R}^{n_0 \times n_0}.
\end{equation}
As shown in \cref{sec:annex_galerkin}, this operator is the optimal operator in terms of minimizing
$\| \tilde{\bF}_1 - \bP \tilde{\bF}_0 \bR \|_{F,\bW_1}^2$ 
for $\bW_1 = n_1^{-1} \bI_{n_1}$ and the grid transfer operators defined by \cref{eq:def_prolongation_restriction}. 
It follows from \cref{eq:def_galerkin,eq:mat_ident_restr_prol_hartley} that the Galerkin operator $\tilde{\bF}_0$ can be diagonalized in the Hartley basis as $\tilde{\bF}_0 = \bH_0 \bLambda_0 \bH_0^{\transp}$ with 
$\bLambda_0 = \bC \bLambda_1 \bC^{\transp}$.
In other words,
the eigenvalues $\{\lambda_k^0\}_{k=0}^{n_0-1}$ of the Galerkin operator $\tilde{\bF}_0$ can be expressed from the eigenvalues $\{\lambda_k^1\}_{k=0}^{n_1-1}$ of $\tilde{\bF}_1$,
\begin{equation}\label{eq:eigenvalues_galerkin}
    \lambda_k^0 = 
    c_k^2 \lambda_k^1 
    + 
    c_{n_0+k}^2 \lambda_{n_0+k}^1,
    \quad \forall k=0,\dots,n_0-1.
\end{equation}
The resulting blocks of $\bM$ then read 
\begin{align}
    \bM_{11} &= \Diag(\{ c_k^4 \lambda_k^1 + c_k^2 c_{n_0+k}^2 \lambda_{n_0+k}^1 \}_{k=0}^{n_0-1} ), 
    \label{eq:MLMC_M11_gal}\\
    \bM_{22} &= \Diag(\{ c_{n_0+k}^4 \lambda_{n_0+k}^1 +  c_k^2 c_{n_0+k}^2 \lambda_k^1 \}_{k=0}^{n_0-1} )
    = \Diag(\{ c_{k}^4 \lambda_{k}^1 +  c_{k-n_0}^2 c_{k}^2 \lambda_{k-n_0}^1 \}_{k=n_0}^{2n_0-1} ), 
    \label{eq:MLMC_M22_gal}\\
    \bM_{12} &= \Diag(\{ -c_k^3 c_{n_0+k} \lambda_k^1 - c_k c_{n_0+k}^3 \lambda_{n_0+k}^1 \}_{k=0}^{n_0-1} ) = \bM_{21}.
    \label{eq:MLMC_M12_M21_gal}
\end{align}
Using elementary trigonometric identities, we remark that 
$c_{k-n_0} = c_{n_0+k-2n_0} = -c_{n_0+k}$. 
Therefore, the main diagonal of $\bM$ deviates from $\bLambda_1$ by a multiplicative damping factor $c_k^4 \leq 1$ for ${k=0,\ldots,2n_0-1}$, on the one hand, and by the addition of a spurious, complementary eigenvalue, though also damped by $c_k^2c_{n_0+k}^2 < 1$ for ${k=0,\ldots,2n_0-1}$.
The off-diagonal blocks $\bM_{12}$ and $\bM_{21}$ introduce spurious terms that contribute to the difference ${\bLambda_1 - \bM}$ and thus increase the variance of the 2-level estimator.
Note that, because these terms are to be compared with $0$, the comparison with the eigenvalues of the fine-grid operator $\tilde{\bF}_1$ is of little interest, which is why we simply consider the damping factors $-c_k c_{n_0+k}$ with respect to eigenvalues $\lambda^0_k$ of the coarse-grid operator $\tilde{\bF}_0$, for ${k=0, \ldots, n_0-1}$, given by \cref{eq:MLMC_M12_M21}.

The evolution with $k=0,\ldots,n_1-1$ of the four damping factors is presented in \cref{fig:coeff_M_unfiltered}.
We observe that the eigenvalues $\lambda^1_k$ associated with low-frequency vectors of the fine Hartley basis, i.e., for $k$ close to $0$ and $n_1-1$ (see \cref{fig:hartley_basis}), are well-represented on the main diagonal of $\bM$, since $c_k^4 \approx 1$.
At the same time, the damping factors $c_k^2c_{n_0+k}^2$ corresponding to spurious components are close to $0$, resulting in small values of the first and last few diagonal entries of the difference $\bLambda_1 - \bM$.
In particular, we note that the first entry of $\bM_{11}$ exactly matches that of $\bLambda_1$, i.e., $\lambda^1_0$, because $c_0=1$ and $c_{n_0} = 0$.
On the other hand, eigenvalues $\lambda^1_k$ associated with medium- to high-frequency vectors of the fine Hartley basis are severely damped since $c_k^4$ quickly decreases to $0$ as $k$ approaches $n_0=n_1/2$.
Spurious diagonal components are also somewhat damped by a factor $c_k^2c_{n_0+k}^2$, which is maximal for $k\in\{n_1/4, 3n_1/4\}$.
Finally, the spurious, off-diagonal components are damped by factors $-c_k^3 c_{n_0+k}$ and $-c_k c_{n_0+k}^3$, whose magnitudes are maximal for $k\in\{n_1/8+1, 7n_1/8-1\}$ and $k\in\{3n_1/8-1,5n_1/8+1\}$.

\begin{figure}[!ht]
    \centering%
    \subfigure[Damping factors in {$\bM$}.\label{fig:coeff_M_unfiltered}]{\includegraphics[width=.45\linewidth]{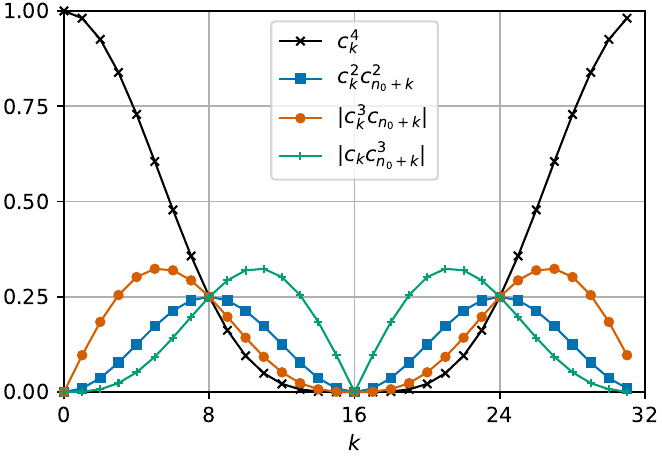}}\quad%
    \subfigure[Damping factors in {$\bar{\bM}$}.\label{fig:coeff_M_filtered}]{\includegraphics[width=.45\linewidth]{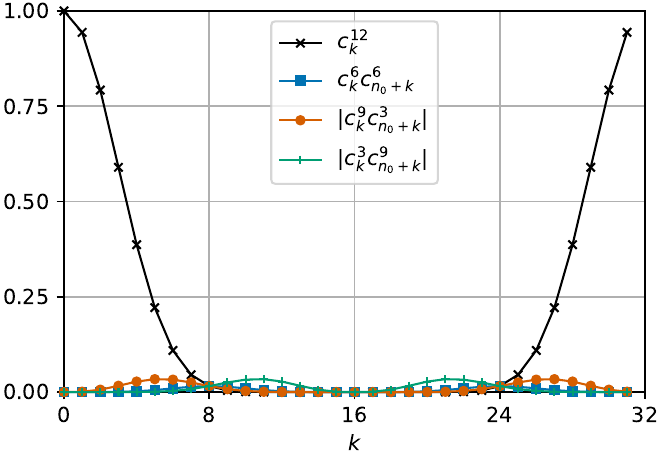}}
    \caption{Damping factors of the eigenvalues in $\bM$ and $\bar{\bM}$ as functions of $k = 0, \dots, n_1-1$ with $n_1=2n_0=32$, when using the Galerkin operator \cref{eq:def_galerkin}.
    The black curves correspond to the factors of the correct eigenvalues on the main diagonal, while blue curves represent the factors of the spurious eigenvalues.
    The orange and green curves correspond to the factors of the off-diagonal blocks.}
    \label{fig:coeff_M}
\end{figure}

\subsection{Two-level F-MLMC with linear, symmetric, circulant simulators}

We now turn to the F-MLMC estimator based on the second-order Shapiro filter \cite{shapiro_smoothing_1970, falissard_genuinely_2013, falissard_uneven-order_2015} defined as
\begin{equation}\label{eq:def_filtering_operator}
    \bS_{\ell} \bydef \frac{1}{4}
    \begin{bmatrix}
        2 & 1 & & & 1 \\
        1 & 2 & 1 & & \\
        & \ddots & \ddots & \ddots & \\
        & & 1 & 2 & 1 \\
        1 & & & 1 & 2
    \end{bmatrix} \in \R^{n_\ell \times n_\ell}.
\end{equation}
With the specific grid transfer operators defined in \cref{eq:def_prolongation_restriction}, the operator $\bar{\bP}_{\ell-1}^{\ell}$ corresponds to the linear interpolation operator between the levels $\ell-1$ and $\ell$.
We again restrict ourselves to a two-level analysis and define the filtered grid transfer operators as $\bar{\bP} = \bS_1 \bP$ and $\bar{\bR} = \bR \bS_1$, where $\bP \bydef \bP_0^1$ and $\bR \bydef \bR_1^0$ as in \cref{sec:hartley-basis}, and where $\bS_1$ denotes the second-order Shapiro filter defined in \cref{eq:def_filtering_operator}.
Similarly to \cref{eq:prolong_hartley,eq:restrict_hartley}, it is possible to study the effect the filtered transfer operators on the Hartley basis vectors (see \cref{app:proof_filtered_prol_restr_hartley}).
For the prolongation $\bar{\bP}$, we have
\begin{equation}\label{eq:filtered_prolong_hartley}
    \bar{\bP} \bh^0_k
    = c_k^3 \bh^1_k - c_{n_0+k}^3 \bh^1_{n_0+k},
    \quad \forall k=0,\dots,n_0-1.
\end{equation}%
The addition of the Shapiro filter raises the damping factors $c_k$ to the power 3. The prolongated Hartley basis vectors are thus more severely damped than they were without filtering.
Again, the most damped fine-grid basis vectors $\bh^1_k$ are those corresponding to $k$ close to $n_0=n_1/2$, i.e., high-frequency signals.
For the restriction operator we have
\begin{equation}\label{eq:filtered_restrict_hartley}
    \bar{\bR} \bh^1_k
    = 2 c_k^3 \bh^0_k,
    \quad 
    \mbox{and }
    \bar{\bR} \bh^1_{n_0+k}
    = - 2 c_{n_0+k}^3 \bh^0_k,
    \quad \forall k=0,\dots,n_0-1.
\end{equation}%
Similar conclusions can be drawn as for the unfiltered case, but with the damping factors raised to the power of 3, thus increasing their effect, which still affects more strongly the fine-grid, high-frequency signals that cannot be represented on the coarse grid.
Similarly to the unfiltered case (see \cref{rmk:mat_ident_restr_prol_hartley}), identities \cref{eq:filtered_prolong_hartley,eq:filtered_restrict_hartley} can be recast as
\begin{align}\label{eq:mat_ident_restr_prol_hartley_filtered}
    \bar{\bP} \bH_0 & = \bH_1 \bC_3^{\transp},
    & \bar{\bR}\bH_1 & = 2\bH_0 \bC_3,
    &\bH_1^{\transp} \bar{\bP} & = 2\bC_3^{\transp} \bH_0^{\transp}
    & \bH_0^{\transp} \bar{\bR} &= \bC_3 \bH_1^{\transp},
\end{align}
where 
$\bC_3 \bydef 
    \begin{bmatrix}
    \Diag(\{c_k^3\}_{k=0}^{n_0-1}) & \Diag(\{-c_{n_0+k}^3\}_{k=0}^{n_0-1})
    \end{bmatrix}
    \in \R^{n_0 \times 2n_0}
$.%

The impact of filtering on the total variance of the MLMC estimator is now assessed, considering a 2-level MLMC estimator and assuming that $\tilde{\bF}_1$ and $\tilde{\bF}_0$ are symmetric, circulant matrices.
From \cref{eq:F_diag_Hartley,eq:mat_ident_restr_prol_hartley_filtered}, we deduce the decomposition 
$\bar{\bF}_0 \bydef \bar{\bP} \tilde{\bF}_0 \bar{\bR} = \bH_1 \bar{\bM} \bH_1^{\transp}$, 
where 
$
\bar{\bM} 
\bydef
\begin{bmatrix}
    \bar{\bM}_{11} & \bar{\bM}_{12} \\
    \bar{\bM}_{21} & \bar{\bM}_{22}
\end{bmatrix}
$,
with 
\begin{align}
    \bar{\bM}_{11} &= \Diag(\{ c_k^6 \lambda_k^0 \}_{k=0}^{n_0-1} ),\\
    \bar{\bM}_{22} &= \Diag(\{ c_{n_0+k}^6 \lambda_k^0 \}_{k=0}^{n_0-1} ),\\
    \bar{\bM}_{12} &= \Diag( \{ -c_k^3 c_{n_0+k}^3 \lambda_k^0 \}_{k=0}^{n_0-1} ) = \bar{\bM}_{21}.
\end{align}%
Upon replacing $\bF_0$ with its filtered counterpart $\bar{\bF}_0$, \cref{eq:variance_mlmc_linear_models_frob} can be written as in \cref{eq:variance_mlmc_circulant_models}, but with $\bM$ replaced by $\bar{\bM}$.
The sparsity pattern of $\bar{\bM}$ is identical to that of $\bM$, and its entries are similar, but with damping factors raised to increased powers.
We remark that the off-diagonal blocks $\bar{\bM}_{12}$ and $\bar{\bM}_{21}$ have entries that are more strongly damped than those of their unfiltered counterparts, $\bM_{12}$ and $\bM_{21}$,
as can be visualized in \cref{fig:coeff_M} (orange plots),
which contributes to reducing the off-diagonal entries of $\bLambda_1 - \bar{\bM}$ compared to those of $\bLambda_1 - \bM$.

To study more closely the diagonal entries of $\bar{\bM}$ and compare them to  the eigenvalues of $\tilde{\bF}_1$ in $\bLambda_1$ we resort to the Galerkin operator defined by $\tilde{\bF}_0 \bydef \frac{1}{4} \bar{\bR} \tilde{\bF}_1 \bar{\bP}$.
This definition is inspired by the form of \cref{eq:def_galerkin}, although there is no guarantee of its optimality.
Then, from \cref{eq:F_diag_Hartley,eq:mat_ident_restr_prol_hartley_filtered}, it follows that the Galerkin operator $\tilde{\bF}_0$ can be diagonalized in the Hartley basis as $\tilde{\bF}_0 = \bH_0 \bLambda_0 \bH_0^{\transp}$ with 
$\bLambda_0 = \bC_3 \bLambda_1 \bC_3^{\transp}$,
or, equivalently,
\begin{equation}\label{eq:eigenvalues_galerkin_filtered}
    \lambda_k^0 = 
        c_k^6 \lambda_k^1 
        + 
        c_{n_0+k}^6 \lambda_{n_0+k}^1 ,
    \quad \forall k=0,\dots,n_0-1.
\end{equation}%
The diagonal blocks of $\bar{\bM}$ thus become
\begin{align}
    \bar{\bM}_{11}
    &= \Diag(\{ c_k^{12} \lambda_k^1 + c_{n_0+k}^6 c_k^6 \lambda_{n_0+k}^1 \}_{k=0}^{n_0-1} ), \\
    \bar{\bM}_{22}
    &= \Diag(\{ c_{n_0+k}^{12} \lambda_{n_0+k}^1 + c_{n_0+k}^6 c_k^6 \lambda_k^1 \}_{k=0}^{n_0-1} )
    = \Diag(\{ c_{k}^{12} \lambda_{k}^1 +  c_{k-n_0}^6 c_{k}^6 \lambda_{k-n_0}^1 \}_{k=n_0}^{2n_0-1} ), \\
    \bar{\bM}_{12}
    &= \Diag(\{ -c_k^9 c_{n_0+k}^3 \lambda_k^1 - c_k^3 c_{n_0+k}^9 \lambda_{n_0+k}^1 \}_{k=0}^{n_0-1} ) = \bar{\bM}_{21}.
\end{align}
The conclusions are the same as for the diagonal blocks of $\bM$, but with damping factors raised to higher powers.
Consequently, the spurious eigenvalues on the main diagonal are damped more strongly than in the unfiltered case, as can be visualized in \cref{fig:coeff_M} (blue plots).
Moreover, the factors $c_k^{12}$ also damp more strongly the consistent eigenvalues compared to the factors $c_k^4$ of $\bM$, as can be visualized in \cref{fig:coeff_M} (black plots).
Although the addition of filters induces a certain loss of (mainly high-frequency) information, the spurious signals introduced by the grid transfer operators are significantly reduced.
These results suggest that, for filtering to be beneficial, a tradeoff needs to be found between reducing the detrimental effects induced by the grid transfer operators and degrading of the information of the original signal.
The second-order Shapiro filters considered in this paper seem to offer a good compromise.
Further endeavors to study and improve the filtering process may be pursued in future work.

\section{Numerical experiments}%
\label{sec:experiments}%

As discussed in the introduction, one area where the estimation of the expectation of discretized random fields arises is covariance modelling, specifically when estimating the intrinsic variances of a diffusion-based covariance operator using a randomization method~\cite{weaver_correlation_2001, weaver_evaluation_2021}.
This particular problem, which is briefly outlined below, will be the focus of our numerical experiments.

\subsection{Problem  description}\label{sec:problem_presentation}
Let $u \colon \mathcal{D} \to \R$ and $b \colon \mathcal{D} \to \R$ be
square-integrable functions on the domain $\mathcal{D} \subset \mathbb{R}^d$
where $d \in \{1,2,3\}$ is the spatial dimension. 
We consider numerical solutions of the following elliptic equation, subject to 
application-dependent boundary conditions (BCs):
\begin{equation}\label{eq:elliptic}
    ( I - \nabla \cdot \bK \nabla)^m u = b,
\end{equation}
where $m$ is a positive integer, $I$ is the identity operator,
and $\bK \colon \mathcal{D} \to \R^{d \times d}$ 
is a symmetric, positive-definite (SPD) tensor field with entries $[K_{ij}]_{i,j=1,\ldots,d}$.
If $\bK$ is constant, the integral solution on $\R^d$ corresponds to the application of a covariance operator, whose kernel is a covariance function from the Mat\'ern class, to $b$ \cite{whittle63,guttorp06}.

We assume that the operator in \cref{eq:elliptic} is discretized in space on a (not necessarily structured) 
grid of $n$ cells.  
We can then deduce the covariance matrix associated with the numerical solution of \cref{eq:elliptic} as
\begin{equation}\label{eq:diffusion_op}
    \bL \bydef ( \bI - \bDelta )^{-m}
    \bW^{-1},
\end{equation}
where $\boldsymbol{\Delta}$ is the matrix representing a spatial discretization of the differential operator ${\nabla \cdot \bK \nabla}$,
and ${\bW \in \mathbb{R}^{n \times n}}$ is an SPD Gram matrix that encodes the geometrical and structural information related to the discrete approximation of the diffusive term on the grid.
Specifically, $\bW$ is such that $\bDelta$ is self-adjoint with respect to
$\langle \cdot, \cdot \rangle_{\bW}$,
i.e., $\bW\bDelta = \bDelta^\transp \bW$. 
Consequently, the matrix $\bL$ is self-adjoint (symmetric) with respect to the canonical inner product.
In the experiments, we consider only a diagonal diffusivity tensor $K_{ij} = K_{ij}\delta_{ij}$ where
$\delta_{ij}$ is the Kronecker delta. Specifically, we define the diagonal elements according to the relation
${K_{ii}(\bm{x}) = (2m-d-2)^{-1} (D_{ii}(\bm{x}))^2}$ where the elements
$[D_{ii}(\bm{x})]_{i=1,\ldots,d}$ correspond to the directional correlation length-scales at the spatial location $\bm{x}$, and $m > d/2 + 1$~\cite[section~3]{weaver_diffusion_2013}.

The matrix $\bL$ is SPD but does not define a covariance matrix with meaningful variances for applications like data assimilation. 
As such, $\bL$ must be normalized by its diagonal so that the desired variances can be applied. Thus, we define the covariance matrix of interest as
$\bB = \bSigma \bGamma \bL \bGamma \bSigma$,
where $\bGamma = \Diag(\diag(\bL))^{-1/2}$ is a normalizing diagonal matrix 
such that $\Diag(\diag(\bGamma\bL\bGamma)) = \bI_n$,
and
$\bSigma^2 = \bSigma \bSigma$ is the diagonal matrix with entries corresponding to the desired variances,
i.e., $\diag(\bB) = \diag(\bSigma^2) = (\sigma_1^2, \ldots, \sigma_n^2)$.
In these expressions, the operator $\diag(\cdot)$ maps a matrix to the vector consisting of the diagonal elements of that matrix, while the operator $\Diag(\cdot)$ maps a vector to the diagonal matrix whose diagonal consists of the entries of that vector.

The main computational challenge lies in computing the diagonal $\btheta \bydef \diag(\bL) = (\theta_1, \ldots, \theta_n)$.
Indeed, for large-scale problems, the matrix $\bL$ is not assembled, and only applications of $\bL$ to vectors are accessible.
Thus, its diagonal entries are not explicitly stored and need to be determined differently~\cite{weaver_evaluation_2021}.
A direct way would be to recover these by applying $\bL$ to the canonical basis vectors of $\R^n$, i.e., $\theta_k = (\bL \be_k)_k$, for $k=1,\ldots,n$.
For large $n$, this approach is not computationally tractable.
Here, we focus on an alternative strategy, used in operational systems, consisting of approximating $\btheta$ by randomization~\cite{weaver_evaluation_2021}.
Taking $m=2q$, $\bL$ can be subsequently factored (see \cref{app:facto_L}) as
$\bL = \bA \bW \bA^\transp$, where
\begin{equation}\label{eq:sqrt_diffusion_op}%
    \bA 
    \bydef ( \bI - \boldsymbol{\Delta} )^{-q} \bW^{-1},
\end{equation}%
which, as for $\bL$ itself, cannot be explicitly assembled in large-scale applications.
This decomposition of $\bL$ implies that, for any random vector $\bX$ with $\Exp[\bX] = \bzero_{n}$ and 
$\Exp[\bX\bX^\transp] = \bW$,
\begin{equation}\label{eq:theta}
    \btheta
    =
    \diag(\bL)
    =
    \diag(\bA \Exp[\bX \bX^\transp] \bA^\transp)
    =
    \diag(\Cov[\bA\bX])
    =
    \Var[\bA\bX]
    =
    \Exp[\bA\bX \odot \bA\bX],
\end{equation}
where 
$\Exp[\cdot]$ and $\Var[\cdot]$ denote the element-wise expectation and variance of a random vector, $\Cov[\cdot]$ denotes the covariance matrix of a random vector,
and where
$\odot$ denotes the Schur product (a.k.a.\ the Hadamard or element-wise product) between two vectors or matrices of the same size.
One way to construct such a vector $\bX$ is to define $\bX = \bV \bZ$, where $\bV$ arises from the factorization ${\bW = \bV \bV^{\transp}}$, and where $\bZ$ is a random vector with $\Exp[\bZ] = \bzero_{n}$ and $\Exp[\bZ\bZ^\transp] = \bI_n$.
In the rest of this paper, we consider normal random vectors $\bX \sim \mathcal{N}(\bzero_{n}, \bW)$, which can be constructed by $\bX= \bV\bZ$, with $\bZ \sim \mathcal{N}(\bzero_{n}, \bI_{n})$.
The expectation $\btheta$ of the $\R^n$-valued random vector ${\bY \bydef \bA\bX \odot \bA\bX}$ can then be estimated using MC sampling. 
Given a random $M$-sample $\{\bX^{(i)}\}_{i=0}^M$ of $\bX$, an unbiased estimator $\hat{\btheta}$ of $\btheta$ is the sample mean
\begin{equation}\label{eq:mc_diffusion}
    \hat{\btheta} = \frac{1}{M} \sum_{i=1}^{M} (\bA \bX^{(i)}) \odot (\bA \bX^{(i)}).
\end{equation}
We remark that the estimator $\hat{\btheta}$ only requires $M$ applications of $\bA$ to a vector, typically with ${M \ll n}$ in large-scale applications.
Furthermore, by construction, the MC estimator defined by \cref{eq:mc_diffusion} yields non-negative estimates, which is a fundamental requirement for the problem under consideration. 
Alternative MC estimators have been proposed in the literature for estimating the diagonal of a general (not necessarily symmetric) matrix \cite{Bekas2007_EstimatorDiagonalMatrix,hallman_monte_2023}, which do not employ a factored form of the matrix and hence do not guarantee non-negative estimates.
Here, we investigate the use of the MLMC and F-MLMC estimators described in \cref{sec:mlmc_for_discretized_fields,sec:fmlmc} to improve (in terms of RMSE) the estimation of $\btheta$ and hence the efficiency of determining accurate normalization coefficients $\bGamma$.
It should be noted that the proposed {(F)-}MLMC construction does not guarantee non-negative estimates either, although negative estimates were never encountered in our experiments.

In practice, the matrix-vector product $\bA \bx$ can be computed by solving the sequence of SPD systems of linear equations $(\bW (\bI - \bDelta) \by^{(j)}= \bz^{(j)})_{j=1}^q$, with $\bz^{(1)} \bydef \bx$ and $\bz^{(j)} \bydef \bW\by^{(j-1)}$ for $j=2,\ldots,q$, so that $\by^{(q)} = \bA \bx$.
In our experiments, the numerical solving of these systems is achieved by precomputing a Cholesky decomposition of $\bW (\bI - \bDelta)$.
However, for large-scale problems, it is not reasonably possible to compute or store such a decomposition, and sparse iterative methods are typically used instead~\cite{Weaver2016_CorrelationOperatorsBased,Chrust2025_ImpactEnsemblebasedHybrid}.

\subsection{1D illustration}\label{sec:1d_illustration}

Before turning to the (F-)MLMC estimation of $\btheta$, we briefly consider a simpler experiment to verify numerically the conclusions of the spectral analysis of \cref{sec:spectral_analysis}.
Specifically, we consider a hierarchy of simulators $(f_\ell)_{\ell=0}^L$ defined through \cref{eq:f_simu_restr_prolong_fine} by $\tilde{f}_{\ell} \colon \bx_\ell \mapsto \bA_\ell \bx_\ell$, where $\bA_\ell$ is the factor in \cref{eq:sqrt_diffusion_op} arising from a cell-centered discretization on the 1D domain $\mathcal{D} := [0,1]$, with periodic boundary conditions. 
Specifically, we consider the same setting and grid transfer operators as in \cref{sec:spectral_analysis_1d_setting}, and
we define hierarchies of different depths (i.e., numbers of levels), corresponding to $L \in \{0, \ldots, 5\}$, with a fixed finest discretization corresponding to $n_L = 512$.
The diffusion tensor field $\bK$, which reduces here to a scalar field, is taken to be constant by setting ${D_{11}(\bm{x}) = D \in \R}$ for all ${\bm{x} \in \mathcal{D}}$.
The scalar value $D$ will be referred to as the length-scale.
We note that in this setting, the matrices $\bA_\ell$ are symmetric and circulant~\cite{Goux2024_ImpactCorrelatedObservation}.

The statistical parameter of interest is $\Exp[f_L(\bX_L)] = \Exp[\bA_L \bX_L]$, with $\bX_L \sim \mathcal{N}(\bzero_{n_L}, \bW_{L})$.
It trivially follows that $\Exp[f_L(\bX_L)] = \bzero_{n_L}$.
The MLMC estimator $\muMLMC_L$ is defined as in \cref{eq:mlmc_estimator_transfer,eq:mlmc_estimator_transfer_v2}, where the transfer operators are defined as in \cref{eq:def_prolongation_restriction}.
With this choice, it follows from \cref{eq:composed_restr_prolong,eq:prop_restr_prolong} that the MLMC estimator $\muMLMC_L$
involves, at each correction level, input random vectors that are such that 
$\bR_L^{0} \bX_L \sim \mathcal{N}(\bzero_{n_0}, \bW_0)$ and
$[\begin{matrix}(\bR_L^{\ell} \bX_L)^{\transp} & (\bR_L^{\ell-1} \bX_L)^{\transp}\end{matrix}]^{\transp}
\sim
\mathcal{N}(\bzero_{n_\ell + n_{\ell-1}}, \tilde{\bW}_\ell)$ for $\ell=1,\ldots,L$, with
\begin{equation}\label{eq:mlmc_coupling}
    \tilde{\bW}_\ell
    \bydef
    \begin{bmatrix}
        \bW_{\ell} & \bW_{\ell}\bP_{\ell-1}^{\ell}  \\
        \bR_{\ell}^{\ell-1}\bW_{\ell} & \bW_{\ell-1}
    \end{bmatrix} \in \R^{(n_\ell + n_{\ell-1}) \times (n_\ell + n_{\ell-1})},
\end{equation}
which has the same structure as \cite[eq.~{(3.12)}]{croci_efficient_2018}.
In fact, \cref{eq:mlmc_coupling} and \cite[eq.~{(3.12)}]{croci_efficient_2018} are equivalent under the functional, piecewise constant representation of the input fields.
The filtered estimator $\muFMLMC_L$ is defined as in \cref{eq:fmlmc_estimator}, with the same transfer operators and the 2nd-order Shapiro filters defined in \cref{eq:def_filtering_operator}.

We now describe the cost model that will be prescribed in the remainder of this paper.
  First, the cost $\tilde{\mathcal{C}}_\ell$ of evaluating $\tilde{f}_\ell$ is assumed to be proportional to the number of cells $n_\ell$ on level $\ell$, i.e., $\tilde{\mathcal{C}}_\ell = \alpha n_\ell$.
  This is consistent with a fixed number $q \times T$ of matrix-free applications of $\bA_\ell$, where $q$ represents the number of diffusion steps in \cref{eq:sqrt_diffusion_op}, i.e., the number of linear systems to be solved, and $T$ the fixed number of iterations used to solve each of these systems.
  Second, the cost $\mathcal{C}_\ell^{\transf}$ of applying the (possibly filtered) grid transfer operators is modeled as $\mathcal{C}_\ell^{\transf} = \beta \sum_{k=\ell+1}^L n_k$, where $\beta$ denotes the constant (level-independent) number of floating-point operations per cell on level $k$ for applying the (possibly filtered) restriction operator from level $k$ to level $k-1$.
  This model is consistent with the fact that the restriction from the finest grid is defined as the composition of restriction operators between successive levels (see \cref{eq:composed_restr_prolong,eq:composed_restr_prolong_fmlmc}).
  Exploiting the nestedness of the tessellations, and assuming a constant refinement factor $s \bydef n_\ell / n_{\ell-1} > 1$ between two successive grids, the restriction cost reduces to $\mathcal{C}_\ell^{\transf} = \frac{s}{s-1} \beta (s^{L-\ell}-1) n_\ell$.
  Eventually, the cost $\mathcal{C}_\ell = \tilde{\mathcal{C}}_\ell + \mathcal{C}_\ell^{\transf}$ of evaluating $f_\ell$ is given by
\begin{equation}
  \label{eq:cost_model}
  \mathcal{C}_{\ell} = s^{\ell-L} \left[ 1+\gamma \left( s^{L-\ell}-1 \right) \right] \mathcal{C}_L,
  \qquad
  \mbox{ with }
  \gamma \bydef \dfrac{s}{s-1} \dfrac{\beta}{\alpha} >0.
\end{equation}
It is then easy to see from \cref{eq:cost_model} that $\mathcal{C}_{\ell} / \mathcal{C}_{\ell-1} \to s$ as $\gamma \to 0$.

  In the target application, Chebyshev iterations (see, e.g., \cite{Golub1961_ChebyshevSemiiterativeMethods,Golub1961_ChebyshevSemiiterativeMethodsa,Golub2013_MatrixComputations}) are typically used~\cite{Chrust2025_ImpactEnsemblebasedHybrid}.
  We note that each iteration involves one application of the operator $(\bI_{n_\ell} - \bDelta_\ell)$, which, for a second-order discretization of the 1D Laplace operator represented by a 3-point stencil, requires 3 floating-point operations per cell, hence $\alpha = 3qT$.
  In what follows, we choose $m=2q=10$ and $T=20$, which are typical values~\cite{Chrust2025_ImpactEnsemblebasedHybrid}, so that $\alpha = 300$.
  For the unfiltered MLMC, the 1D restriction, as defined in \cref{eq:def_prolongation_restriction}, from a given fine level $k$ to the immediately coarser level $k-1$ requires $s=2$ operations per coarse cell, i.e., a total of $sn_{k-1}=n_k$ operations, hence $\beta=1$.
  The 1D pre-filtering defined in \cref{eq:def_filtering_operator}, represented by a 3-point stencil, requires 3 additional operations on level $k$, so that $\beta=4$ for the filtered MLMC.
\Cref{tab:cost_model_1d} summarizes the normalized costs $\mathcal{C}_{\ell} / \mathcal{C}_{L}$ and cost reduction factors $\mathcal{C}_{\ell}/ \mathcal{C}_{\ell-1}$ of the simulators $f_\ell$ (MLMC) and $\bar{f}_\ell$ (F-MLMC) for this 1D example.

\begin{table}[!ht]%
  \centering%
  {\sisetup{round-mode=places, round-precision=2, round-pad=true}%
    \begin{tabular}{ccccc}
      \toprule
      & \multicolumn{2}{c}{\textbf{MLMC}} & \multicolumn{2}{c}{\textbf{F-MLMC}}\\
      & \multicolumn{2}{c}{($\gamma = 1/150$)} & \multicolumn{2}{c}{($\gamma = 4/150$)}\\
      \cmidrule(lr){2-3}\cmidrule(lr){4-5}
      {$\ell$} & {$\mathcal{C}_{\ell} / \mathcal{C}_{L}$} & {$\mathcal{C}_{\ell}/ \mathcal{C}_{\ell-1}$}
      & {$\mathcal{C}_{\ell} / \mathcal{C}_{L}$} & {$\mathcal{C}_{\ell}/ \mathcal{C}_{\ell-1}$}\\
      \midrule
      $L$   & \num{1}             & \num{1.98675496689} & \num{1}             & \num{1.94805194805} \\
      $L-1$ & \num{0.50333333333} & \num{1.97385620915} & \num{0.51333333333} & \num{1.9012345679} \\
      $L-2$ & \num{0.255}         & \num{1.94904458599} & \num{0.27}          & \num{1.8202247191}\\
      $L-3$ & \num{0.13083333333} & \num{1.90303030303} & \num{0.14833333333} & \num{1.69523809524}\\
      $L-4$ & \num{0.06875}       & \num{1.82320441989} & \num{0.0875}        & \num{1.53284671533}\\
      $L-5$ & \num{0.03770833333} & N/A                 & \num{0.05708333333} & N/A \\
      \bottomrule
    \end{tabular}%
  }%
  \caption{Normalized costs $\mathcal{C}_{\ell} / \mathcal{C}_{L}$ and cost reduction factors $\mathcal{C}_{\ell}/ \mathcal{C}_{\ell-1}$ of the simulators $f_\ell$ (MLMC) and $\bar{f}_\ell$ (F-MLMC) corresponding to the 1D example of \cref{sec:1d_illustration}.}
  \label{tab:cost_model_1d}
\end{table}

For all the experiments conducted in the remainder of this paper, the optimal multilevel sample allocation is given by \cref{eq:optimal_sample_alloc} for a fixed computational budget $\mathcal{C} = 100 \mathcal{C}_L$.
The variances $\mathcal{V}_{\ell}$ required for the computation of \cref{eq:optimal_sample_alloc} are estimated in a preprocessing stage from a pilot sample of size \num{1000}.

In what follows, $\hat{\bmu}_L$ denotes an unbiased estimator of $\bmu_L \bydef \Exp[f_L(\bX_L)]$, typically, the plain MC estimator $\muMC_L$ on the finest level $L$, the MLMC estimator $\muMLMC_L$ or the F-MLMC estimator $\muFMLMC_L$.
We recall that the unbiasedness of $\hat{\bmu}_L$ implies that $\mse(\hat{\bmu}_L, \bmu_L) = \mathcal{V}(\hat{\bmu}_L)$.
In order to examine the different scales of the proposed estimators, we decompose the variance $\mathcal{V}(\hat{\bmu}_L)$ into contributions of the individual Hartley modes.
Specifically, exploiting the orthogonality of $\bH_L$, it follows from \cref{eq:prop_norms_orthog_proj} that
\begin{equation}\label{eq:mse_hartley_decomposition}
    \mathcal{V}(\hat{\bmu}_L) 
    = \Exp [ \| \hat{\bmu}_L - \bmu_L \|_{\bW_{L}}^2 ]
    = \Exp [ \| \bH_L^{\transp} \bV_L^{\transp} (\hat{\bmu}_L - \bmu_L) \|_{\bW_{L}}^2 ],
\end{equation}
which, owing to the linearity of the expectation operator and the fact that $\bW_L = n_L^{-1}\bI_{n_L}$, may be recast as
\begin{equation}
    \mathcal{V}(\hat{\bmu}_L) 
    = \|\bnu\|_1
    \bydef \sum_{k=0}^{n_L-1} \nu_k,
    \qquad
    \nu_k 
    \bydef 
    n_L^{-2} \Exp[ ((\bh^L_k)^{\transp}(\hat{\bmu}_L - \Exp[f_L(\bX_L)]))^2 ]
    =
    n_L^{-2} \Var[ (\bh^L_k)^{\transp} \hat{\bmu}_L ],
\end{equation}
where $\bnu = (\nu_k)_{k=0}^{n_L-1}$ will be referred to as the spectral variance.
Furthermore, we define the cumulative spectral variance $\bnucum = (\nucum_k)_{k=0}^{n_L-1}$ such that $\nucum_k = \sum_{k'=0}^k \nu_{k'}$, implying that the total variance is given by $\nucum_{n_L-1} = \|\bnu\|_1$.
For better visualization and interpretation, the columns of the 1D Hartley matrix $\bH_L$ used in \cref{eq:mse_hartley_decomposition} are actually reordered so that they are sorted by increasing representable frequency on the corresponding discrete grid (see \cref{fig:hartley_basis}, which depicts the Hartley basis vectors without reordering).
Specifically, the new matrix with reordered columns is obtained as $\bH_L \bPi$, where $\bPi = (\Pi_{j,k})_{j,k=0}^{n_L-1}$ is the permutation matrix defined by $\Pi_{j, 2k} = \delta_{j,k}$ and $\Pi_{j, 2k+1} = \delta_{j, n_L-k-1}$, for $j=0,\ldots,n_L-1$ and $k=0,\ldots, n_L/2-1$ (assuming $n_L$ is even), and where $\delta_{i,j}$ denotes the Kronecker delta.

\begin{figure}[!ht]
    \centering%
    \subfigure[Spectral variance {$\bnu$} of MC and MLMC estimators.\label{fig:illustration_spectral_006}]{\includegraphics[width=.45\linewidth]{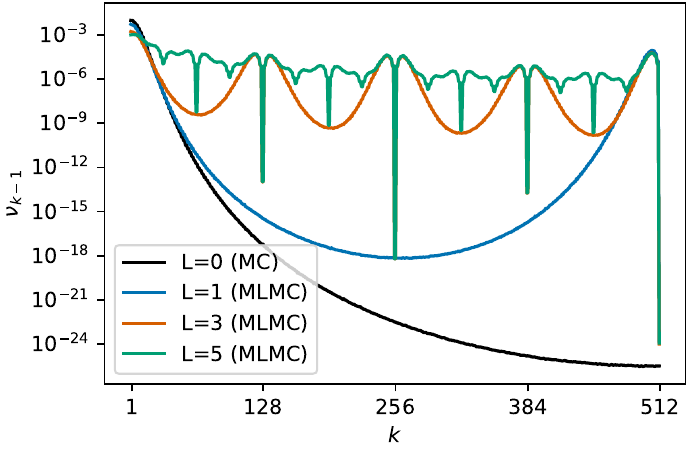}}\quad%
    \subfigure[Spectral variance {$\bnu$} of MC and F-MLMC estimators.\label{fig:illustration_filtered_spectral_006}]{\includegraphics[width=.45\linewidth]{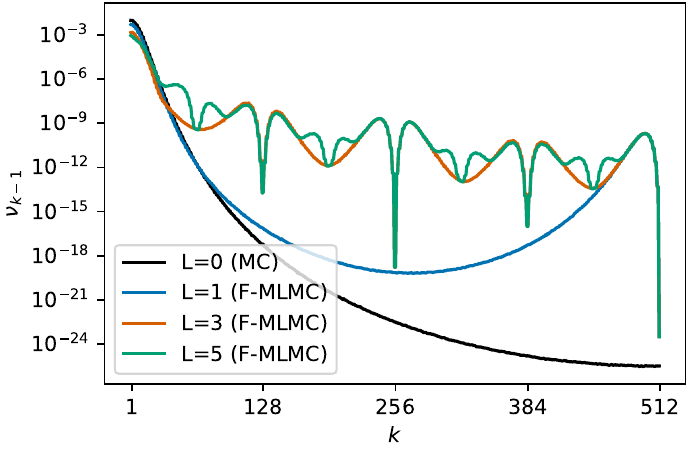}}\\
    \subfigure[Cumulative spectral variance {$\bnucum$} of MC and MLMC estimators.\label{fig:illustration_cml_006}]{\includegraphics[width=.44\linewidth]{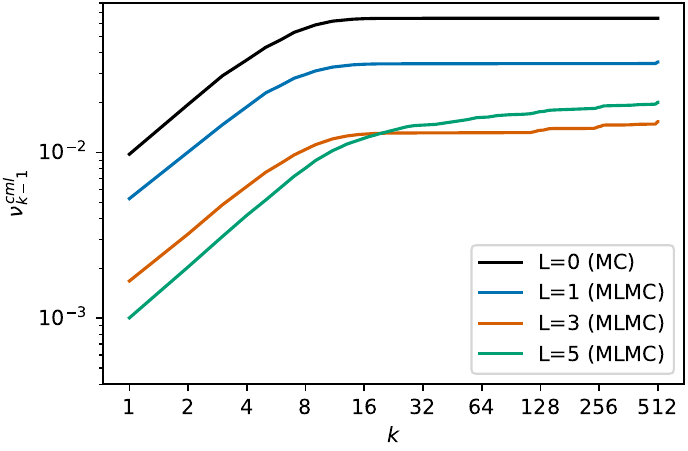}}\quad%
    \subfigure[Cumulative spectral variance {$\bnucum$} of MC and F-MLMC estimators.\label{fig:illustration_filtered_cml_006}]{\includegraphics[width=.44\linewidth]{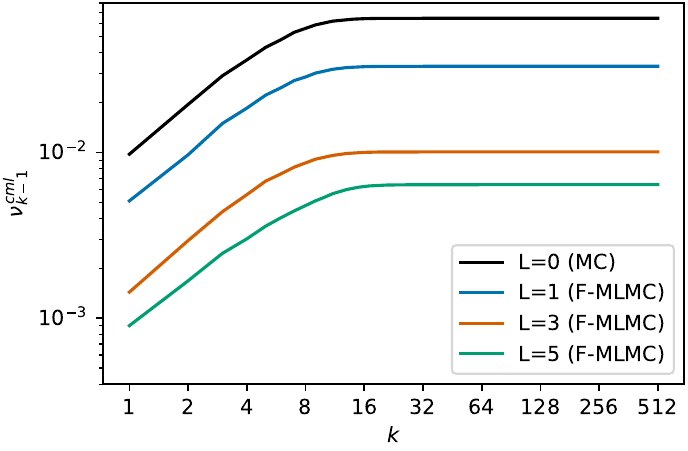}}
    \caption{Spectral variance (top) and cumulative spectral variance (bottom) of the MC estimator ($L=0$) and different MLMC (left) and F-MLMC (right) estimators ($L \in \{1,3,5\}$), for the illustrative 1D estimation problem described in \cref{sec:1d_illustration}, with length-scale $D=0.06$. 
    The finest level $L$ always corresponds to a discretization with $n_L=512$ cells, and the total budget is set to $\mathcal{C}=100\mathcal{C}_L$.
    The variance is estimated from \num{1000} estimators.}%
    \label{fig:illustration_spectral_cml_006}
\end{figure}

\Cref{fig:illustration_spectral_cml_006} reports the spectral variance $\bnu$, as well as the corresponding cumulative variance $\bnucum$, associated with the 2-, 4- and 6-level MLMC and F-MLMC estimators ($L\in\{1,3,5\}$) and with the single-level MC estimator ($L=0$), for the estimation of $\Exp[f_L(\bX_L)] = \bzero_{n_L}$, where the simulators correspond to length-scale $D=0.06$.
\Cref{fig:illustration_spectral_006,fig:illustration_filtered_spectral_006} show that, for the single-level MC estimator, most of the error arises from the first Hartley modes, which are associated with the large scales, or low frequencies, of the discretized field. 
This is confirmed in \cref{fig:illustration_cml_006,fig:illustration_filtered_cml_006}, where the cumulative variance of the MC estimator rapidly increases before reaching a plateau, showing that the variance is concentrated on the first few modes. 
For the 2-level MLMC estimator ($L=1$), the spectral variance plotted in \cref{fig:illustration_spectral_006} starts with a similar decay as that of the MC estimator in the low frequencies, before increasing again in the high frequencies.
Nevertheless, the variance is still concentrated on the first few modes, and the low-frequency components of the spectral variance are lower than those of the single-level MC estimator.
This translates into a lower plateau reached by the cumulative spectral variance, as shown in \cref{fig:illustration_cml_006}.
Moreover, the variance increase in the high-frequencies translates into a noticeable increase in the cumulative variance in the last few Hartley modes, which, in turn, results in a non-negligible increase in the total variance.
As more levels are added, the spectral variance significantly deteriorates in the high-frequencies.
While for the 4-level MLMC estimator ($L=3$), this deterioration is compensated by a lower variance in the low frequencies, this is no longer the case for the 6-level MLMC estimator ($L=5$), whose cumulative spectral variance eventually gets larger than that of the 4-level MLMC estimator, thus resulting in a larger total variance.

\Cref{fig:illustration_filtered_spectral_006,fig:illustration_filtered_cml_006} show the spectral variance $\bnu$ and the cumulative variance $\bnucum$ of the F-MLMC estimators with different grid hierarchies, corresponding to $L\in\{1,3,5\}$.
The effects of the filters are especially visible on the spectral variance plotted in \cref{fig:illustration_filtered_spectral_006}, which is significantly reduced, not only in the high frequencies, but also in the lowest ones (i.e., corresponding to the first few Hartley modes), compared to that of the unfiltered MLMC plotted in \cref{fig:illustration_spectral_006}.
The reduced error in the high frequencies (i.e., small scales) prevents the cumulative spectral variance, plotted in \cref{fig:illustration_filtered_cml_006}, from being significantly impacted in the last few Hartley modes, as opposed to that of the unfiltered MLMC, plotted in \cref{fig:illustration_cml_006}.
Furthermore, the reduced error in the low frequencies (i.e., large scales) translates into a lower plateau of the cumulative variance, hence a lower total variance, than for the unfiltered MLMC.
This is well visible in \cref{fig:illustration_total_variance_006}, which summarizes the total variance of the MC estimator and the different MLMC and F-MLMC estimators for $L\in\{1,\ldots,5\}$.
Specifically, the addition of filters leads to a 90\% reduction in total variance of the F-MLMC estimator compared to the single-level MC estimator, and to a 55\% reduction compared to the best, 5-level unfiltered MLMC estimator.

\begin{figure}[ht]
    \centering%
    \subfigure[{$D=0.06$}.\label{fig:illustration_total_variance_006}]{\includegraphics[width=.45\linewidth]{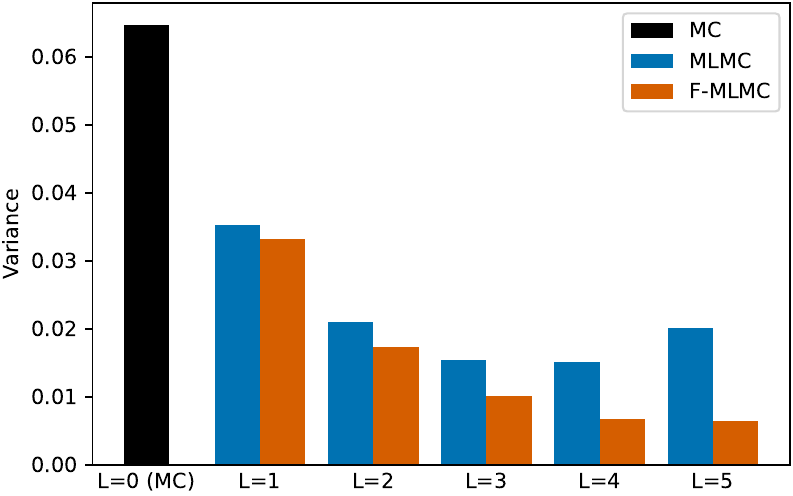}}\quad%
    \subfigure[{$D=0.01$}.\label{fig:illustration_total_variance_001}]{\includegraphics[width=.45\linewidth]{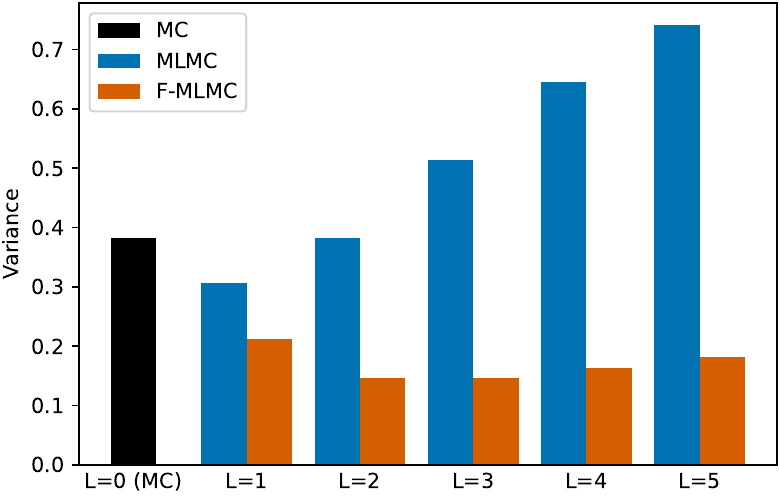}}
    \caption{Total variance of the MC estimator ($L=0$) and different MLMC and F-MLMC estimators ($L \in \{1,\ldots,5\}$), for the illustrative 1D estimation problem described in \cref{sec:1d_illustration}, with length-scale $D=0.06$ (left) and $D=0.01$ (right). 
    The finest level $L$ always corresponds to a discretization with $n_L=512$ cells, and the total budget is set to $\mathcal{C}=100\mathcal{C}_L$.
    The variance is estimated from \num{1000} estimators.}
    \label{fig:illustration_total_variance}
\end{figure}

\Cref{fig:illustration_pre_post} compares F-MLMC estimators with pre-filtering only, post-filtering only, and both pre- and post-filtering, along with the (unfiltered) MLMC and MC estimators, on the experiment with $D=0.06$.
The computational cost model for these estimators is the same as in \cref{eq:cost_model}, with $\beta=4$ for the F-MLMC estimator with pre-filtering only, and $\beta=1$ for the F-MLMC estimator with post-filtering only.
The total variances reported on \cref{fig:illustration_total_var_pre_post} show that both the pre-filtering and the post-filtering operations contribute to reducing the variance of the estimator.
The cumulative variance of the 6-level estimators ($L=5$) presented in \cref{fig:illustration_cml_pre_post} explicitly shows the effects of the pre- and post-filtering on the different frequencies of the variance.
Using an estimator with post-filtering alone leads to a reduction of the variance in the high frequencies, since the spurious high-frequencies components introduced by the prolongation operator are damped.
In addition, we observe that the post-filtering also improves the estimation of the low frequencies.
On the other hand, the pre-filtering reduces the frequencies that cannot be represented on the coarse grids and that lead to spurious low-frequency components, thus improving the variance in those low frequencies, corresponding to the very first entries of the cumulative variance vector $\bnucum$.
The joint use of pre- and post-filtering combines both benefits, further reducing the variance over the entire range of frequencies.

\begin{figure}[!ht]
    \centering%
    \subfigure[{Total variance}.\label{fig:illustration_total_var_pre_post}]{\includegraphics[width=.46\linewidth]{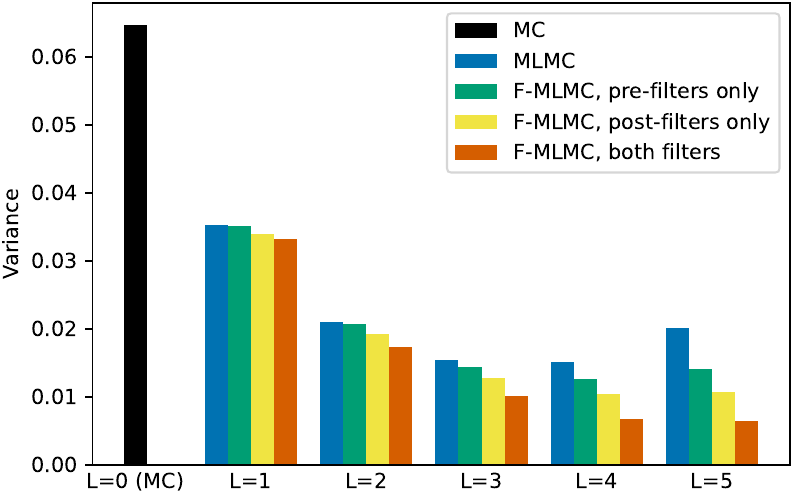}}\quad%
    \subfigure[{Cumulative variance}.\label{fig:illustration_cml_pre_post}]{\includegraphics[width=.44\linewidth]{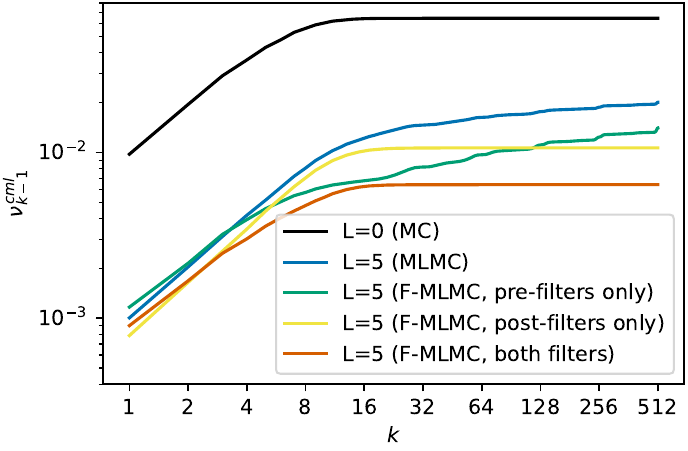}}
    \caption{Total variance (left) of the MC estimator ($L=0$) and different MLMC and F-MLMC estimators,  ($L \in \{1,\ldots,5\}$), and cumulative spectral variance (right) of the different estimators with $L=5$.
    Both figures are for the illustrative 1D estimation problem described in \cref{sec:1d_illustration}, with length-scale $D=0.06$.
    The finest level $L$ always corresponds to a discretization with $n_L=512$ cells, and the total budget is set to $\mathcal{C}=100\mathcal{C}_L$.
    The variance is estimated from \num{1000} estimators.}
    \label{fig:illustration_pre_post}
\end{figure}

It should be noted that the length-scale $D=0.06$ used for the experiments presented above induces output fields mostly composed of large scales (low frequencies). 
Decreasing its value increases the frequencies of the output field, thus introducing smaller scales. 
\Cref{fig:illustration_spectral_cml_001} shows the spectral and cumulative variance of the MC, MLMC and F-MLMC estimators for a smaller length-scale $D=0.01 \approx 5 n_L^{-1}$.
We observe that the decay of the spectral variance of the MC estimator as the frequency increases is slower than for $D=0.06$, so that the variance is concentrated on a wider low-frequency range.
For the MLMC estimators, we see that the significant deterioration of the variance in the high-frequencies is no longer compensated by a better estimation in the low frequencies, except for the 2-level estimator, which remains slightly better than the single-level estimator in terms of total variance. Again, the addition of filters improves the multilevel estimation in both the low and the high frequencies, which in turn benefits the cumulative and thus the total variance.
Filtering is here even more beneficial than for $D=0.06$, in the sense that the F-MLMC estimator has a significantly lower variance than the single-level MC estimator, as shown in \cref{fig:illustration_total_variance_001}, which the unfiltered MLMC estimator fails to achieve.
Specifically, the variance of the 4-level F-MLMC estimator is reduced by about 60\% compared to the single-level MC estimator, and by more than 50\% compared to the best, 2-level unfiltered MLMC estimator.

\begin{figure}[!ht]
    \centering%
    \subfigure[Spectral variance {$\bnu$} of MC and MLMC estimators.]{\includegraphics[width=.45\linewidth]{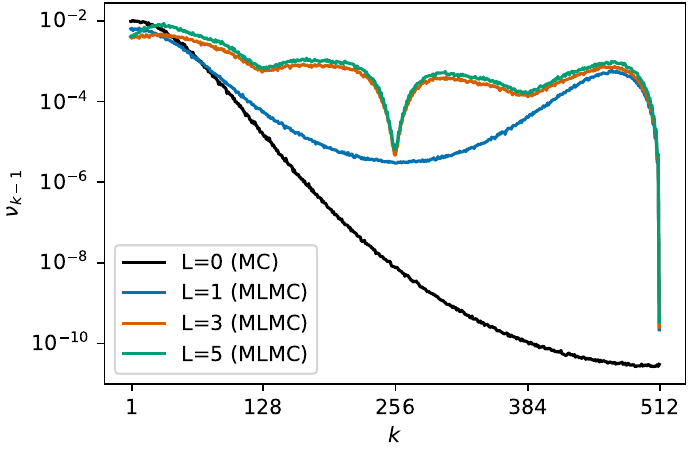}}\quad%
    \subfigure[Spectral variance {$\bnu$} of MC and F-MLMC estimators.]{\includegraphics[width=.45\linewidth]{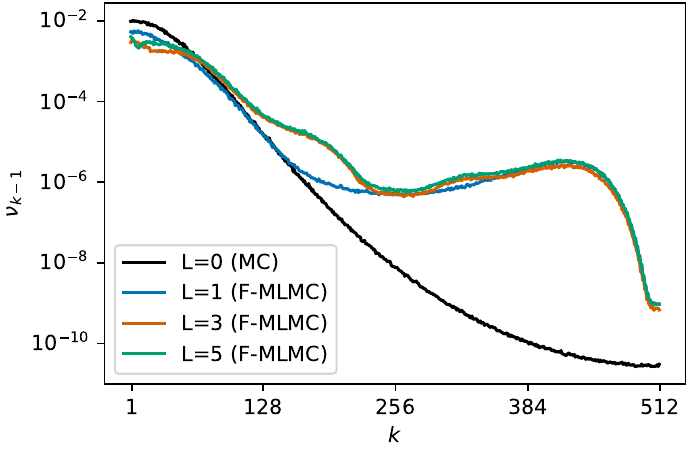}}\\
    \subfigure[Cumulative spectral variance {$\bnucum$} of MC and MLMC estimators.]{\includegraphics[width=.44\linewidth]{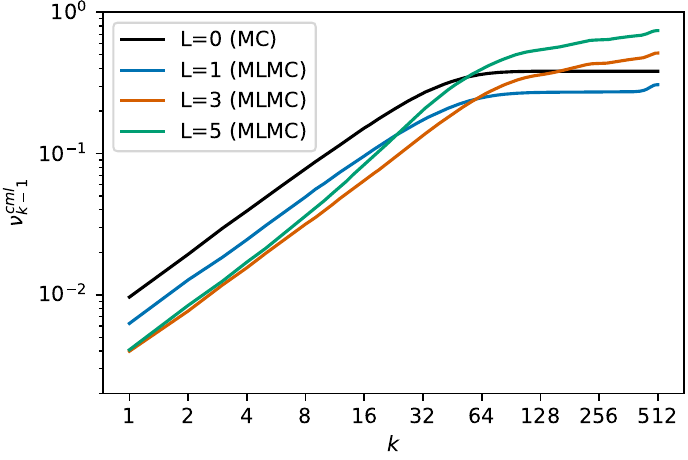}}\quad%
    \subfigure[Cumulative spectral variance {$\bnucum$} of MC and F-MLMC estimators.]{\includegraphics[width=.44\linewidth]{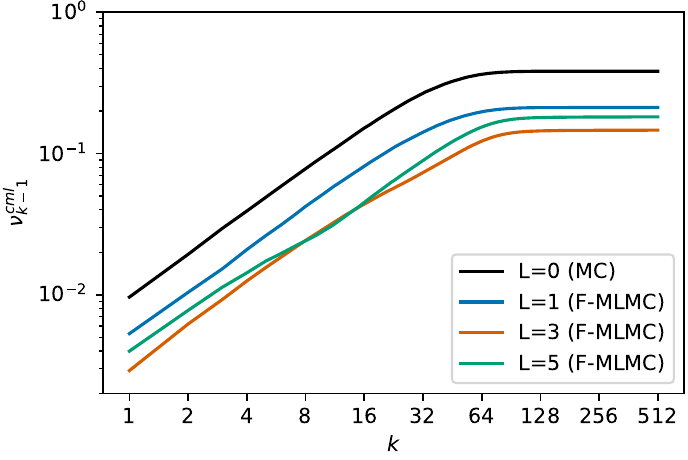}}
    \caption{Same as \cref{fig:illustration_spectral_cml_006} but with length-scale $D=0.01$.}
    \label{fig:illustration_spectral_cml_001}
\end{figure}

These experiments highlight that the effect of spurious high frequencies caused by grid transfer operations is detrimental to the MLMC estimator.
The addition of pre- and post-filtering operations is necessary to mitigate these effects, so that the multilevel estimator can reach its full potential.

\subsection{2D application}\label{sec:2D_application}

We now apply the MLMC and F-MLMC methodology to the variance estimation problem described in \cref{sec:problem_presentation}, i.e., the estimation of $\btheta$ in \cref{eq:theta}, in a 2D setting with more complex diffusion operators than those used in the 1D illustration of the previous section.
Specifically, the diffusivity field $\bK$ is now specified to be non-uniform, making the variance field also non-uniform. 
The considered domain is 
$\mathcal{D} = (0,2) \times (0,1) \subset \mathbb{R}^2$.
The boundary conditions are chosen periodic along both directions.
The 2D diffusivity tensor field $\bK$ is chosen to be diagonal and heterogeneous
\begin{equation}\label{eq:2d_diffusion_tensor_field}
    \bK =
    \frac{1}{2m-4}
    \begin{bmatrix}
        D_{11}^2 & 0 \\
        0 & D_{22}^2
    \end{bmatrix},
\end{equation}
where, for $i=1,2$, $D_{ii} \colon \mathcal{D} \to \mathbb{R}$ represents a length-scale field in the $i$-th direction. 
As previously, we let $m=2q=10$.
We model $D_{11}=\zeta(\omega_1)$ and $D_{22}=\zeta(\omega_2)$ as 
two different realizations of a 2D, periodic Gaussian random field $\zeta$ over $\mathcal{D}$ of uniform mean $\mu_\zeta$, and of quasi-Gaussian covariance structure with uniform variance $\sigma_\zeta^2 = (\mu_\zeta/5)^2 = 0.04 \mu_\zeta^2$ and uniform length-scale $D_\zeta$.
In the following experiments, two sets of parameters are considered for $\zeta$, namely $(\mu_\zeta = 0.12,D_\zeta = 0.2)$ and $(\mu_\zeta = 0.02,D_\zeta = 0.04)$.
The corresponding realizations used in the subsequent experiments are depicted in \cref{fig:Heterogeneous_D}.

\begin{figure}[!ht]
    \centering%
    \subfigure[{$\mu_\zeta = 0.12$} and {$D_\zeta = 0.2$}.\label{fig:Heterogeneous_D_012}]{\includegraphics[width=.45\linewidth]{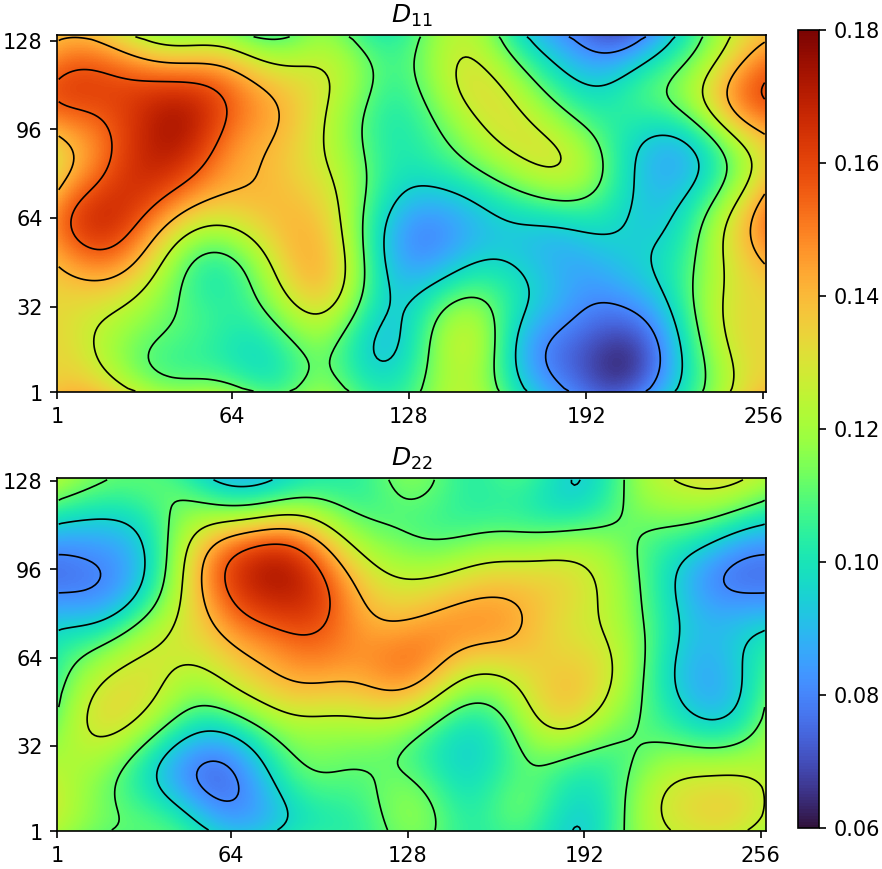}}\quad%
    \subfigure[{$\mu_\zeta = 0.02$} and {$D_\zeta = 0.04$}.\label{fig:Heterogeneous_D_002}]{\includegraphics[width=.46\linewidth]{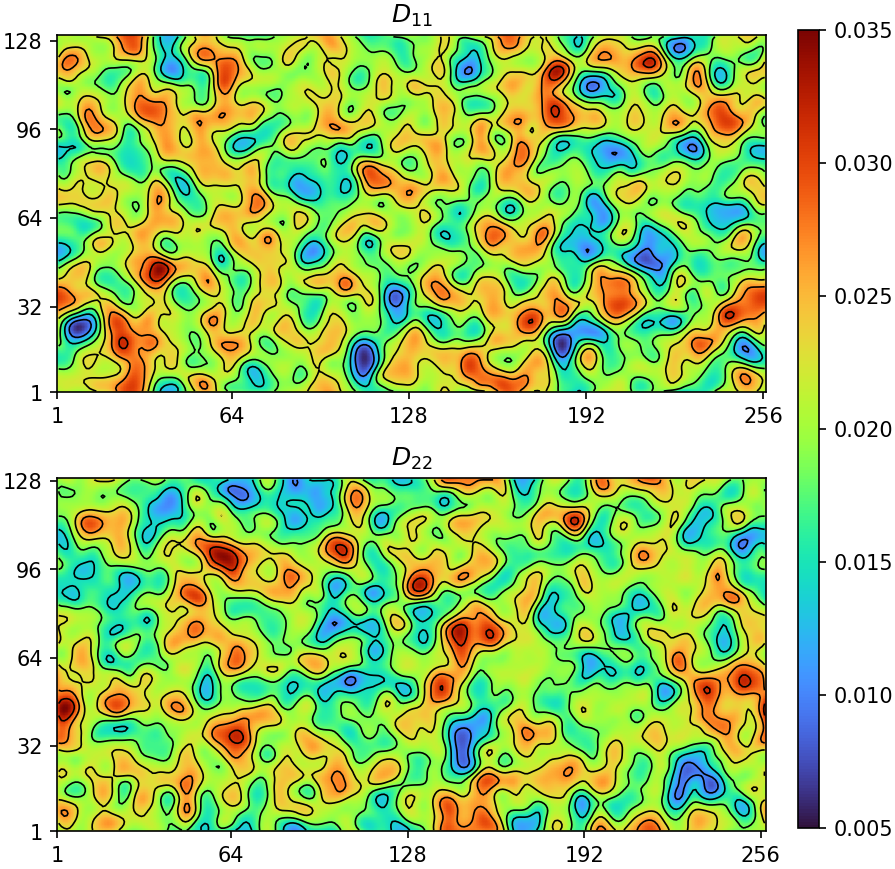}}
    \caption{Length-scale fields $D_{11}$ (top) and $D_{22}$ (bottom) used in the reported experiments for the definition of $\bK$ in \cref{eq:2d_diffusion_tensor_field}, with $\mu_\zeta = 0.12$ and $D_\zeta = 0.2$ (left), and with $\mu_\zeta = 0.02$ and $D_\zeta = 0.04$ (right).
    }
    \label{fig:Heterogeneous_D}
\end{figure}

The diffusion operator in \cref{eq:elliptic} is discretized on Cartesian grids of size $n_\ell = n_\ell^x \times n_\ell^y$, with the length-scale fields $D_{11}$ and $D_{22}$ discretized at edge centers, and the solution discretized at cell centers, as depicted in \cref{fig:2d_discr}.
The coarse operators $\bA_{\ell}$ are based on restrictions of $D_{11}$ and $D_{22}$, obtained by recursively averaging, at each point of a given coarse level, the values of the two nearest points of the immediately finer length-scale field, from the finest level $L$ down to the desired level $\ell < L$.
The finest grid considered here is composed of $n_L = 256\times 128$ cells. 
Three coarser grids are used with a uniform coarsening factor of 4, i.e., $n_{L-1} = 128\times64$, $n_{L-2} = 64\times32$ and $n_{L-3} = 32\times16$.
The resulting Gram matrix on level $\ell$ is given by 
$\bW_\ell = 2 n_\ell^{-1} \bI_{n_\ell}$.
The discrete operators $\bA_\ell \in \R^{n_\ell \times n_\ell}$ are designed to apply to and return vectors of size $n_\ell$ whose entries are associated with cell centers that are sorted by increasing $x$-coordinate first, then by increasing $y$-coordinate.
In other words, the entry indexed by $k=jn_\ell^x+i$ in such vectors is associated with a cell center located at $(x_i, y_j) \in \R^2$, where $x_i \bydef 2(i + 1/2) / n_\ell^x$ and $y_j \bydef (j + 1/2) / n_\ell^y$, for $i=0,\ldots,n_\ell^x-1$ and $j=0,\ldots,n_\ell^y-1$.

\begin{figure}[!ht]\centering%
\newcommand{\xedgecolor}{orange}%
\newcommand{\yedgecolor}{teal}%
\newcommand{\centercolor}{black}%
\DeclareRobustCommand\mytikzdot{\tikz{\node[fill=\centercolor,circle,inner sep=.48ex] {};}}%
\DeclareRobustCommand\mytikzdiam{\tikz{\node[fill=\yedgecolor,diamond,inner sep=.4ex] {};}}%
\DeclareRobustCommand\mytikzrect{\tikz{\node[fill=\xedgecolor,rectangle,inner sep=.6ex] {};}}%
\DeclareRobustCommand\mytikzdotc{\tikz{\node[draw=\centercolor,fill=none,circle,inner sep=.48ex] {};}}%
\DeclareRobustCommand\mytikzdiamc{\tikz{\node[draw=\yedgecolor,fill=none,diamond,inner sep=.4ex] {};}}%
\DeclareRobustCommand\mytikzrectc{\tikz{\node[draw=\xedgecolor,fill=none,rectangle,inner sep=.6ex] {};}}%
    \resizebox{0.7\linewidth}{!}{%
    \begin{tikzpicture}[scale=1.4,font=\footnotesize]%
    \draw[step=1,semithick,gray,dotted] (1.75,1.75) grid (6.25,4.25);
    \draw[step=2,semithick,darkgray] (1.75,1.75) grid (6.25,4.25);
    \foreach \x in {2.5,...,5.5}{
        \foreach \y in {2.5,...,3.5}{
            \node[fill=\centercolor,circle,inner sep=2pt] at (\x,\y) {};
        }
    }
    \foreach \x in {3,5}{
        \foreach \y in {3}{
            \node[draw=\centercolor,semithick,circle,inner sep=2pt] at (\x,\y) {};
        }
    }
    \foreach \x in {2,...,6}{
        \foreach \y in {2.5,...,3.5}{
            \node[fill=\xedgecolor,rectangle,inner sep=2.6pt] at (\x,\y) {};
        }
    }
    \foreach \x in {2,4,6}{
        \foreach \y in {3}{
            \node[draw=\xedgecolor,semithick,rectangle,inner sep=2.6pt] at (\x,\y) {};
        }
    }
    \foreach \x in {2.5,...,5.5}{
        \foreach \y in {2,...,4}{
            \node[fill=\yedgecolor,diamond,inner sep=1.8pt] at (\x,\y) {};
        }
    }
    \foreach \x in {3,5}{
        \foreach \y in {2,4}{
            \node[draw=\yedgecolor,semithick,diamond,inner sep=1.8pt] at (\x,\y) {};
        }
    }
    %
    \coordinate (leg) at (6.75,3);
    \node[fill=\yedgecolor,diamond,inner sep=1.8pt, anchor=center, above=1ex of leg, 
        label={right:{Fine $D_{22}$ points}}] (fdy) {};
    \node[fill=\xedgecolor,rectangle,inner sep=2.6pt, anchor=center, above=1.5ex of fdy, 
        label={right:{Fine $D_{11}$ points}}] (fdx) {};
    \node[fill=\centercolor,circle,inner sep=2pt, anchor=center, above=1.5ex of fdx,
        label={right:{Fine solution points}}] (fx) {};
    \node[draw=\centercolor,semithick,circle,inner sep=2pt, anchor=center, below=1ex of leg,
        label={right:{Coarse solution points}}] (cx) {};
    \node[draw=\xedgecolor,semithick,rectangle,inner sep=2.6pt, anchor=center, below=1.5ex of cx,
        label={right:{Coarse $D_{11}$ points}}] (cdx) {};
    \node[draw=\yedgecolor,semithick,diamond,inner sep=1.8pt, anchor=center, below=1.5ex of cdx,
        label={right:{Coarse $D_{22}$ points}}] (cdy) {};
\end{tikzpicture}%
%
    }
    \caption{Illustration of the 2D discretization of the differential operator ${\nabla \cdot \bK \nabla}$ in \cref{eq:elliptic} on a coarse uniform grid, represented by solid lines, whose sub-tessellation into a finer, geometrically nested, uniform grid is represented by dotted lines.
    The discretized solution points on the fine grid (respectively, on the coarse grid) are located at the center of the fine (respectively, coarse) cells and represented by~\mytikzdot{} (respectively, \mytikzdotc{}).
    The fine (respectively, coarse) discretized scalar field $D_{11}$ is defined at points located on the fine (respectively, coarse) vertical edges and represented by~\mytikzrect{} (respectively, \mytikzrectc{}).
    The fine (respectively, coarse) discretized scalar field $D_{22}$ is defined at points located on the fine (respectively, coarse) horizontal edges and represented by~\mytikzdiam{} (respectively, \mytikzdiamc{}).}
    \label{fig:2d_discr}
\end{figure}

The hierarchy of simulators $(f_\ell)_{\ell=0}^L$ is defined through \cref{eq:f_simu_restr_prolong_fine} by $\tilde{f}_{\ell} \colon \bx_\ell \mapsto (\bA_\ell \bx_\ell) \odot (\bA_\ell \bx_\ell)$, and we are interested in the multilevel estimation of the fine discretized field $\btheta_L \bydef \Exp[f_L(\bX_L)]$.
With the Cartesian ordering of the unknowns described above, the 2D prolongation and restriction operators may be constructed as the Kronecker product of their 1D counterpart, defined in \cref{eq:def_prolongation_restriction}, in the $x$ and $y$ directions.
The 2D grid transfer operators defined in this manner satisfy \cref{eq:prop_restr_prolong}.
Likewise, the second-order 2D Shapiro filter is defined as the Kronecker product of two 1D, second-order Shapiro filters defined in \cref{eq:def_filtering_operator}.
The cost model is still defined as in \cref{eq:cost_model}, with the 2D refinement factor $s=4$.
In two dimensions, the application of $(\bI_{n_\ell} - \bDelta_\ell)$ now requires 5 floating-point operations per cell on level $\ell$, consistent with a 5-point stencil discretization, so that $\alpha=500$.
The unfiltered 2D restriction still corresponds to $\beta=1$, while 6 additional operations are required to apply the 2D pre-filtering operator, consistent with two sweeps (one horizontal and one vertical) of 1D filters, each represented by a 3-point stencil, hence $\beta=7$ for F-MLMC.
\Cref{tab:cost_model_2d} summarizes the normalized costs $\mathcal{C}_{\ell} / \mathcal{C}_{L}$ and cost reduction factors $\mathcal{C}_{\ell}/ \mathcal{C}_{\ell-1}$ of the simulators $f_\ell$ (MLMC) and $\bar{f}_\ell$ (F-MLMC) for this 2D application.
The baseline for the following experiments is a crude, single-level MC ($L=0$) computed with a sample of size 100; hence the total budget is $\mathcal{C} = 100 \mathcal{C}_L$.

\begin{table}[!ht]%
  \centering%
  {\sisetup{round-mode=places, round-precision=2, round-pad=true}%
    \begin{tabular}{ccccc}
      \toprule
      & \multicolumn{2}{c}{\textbf{MLMC}} & \multicolumn{2}{c}{\textbf{F-MLMC}}\\
      & \multicolumn{2}{c}{($\gamma = 1/375$)} & \multicolumn{2}{c}{($\gamma = 7/375$)}\\
      \cmidrule(lr){2-3}\cmidrule(lr){4-5}
      {$\ell$} & {$\mathcal{C}_{\ell} / \mathcal{C}_{L}$} & {$\mathcal{C}_{\ell}/ \mathcal{C}_{\ell-1}$}
      & {$\mathcal{C}_{\ell} / \mathcal{C}_{L}$} & {$\mathcal{C}_{\ell}/ \mathcal{C}_{\ell-1}$}\\
      \midrule
      $L$   & \num{1}       & \num{3.96825396825} & \num{1}      & \num{3.787878787879} \\
      $L-1$ & \num{0.252}   & \num{3.87692307692} & \num{0.264}  & \num{3.3} \\
      $L-2$ & \num{0.065}   & \num{3.56164383562} & \num{0.08}   & \num{2.352941176471}\\
      $L-3$ & \num{0.01825} & N/A                 & \num{0.034}  & N/A \\
      \bottomrule
    \end{tabular}%
  }%
  \caption{Normalized costs $\mathcal{C}_{\ell} / \mathcal{C}_{L}$ and cost reduction factors $\mathcal{C}_{\ell}/ \mathcal{C}_{\ell-1}$ of the simulators $f_\ell$ (MLMC) and $\bar{f}_\ell$ (F-MLMC) corresponding to the 2D example of \cref{sec:2D_application}.}
  \label{tab:cost_model_2d}
\end{table}

\begin{figure}[!ht]
    \centering
    \includegraphics[width=.85\linewidth]{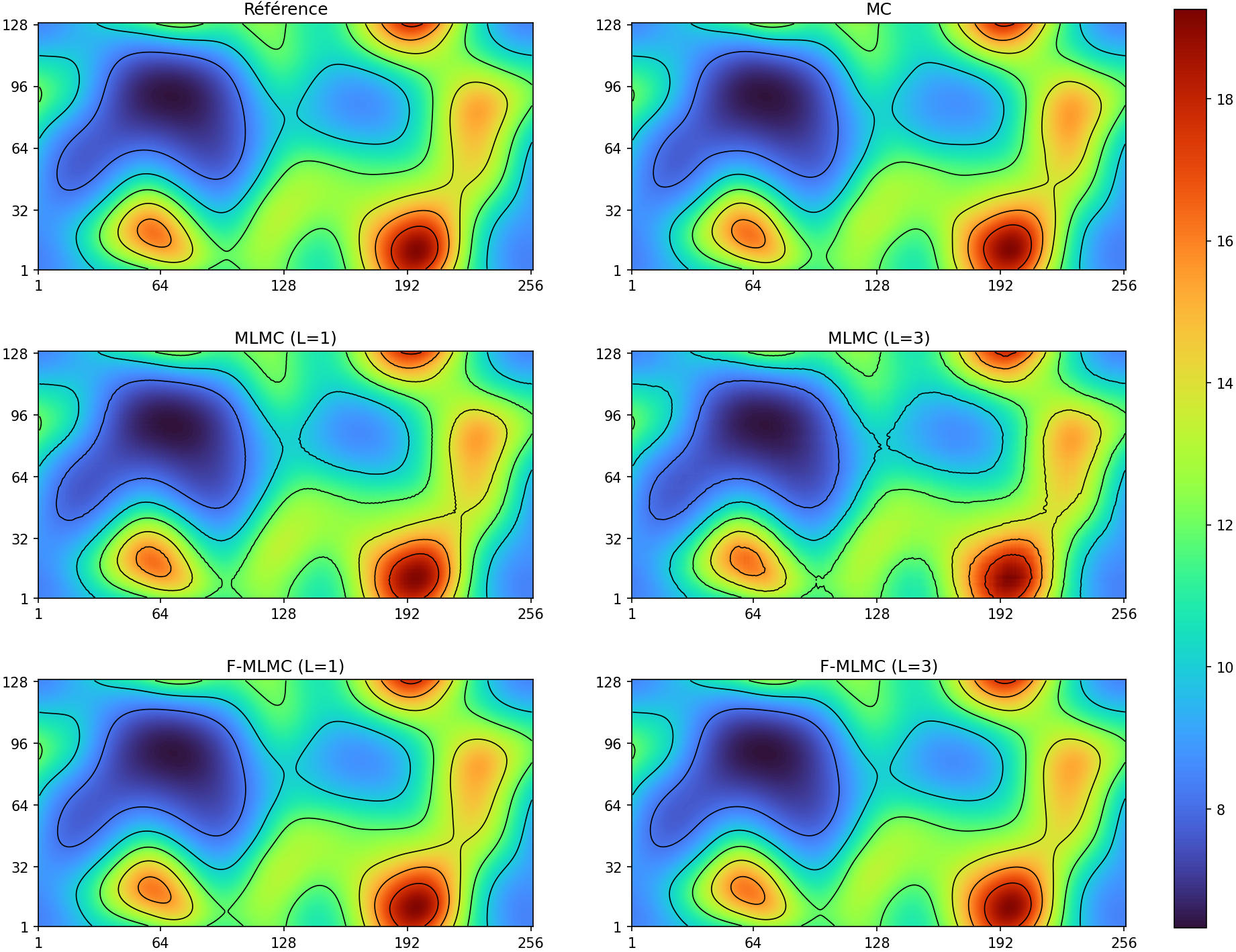}
    \caption{The top-left sub-figure shows the exact discretized field $\btheta$ on level $L$, while the other sub-figures depict the expectation of the single-level MC estimator (top-right), the expectation of the 2- and 4-level MLMC estimators (middle-left and middle-right), and the expectation of the 2- and 4-level F-MLMC estimators (bottom-left and bottom-right).
    The tensor field $\bK$ corresponds to $\mu_\zeta=0.12$ and $D_\zeta=0.2$ (\cref{fig:Heterogeneous_D_012}).
    The expectation is approximated from 500 estimators, each constructed with a computational budget $\mathcal{C} = 100 \mathcal{C}_L$.}
    \label{fig:2D_estimators_expectation}
\end{figure}

As discussed in \cref{sec:problem_presentation}, the $k$-th entry $(\btheta_L)_k$ of the exact discretized field $\btheta_L$ can be computed explicitly as 
$(\btheta_L)_k = (\bL_L \be_k)_k$, for $k=1,\ldots,n_L$,
where $\bL_L$ denotes the discrete diffusion operator defined in \cref{eq:diffusion_op} on the finest level $L$ and $\be_k$ denotes the $k$-th canonical basis vector of $\R^{n_L}$ (i.e., the $k$-th column of $\bI_{n_L}$).
The top-left sub-figure of \cref{fig:2D_estimators_expectation} shows the variance field $\btheta_L$ for the tensor field $\bK$ obtained by \cref{eq:2d_diffusion_tensor_field} from the length-scale fields $D_{11}$ and $D_{22}$ with parameters $\mu_\zeta=0.12$ and $D_\zeta=0.2$ depicted in \cref{fig:Heterogeneous_D_012}.
The other sub-figures represent the expectation of the compared estimators, namely the single-level MC estimator, and the 2- and 4-level (F-)MLMC estimators.
The expectation is approximated from 500 estimators, each constructed with a computational budget $\mathcal{C} = 100 \mathcal{C}_L$.
These figures confirm a key property of the MC and (F-)MLMC estimators, namely that they are unbiased.
Indeed, their expectations visually coincide (up to statistical error due to the estimation) with the reference, i.e., $\Exp[\hat{\btheta}] = \btheta$.
Consequently, the MSE solely consists of the variance of the estimators.

\begin{figure}[!ht]
    \centering
    \includegraphics[width=.85\linewidth]{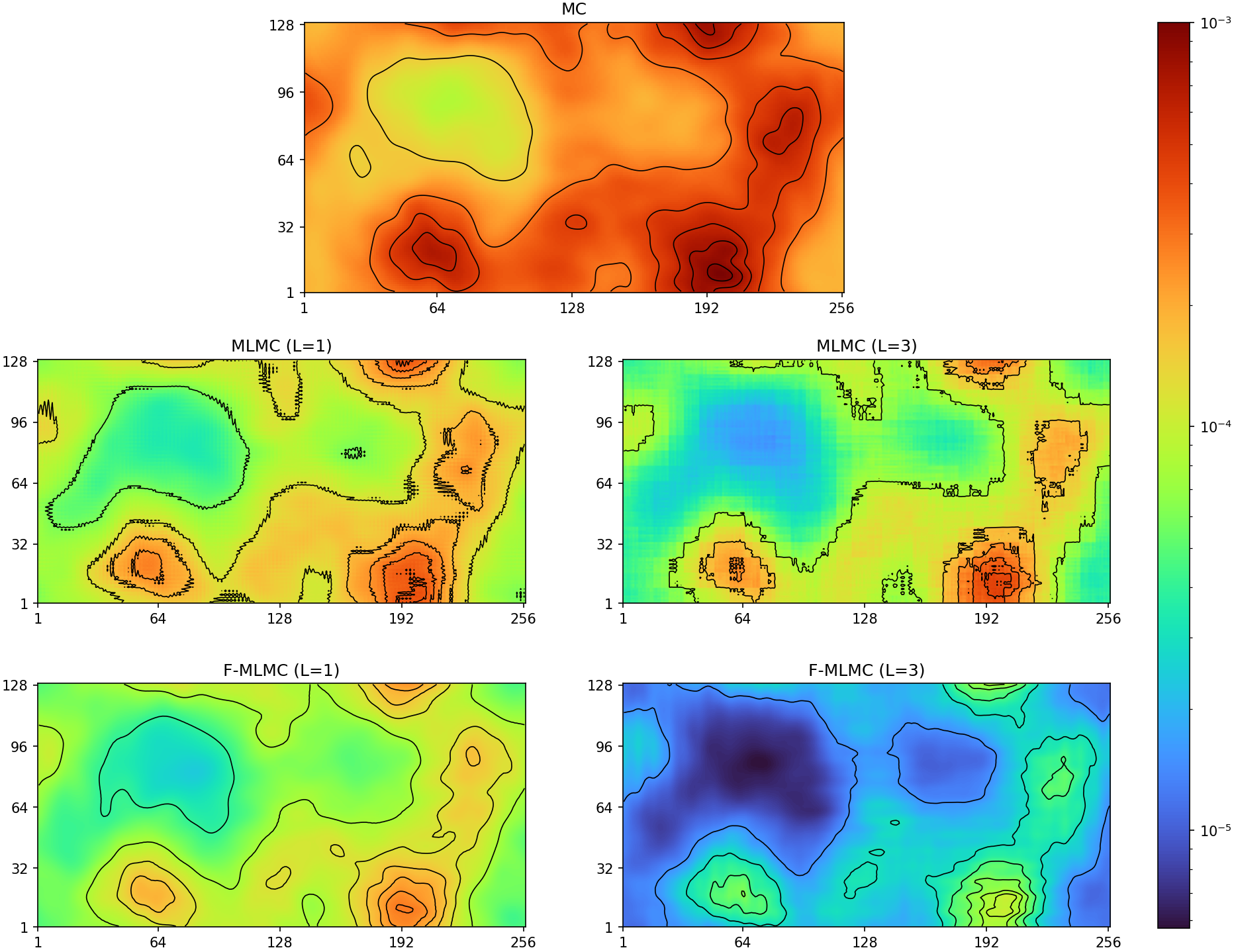}
    \caption{Variance of the MC estimator (top), of the 2- and 4-level MLMC estimators (middle-left and middle-right), and of the 2- and 4-level F-MLMC estimators (bottom-left and bottom-right). 
    The tensor field $\bK$ corresponds to $\mu_\zeta=0.12$ and $D_\zeta=0.2$ (\cref{fig:Heterogeneous_D_012}).
    The variance is approximated from 500 estimators, each constructed with a computational budget $\mathcal{C} = 100 \mathcal{C}_L$.}
    \label{fig:2D_estimators_variance}
\end{figure}

\Cref{fig:2D_estimators_variance} shows the variance of the considered estimators.
Although the variance of the MLMC estimators (middle row) visually is lower than that of the MC estimator (top row), we clearly observe high-frequency fluctuations in the variance field, both for the 2-level (left) and the 4-level (right) unfiltered MLMC estimators.
These high-frequency components of the variance may have significant consequences on individual estimations of $\btheta$.
Indeed, while the MLMC estimations will, on average, match the desired field (owing to the unbiasedness of the estimators), individual estimations will be polluted by high-frequency error components because of the large high-frequency components of the variance.
The bottom row of \cref{fig:2D_estimators_variance} demonstrates, at least visually, that the addition of filtering effectively damps these high-frequency components and reduced the variance.
These observations are confirmed by \cref{fig:2D_estimators_spectral_variance}, which shows the spectral decomposition of the variance in the Hartley space, allowing for the visualization of the contribution of each scale (or frequency) to the variance.
The Hartley matrix $\bH_L$ used to project the 2D variance (discretized) fields onto the Hartley spectral space is defined as the Kronecker product of two 1D Hartley matrices defined in \cref{eq:Hartley_basis}.
The 2D spectral variance of an unbiased estimator $\hat{\btheta}_L$ of $\btheta_L$, representing either the MC estimator on the finest level $L$, or the MLMC or F-MLMC estimator, is defined as $\bnu = (\nu_k)_{k=0}^{n_L-1} \in \R^{n_L}$, with
\begin{equation}\label{eq:2D_spectral_var_field}
    \forall  k = 0, \dots, n_L-1,
    \quad
    \nu_k 
    \bydef 4 n_L^{-2} \Exp[ ((\bh^L_k)^{\transp} (\hat{\btheta}_L - \btheta_L))^2 ]
    =
    4 n_L^{-2} \Var[ (\bh^L_k)^{\transp} \hat{\btheta}_L ],
\end{equation}
where $\bh^L_k$ denotes the $k$-th column of $\bH_L$.
Again, the columns of $\bH_L$ are re-ordered so that, in \cref{fig:2D_estimators_spectral_variance},
the frequencies in the $x$ and $y$ directions increase along the associated axes, starting from the lower-left corner, corresponding to low frequencies (large scales).
Specifically, the 2D Hartley matrix with reordered columns is constructed as the Kronecker product of two 1D Hartley matrices, each with reordered columns.
We observe that the variance of the unfiltered MLMC estimators (middle row) clearly exhibits larger high-frequency components than the single-level MC estimator (top row), with values of the same order of magnitude as the low-frequency components.
This not only translates into noticeable high-frequency fluctuations of the variance field, as evidenced in \cref{fig:2D_estimators_variance}, but it may also deteriorate the overall variance, as was the case in the 1D illustration (see \cref{sec:1d_illustration}).

\begin{figure}[!ht]
    \centering
    \includegraphics[width=.85\linewidth]{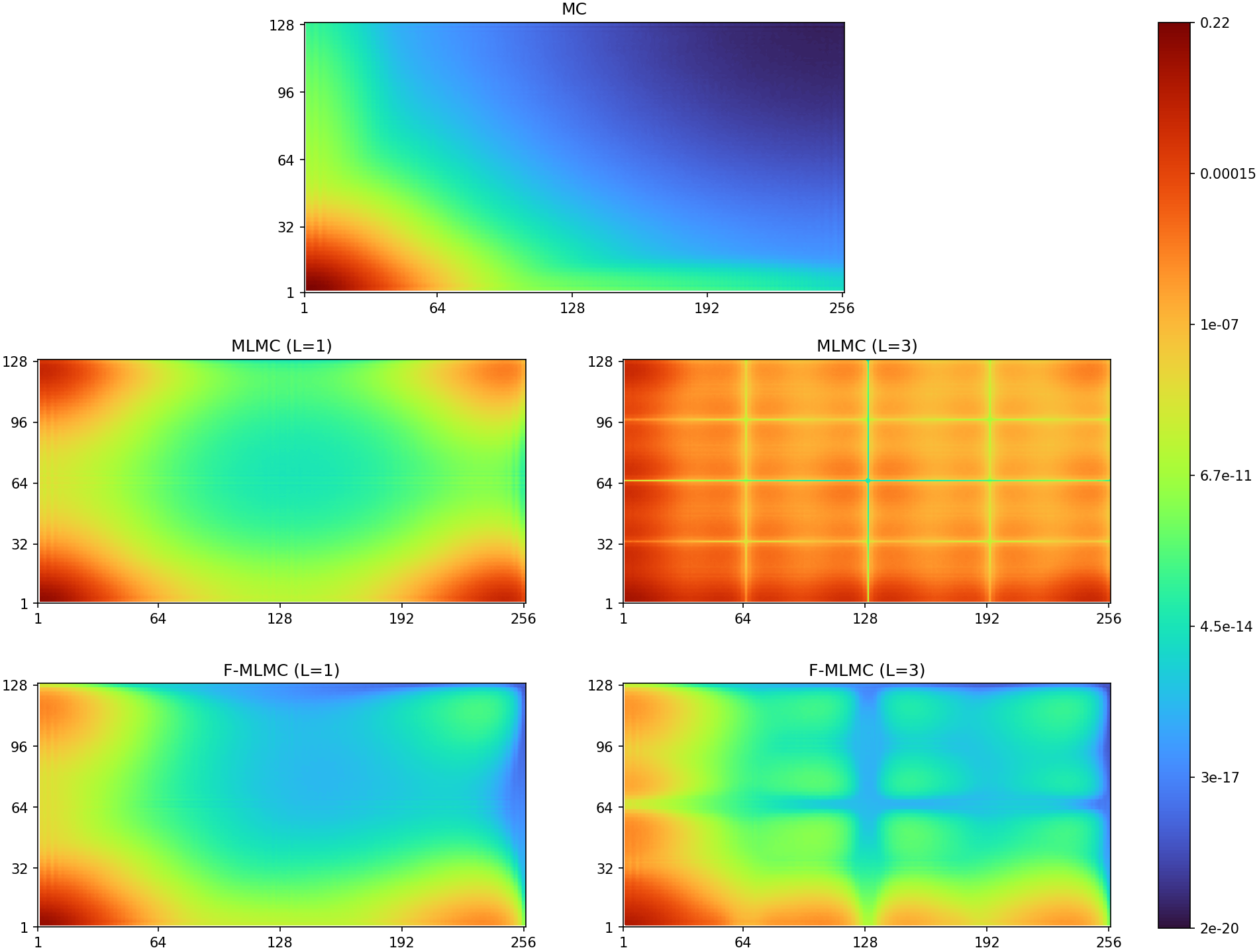}
    \caption{Spectral variance $\bnu$ of the MC estimator (top), of the 2- and 4-level MLMC estimators (middle-left and middle-right), and of the 2- and 4-level F-MLMC estimators (bottom-left and bottom-right). 
    The tensor field $\bK$ corresponds to $\mu_\zeta=0.12$ and $D_\zeta=0.2$ (\cref{fig:Heterogeneous_D_012}).
    The variance is approximated from 500 estimators, each constructed with a computational budget $\mathcal{C} = 100 \mathcal{C}_L$.}
    \label{fig:2D_estimators_spectral_variance}
\end{figure}

\begin{figure}[!ht]
    \centering%
    \subfigure[{$\bK$} with {$\mu_\zeta=0.12$} and {$D_\zeta=0.2$} (\cref{fig:Heterogeneous_D_012}).\label{fig:2D_cumulative_var_012}]{
        \parbox{\linewidth}{%
            \includegraphics[width=.45\linewidth]{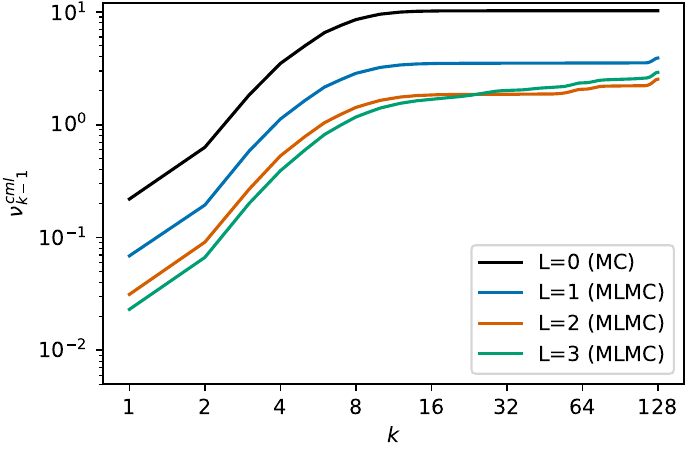}\quad%
            \includegraphics[width=.45\linewidth]{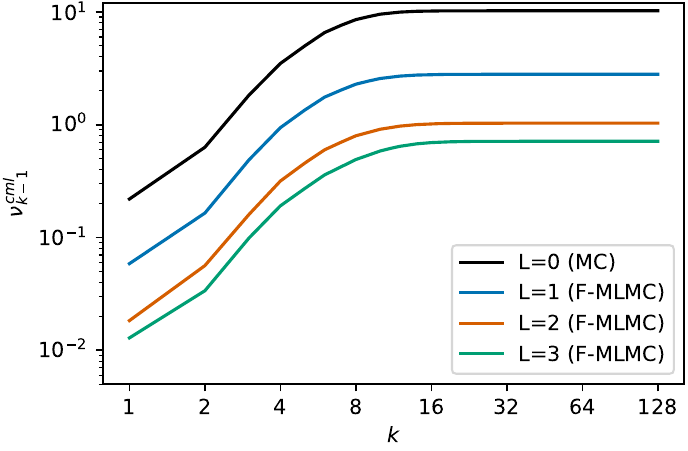}%
        }%
    }\\%
    \subfigure[{$\bK$} with {$\mu_\zeta=0.02$} and {$D_\zeta=0.04$} (\cref{fig:Heterogeneous_D_002}).\label{fig:2D_cumulative_var_002}]{
        \parbox{\linewidth}{%
            \includegraphics[width=.45\linewidth]{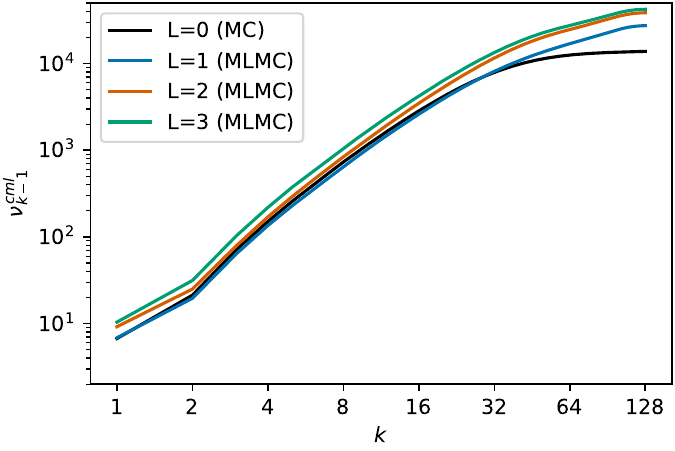}\quad%
            \includegraphics[width=.45\linewidth]{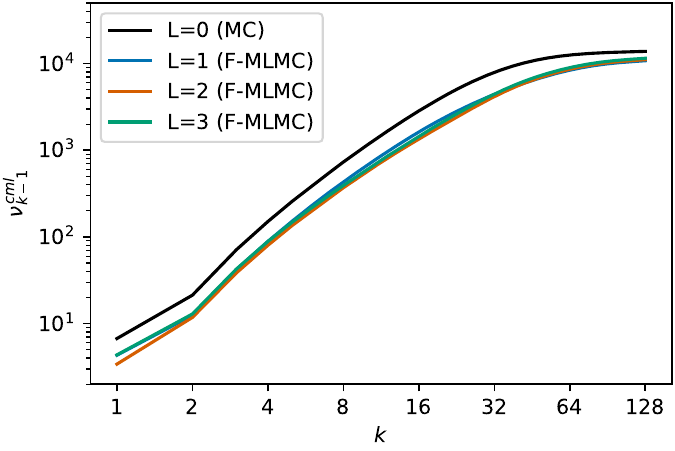}%
        }%
    }%
    \caption{Cumulative variance of the MC estimator ($L=0$) and of different MLMC estimators (left) and F-MLMC estimators (right) with $L \in \{1,2,3\}$. The tensor field $\bK$ corresponds to $\mu_\zeta=0.12$ and $D_\zeta=0.2$ (top), and $\mu_\zeta=0.02$ and $D_\zeta=0.04$ (bottom).
    The variance is approximated from 500 estimators, each constructed with a computational budget $\eta = 100 \mathcal{C}_L$.}
    \label{fig:2D_cumulative_var}
\end{figure}

With the Cartesian indexing and the re-ordering of the Hartley basis vectors described earlier, the 2D cumulative variance $\bnucum = (\nucum_k)_{k=0}^{n_L^y-1} \in \R^{n_L^y}$ is computed by adding the components of the 2D spectral variance \cref{eq:2D_spectral_var_field} shown in \cref{fig:2D_estimators_spectral_variance} in rectangle patterns 
starting with the bottom-left corner.
Specifically,
\begin{equation}
    \nucum_k
    = \sum_{j=0}^k \sum_{i=0}^{2k+1} \nu_{j n_L^x + i},
    \quad
    k = 0, \dots, n_L^y-1.
\end{equation}
As such, the first components of $\bnucum$ represent the cumulative variance associated with the low frequencies (large scales), while the last component $\nucum_{n_L^y-1}$ coincides with the total variance of the estimator.
\Cref{fig:2D_cumulative_var} presents the cumulative variance of the MLMC and F-MLMC estimators for $\bK$ corresponding to $(\mu_\zeta = 0.12,D_\zeta = 0.2)$ and $(\mu_\zeta = 0.02,D_\zeta = 0.04)$.
In the first case (\cref{fig:2D_cumulative_var_012}), for both the MC, the MLMC and the F-MLMC estimators, most of the variance is concentrated in the lower frequencies.
Although the low-frequency contribution to the variance is significantly reduced by the MLMC estimator, a non-negligible contribution of the high frequencies to the variance is noticeable.
Adding filters reduces the error in the high frequencies, as was observed previously in the spectral decomposition of the variance (\cref{fig:2D_estimators_spectral_variance}), but it also reduces the error on the lower frequencies, thus leading to lower total variance.
In the second case (\cref{fig:2D_cumulative_var_002}), corresponding to a tensor field $\bK$ with smaller scales (see \cref{fig:Heterogeneous_D_002}), the MLMC estimators deteriorate the variance compared to the crude MC estimator.
The addition of a coarser grid ($L=1$) and the corresponding grid transfer operators induce significant variance in the high-frequency components, leading to an increased total variance.
The degradation is all the more pronounced as coarser grids are added ($L=2$ and $L=3$).
We observe that the addition of filters mitigates these effects and allows the F-MLMC estimators to produce an effective reduction of the variance.
The addition of coarser grids ($L=2$ and $L=3$) does not bring further improvement, but does not deteriorate the total variance either.
The low-frequency (large-scale) components are already well-captured by the two finer grids.
This can be explained by the fact that the spectral content of the samples is shifted to higher frequencies.
In such a case, a different grid hierarchy, starting from a finer grid, would be more appropriate, although finer discretizations may not always be available or affordable in an operational context.
These results are summarized in terms of total variance in \cref{fig:2D_total_var}.

\begin{figure}[!ht]
    \subfigure[$\bK$ with $\mu_\zeta=0.12$ and $D_\zeta=0.2$ (\cref{fig:Heterogeneous_D_012}).\label{fig:2D_total_var_012}]{\includegraphics[width=.44\linewidth]{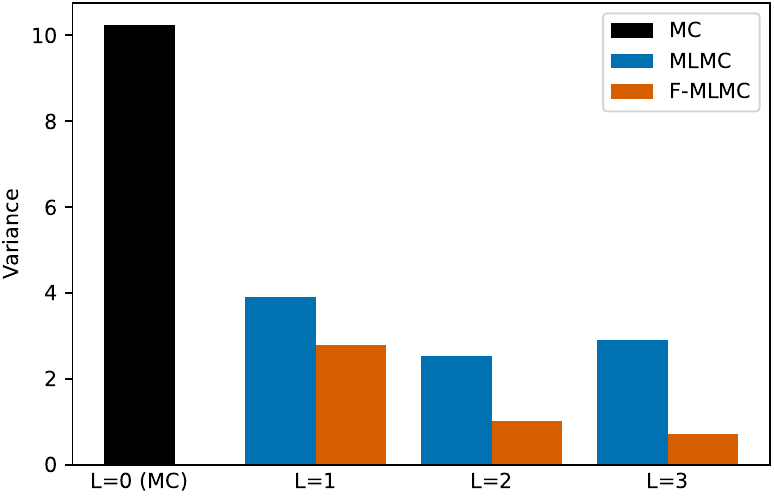}}\quad%
    \subfigure[$\bK$ with $\mu_\zeta=0.02$ and $D_\zeta=0.04$ (\cref{fig:Heterogeneous_D_002}).\label{fig:2D_total_var_002}]{\includegraphics[width=.45\linewidth]{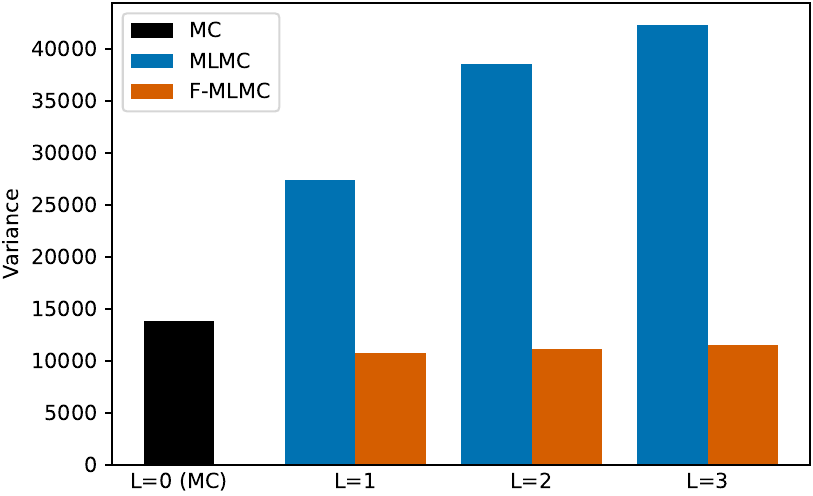}}
    \caption{Total variance of the MC estimator ($L=0$) and of different MLMC estimators and F-MLMC estimators with $L \in \{1,2,3\}$. The tensor field $\bK$ corresponds to $\mu_\zeta=0.12$ and $D_\zeta=0.2$ (left), and $\mu_\zeta=0.02$ and $D_\zeta=0.04$ (right).
    The variance is approximated from 500 estimators, each constructed with a computational budget $\mathcal{C} = 100 \mathcal{C}_L$.}
    \label{fig:2D_total_var}
\end{figure}

\section{Conclusion}\label{sec:conclusion}

In this paper, we focused on the estimation of the expectation of a discretized field using a multilevel, MLMC-like estimator.
The different fidelity levels considered are grids of different resolutions, which prompted us to introduce grid transfer operators in the estimator.
The resulting MLMC estimator can then be used to reduce the variance of the estimation compared to a crude MC estimator, as confirmed with an idealized 1D problem of estimating the discretized intrinsic variance field of a diffusion-based covariance operator.
However, projecting the variance of the MLMC estimator onto a spectral space revealed some discrepancy in the estimation of the different scales of the discretized field.
In our experiments, the MLMC estimator was still able to achieve a lower total variance by improving the estimation of the low-frequency (large-scale) components compared to the MC estimator, but at the expense of degrading the estimation in the higher-frequency (smaller-scale) components.

Inspired by multigrid methods, we proposed an improvement of the MLMC estimator by adding filtering operators, resulting in the F-MLMC estimator.
Filtering out the high-frequency components of a discretized field before restriction and after prolongation removes spurious features, thus yielding a better estimation of both the small- and large-scale components.
These improvements significantly impact the total variance of the F-MLMC estimators, which is also reduced compared to the MLMC estimators in our experiments.
In the specific case of linear, symmetric, circulant simulators, we quantified the effects of grid transfer and filtering operators on the total variance of (F-)MLMC estimators, which allowed us to improve our understanding of the influence of each ingredient.
The proposed F-MLMC estimators were applied to the problem of estimating the discretized intrinsic variance field of a 2D diffusion-based covariance operator with a non-uniform diffusivity field, which relies on non-linear simulators $\tilde{f}_\ell \colon \bx_\ell \mapsto (\bA_\ell \bx_\ell) \odot (\bA_\ell \bx_\ell)$.
The conclusions of these experiments were consistent with the theoretical results derived in the spectral analysis for linear, symmetric, circulant simulators.
Specifically, F-MLMC estimators do reduce the variance in both the low and high frequencies compared to their unfiltered counterparts, thus improving the total variance.
Even in experiments where MLMC estimators were not able to reduce (but actually deteriorated) the variance compared to a crude MC estimator, F-MLMC estimators still achieved lower total variance.

It should be noted that, for the particular problem of estimating discretized variance fields considered in this paper, the proposed MLMC and F-MLMC estimators are not guaranteed to be almost surely non-negative. 
Although negative estimates were not encountered in our experiments, this is nonetheless a serious limitation of the multilevel estimators, as already pointed out for the MLMC estimation of the variance of random variables (possibly with values in Hilbert spaces) \cite{bierig_convergence_2015}.
A similar issue exists for the estimation of covariance matrices, for which multilevel estimators are not guaranteed to be almost surely SPD \cite{hoel_multilevel_2016}, although advanced (but computationally expensive) approaches have been proposed to design multilevel estimators that are SPD by construction \cite{Maurais2023_MultiFidelityCovarianceEstimation, Maurais2025_MultifidelityCovarianceEstimation}.
These crucial issues still constitute an open research area.

Nonetheless, the investigations conducted in this paper and the proposed F-MLMC estimator expand the range of use of MLMC-like methods to discretized fields.
In our study, the use of second-order Shapiro filters demonstrated the benefits of applying pre- and post-smoothing at each level of the MLMC estimators.
A potential next step would be to investigate whether conditions can be derived for selecting the grid transfer and filtering operators, similar to the conditions stated in \cite{Hackbusch1985_MultiGridMethodsApplications,debreu_multigrid_2016} for multigrid methods.
Furthermore, extensions of the MLBLUE techniques \cite{Schaden2020_MultilevelBestLinear, Schaden2021_AsymptoticAnalysisMultilevel} to discretized fields may allow for the derivation of optimal spectral weights for each scale component of the discretized field to act as a post-prolongation filter \cite[sections {4.3--4.4}]{destouches_mlblue_2023}, possibly in combination with the spectral pre-restriction smoothing technique of \cite{Teckentrup2013_FurtherAnalysisMultilevel, Istratuca2025_SmoothedCirculantEmbedding}.
Another direction for future work could be to extend the filtering ideas developed in this paper to multilevel quasi Monte Carlo integration (see, e.g., \cite{Herrmann2019_MultilevelQuasiMonteCarlo,Croci2021_MultilevelQuasiMonte}).

As demonstrated numerically in \cref{sec:experiments}, the introduction of pre-restriction filtering operators significantly improves the multilevel estimator in terms of variance reduction. 
However, in our formulation, this comes at the expense of sampling the input fields on the finest level and resorting to finest-to-coarse restriction operators, thereby increasing the sampling cost on the coarse levels (see our discussion on the cost model in \cref{sec:experiments}).
It may thus be interesting to study the conditions under which the improvements in variance reduction outweigh the deterioration of the sampling cost.
Furthermore, it is sometimes possible to sample coupled pairs of input fields directly on their discretization level (as opposed to our formulation where the input fields are always sampled on the finest level).
Such alternative approaches would result in a more favourable cost model, since the use of restriction operators could then be avoided, but would require alternative smoothing/filtering strategies, e.g., the spectral strategies employed in \cite{Teckentrup2013_FurtherAnalysisMultilevel, Istratuca2025_SmoothedCirculantEmbedding}.
For our target problem of estimating the normalization coefficients of a diffusion-based covariance operator described in \cref{sec:problem_presentation}, future work could focus on alternative (F-)MLMC estimators inspired by \cite{croci_efficient_2018}, i.e., based on the joint distribution of the coupled input fields given by \cref{eq:mlmc_coupling}, thus avoiding the use of restriction operators.
In this case, spectral truncation is not readily available but could be performed, for example, after projecting the input fields onto a spectral basis.
A non-trivial challenge could then lie in maintaining the telescoping sum property of the resulting multilevel estimator to ensure its unbiasedness.
These ideas may be investigated in follow-up research projects.

Finally, an interesting research avenue would consist of exploiting further the hierarchical nature of MLMC techniques by combining them with multigrid iterative methods for solving the systems of linear equations involved in the normalization problem of \cref{sec:experiments}, using the same grid hierarchy, possibly along the lines of \cite{Osborn2017_MultilevelHierarchicalSampling}, \cite{kumar_multigrid_2017,Kumar2019_LocalFourierAnalysis,Robbe2019_RecyclingSamplesMultigrid,Robbe2021_EnhancedMultiindexMonte}, or \cite{Goodman1989_MultigridMonteCarlo,Kazashi2024_MultigridMonteCarlo}.


\section*{Acknowledgments}
This project has received financial support from the CNRS (Centre National de la Recherche Scientifique) through the 80|Prime program and the French national program LEFE (Les Enveloppes Fluides et l'Environnement).

\appendix
\section{Orthogonality of the Hartley matrix}%
\label{app:ortho_hartley}%
We focus here on a Hartley matrix $\bH$ of arbitrary size $n$,
\begin{equation}\label{eq:hartley_matrix_generic}
    (\bH)_{j,k} 
    \bydef 
    \cos\alpha_{jk} + \sin\alpha_{jk}, 
    \quad
    \alpha_{jk} \bydef \frac{(2j+1)k\pi}{n},
    \quad \forall j,k = 0,\dots,n-1.
\end{equation}
We have, for $i,j=0,\ldots,n-1$,
\begin{align}
    (\bH \bH^{\transp})_{i,j}
    & = \sum_{k=0}^{n-1} (\cos\alpha_{ik} + \sin\alpha_{ik}) (\cos\alpha_{jk} + \sin\alpha_{jk}) 
      = \sum_{k=0}^{n-1} \big[ \cos(\alpha_{ik}-\alpha_{jk}) + \sin(\alpha_{ik}+\alpha_{jk}) \big] \\
    & = \sum_{k=0}^{n-1} \cos\frac{2(i-j)k\pi}{n} + \sum_{k=0}^{n-1} \sin\frac{2(i+j+1)k\pi}{n}.
\end{align}
From~\cite[1.342 (1\&2)]{Gradshteyn2015_TableIntegralsSeries}, we deduce, for $i,j=0,\ldots,n-1$,
\begin{align}
    \sum_{k=0}^{n-1} \cos\frac{2(i-j)k\pi}{n} & = n \delta_{ij},
    & \sum_{k=0}^{n-1} \sin\frac{2(i+j+1)k\pi}{n} & = 0,
\end{align}
where $\delta_{ij}$ denotes the Kronecker delta, thus proving that $\bH \bH^{\transp} = n\bI_n$.
Similarly,
\begin{equation}
    (\bH^{\transp} \bH)_{i,j}
    = 
    \sum_{k=0}^{n-1} \cos\frac{(2k+1)(i-j)\pi}{n} + \sum_{k=0}^{n-1} \sin\frac{(2k+1)(i+j)\pi}{n},
\end{equation}
and $\bH^{\transp} \bH = n\bI_n$ follows from~\cite[1.342 (3\&4)]{Gradshteyn2015_TableIntegralsSeries}.

\section{Prolongation of Hartley vectors}%
\label{app:proof_prolong_hartley}%
For $j,k=0,\ldots,n_0-1$, 
\begin{equation}
    (\bP \bh^0_k)_{2j} 
    = (\bP \bh^0_k)_{2j+1} 
    = (\bh^0_k)_j
    = \cos\frac{(4j+2)k\pi}{n_1} + \sin\frac{(4j+2)k\pi}{n_1}.
\end{equation}
Upon writing $(4j+2)k\pi = (4j+1)k\pi + k\pi$ and applying elementary trigonometric identitites,
\begin{equation}
    (\bP \bh^0_k)_{2j}
    = c_k (\bh^1_k)_{2j}
    + 
    \sin\frac{k\pi}{n_1} \left(
        \cos\frac{(4j+1)k\pi}{n_1} - \sin\frac{(4j+1)k\pi}{n_1}
    \right).
\end{equation}
Noticing that
\begin{align}
    \sin\frac{k\pi}{n_1}
    & = \sin\mleft( \frac{(n_0+k)\pi}{2n_0} - \frac{\pi}{2} \mright)
      = -c_{n_0+k}, \\
    \cos\frac{(4j+1)k\pi}{n_1} 
    & = \cos\mleft( \frac{(4j+1)(n_0+k)\pi}{n_1} - \frac{\pi}{2} \mright)
      = \sin\frac{(4j+1)(n_0+k)\pi}{n_1}, \\
    \sin\frac{(4j+1)k\pi}{n_1} 
    & = \sin\mleft( \frac{(4j+1)(n_0+k)\pi}{n_1} - \frac{\pi}{2} \mright)
      = - \cos\frac{(4j+1)(n_0+k)\pi}{n_1},
\end{align}
we conclude that $(\bP \bh^0_k)_{2j} = c_k (\bh^1_k)_{2j} - c_{n_0+k} (\bh^1_{n_0+k})_{2j}$.
Then, upon writing $(4j+2)k\pi = (4j+3)k\pi - k\pi$, a similar derivation leads to $(\bP \bh^0_k)_{2j+1} = c_k (\bh^1_k)_{2j+1} - c_{n_0+k} (\bh^1_{n_0+k})_{2j+1}$, thus proving \cref{eq:prolong_hartley}.

\section{Restriction of Hartley vectors}%
\label{app:proof_restrict_hartley}%
For $j=0,\ldots,n_0-1$ and $k=0,\ldots, 2n_0-1$,
\begin{align}
    (\bR \bh^1_k)_j
    &= (\bh^1_k)_{2j} + (\bh^1_k)_{2j+1} \\
    &= \cos\frac{(4j+1)k\pi}{n_1} + \cos\frac{(4j+3)k\pi}{n_1}
        + 
        \sin\frac{(4j+1)k\pi}{n_1} + \sin\frac{(4j+3)k\pi}{n_1}\\
    & = 2 \cos\frac{(8j+4)k\pi}{2n_1} \cos\frac{2k\pi}{2n_1}
        + 
        2 \sin\frac{(8j+4)k\pi}{2n_1} \cos\frac{2k\pi}{2n_1}\\
    & = 2 c_k \left[
        \cos\frac{(2j+1)k\pi}{n_0} + \sin\frac{(2j+1)k\pi}{n_0}
    \right].
\end{align}
Hence, for $j,k=0,\ldots,n_0-1$, it follows immediately that $\bR \bh^1_k = 2 c_k \bh^0_k$. 
Furthermore,
\begin{equation}
    (\bR \bh^1_{n_0+k})_j
    =
    2 c_{n_0+k} \left[
        \cos\mleft( \frac{(2j+1)k\pi}{n_0} + \pi \mright)
        +
        \sin\mleft( \frac{(2j+1)k\pi}{n_0} + \pi \mright)
    \right]
    = -2 c_{n_0+k} \bh^0_k,
\end{equation}
thus proving \cref{eq:restrict_hartley}.

\section{Symmetric circulant matrices are diagonalizable in the Hartley basis}%
\label{app:symcirc_diag_hartley}%
In this appendix, we prove \cref{thm:cymcirc_diag_hartley_cell} below, which states that symmetric, circulant matrices can be diagonalized in the cell-centered Hartley basis $\bH$ defined by \cref{eq:hartley_matrix_generic}.
To do so, we start by recalling or proving results on the node-centered Fourier basis $\check{\bF} \in \mathbb{C}^{n \times n}$ and Hartley bases $\check{\bH}^\pm \in \R^{n \times n}$ defined by 
\begin{align}\label{eq:fourier_hartley_matrix_generic_node}
    \check{\bF} &\bydef \check{\bH}_c + i \check{\bH}_s, 
    &
    \check{\bH}^\pm &\bydef \check{\bH}_c \pm \check{\bH}_s, 
    &
    (\check{\bH}_c)_{j,k}
    & \bydef
    \frac{1}{\sqrt{n}} \cos\frac{2jk\pi}{n},
    &
    (\check{\bH}_s)_{j,k}
    &\bydef
    \frac{1}{\sqrt{n}} \sin\frac{2jk\pi}{n},
\end{align}
where $i\in\mathbb{C}$ denotes the unit imaginary number such that $i^2=-1$.
It is clear that $\check{\bH}_c$ and $\check{\bH}_s$ are real, symmetric matrices, and hence so are $\check{\bH}^\pm$,
while $\check{\bF}$ is a complex, symmetric (but not Hermitian) matrix.\
Furthermore, $\check{\bF}$ is unitary (see, e.g., \cite{Davis2012_CirculantMatricesSecond}), i.e., $\check{\bF}^*\check{\bF} = \check{\bF}\check{\bF}^*=\bI_n$, where $\check{\bF}^* = \check{\bH}_c - i \check{\bH}_s$ is the Hermitian transpose of $\check{\bF}$.

\begin{lemma}\label{lem:prop_hartley}
    $\check{\bH}_c$ and $\check{\bH}_s$ are such that
    \begin{align}
        \check{\bH}_c^2 + \check{\bH}_s^2  & = \bI_n,
        & \check{\bH}_c\check{\bH}_s & = \check{\bH}_s\check{\bH}_c = \bzero_n,
        & \check{\bH}_c \bone_n & = \sqrt{n} \be_1,
        & \check{\bH}_s \bone_n & = \bzero_n,
    \end{align}
    where $\bzero_n \bydef (0,\ldots,0)^{\transp} \in\R^n$, $\bone_n \bydef (1,\ldots,1)^{\transp} \in\R^n$,
    and
    $\be_1$ denotes the first column of $\bI_n$.
\end{lemma}
\begin{proof}
    The first two identities follow from \cite[Lemma~1]{bini_matrix_1993}, while the last two identities follow from~\cite[1.342 (1\&2)]{Gradshteyn2015_TableIntegralsSeries}. 
\end{proof}

\begin{corollary}\label{coro:prop_hartley}
    $\check{\bH}^\pm \bone_n = \sqrt{n} \be_1$.
\end{corollary}
\begin{corollary}\label{coro:hartley_ortho_mixed}
    $\check{\bH}^\pm$ are orthogonal, i.e., $(\check{\bH}^\pm)^2 = \bI_n$.
    Furthermore, $\check{\bH}^+\check{\bH}^- = \check{\bH}^-\check{\bH}^+ = \check{\bF}^2$.
\end{corollary}
\begin{proof}
    By the definitions \cref{eq:fourier_hartley_matrix_generic_node} and \cref{lem:prop_hartley}, we have
    \begin{align}
        \check{\bH}^\pm \bone_n
        & = \check{\bH}_c \bone_n \pm \check{\bH}_s \bone_n 
        = \check{\bH}_c \bone_n = \sqrt{n} \be_1,\\
        (\check{\bH}^\pm)^2
        & = (\check{\bH}_c^2 + \check{\bH}_s^2) 
        \pm (\check{\bH}_c\check{\bH}_s + \check{\bH}_s\check{\bH}_c)
        = \check{\bH}_c^2 + \check{\bH}_s^2 = \bI_n,\\
        \check{\bF}^2 
        & = \check{\bH}_c^2 - \check{\bH}_s^2 
        + i (\check{\bH}_c\check{\bH}_s + \check{\bH}_s \check{\bH}_c) = \check{\bH}_c^2 - \check{\bH}_s^2,\\
        \check{\bH}^-\check{\bH}^+ 
        & = \check{\bH}_c^2 - \check{\bH}_s^2 
        + (\check{\bH}_c\check{\bH}_s - \check{\bH}_s \check{\bH}_c) = \check{\bF}^2 = \check{\bH}^+\check{\bH}^-.
    \end{align}
\end{proof}

\begin{lemma}\label{lem:prop_hartley_circ}
    Let $\bA = \Circ(\ba) \in \R^{n \times n}$ be a symmetric, circulant matrix whose first column is $\ba = (a_k)_{k=0}^{n-1} \in \R^n$, with $a_{n-i}= a_i$, for $i=1,\ldots,n-1$.
    Then 
    $\check{\bH}_c \bA \check{\bH}_s + \check{\bH}_s \bA \check{\bH}_c
    =
    \check{\bH}_c \bA \check{\bH}_s - \check{\bH}_s \bA \check{\bH}_c
    = \bzero_n$.
\end{lemma}
\begin{proof}
    Proven in the intermediary steps of the proof of \cite[Theorem~1]{bini_matrix_1993}.
\end{proof}

\begin{theorem}\label{thm:cymcirc_diag_hartley_node}
    Let $\bA = \Circ(\ba) \in \R^{n \times n}$ be a symmetric, circulant matrix whose first column is $\ba = (a_k)_{k=0}^{n-1} \in \R^n$, with $a_{n-i}= a_i$, for $i=1,\ldots,n-1$.
    Then $\check{\bH}^+ \bA \check{\bH}^+ = \check{\bH}^- \bA \check{\bH}^- = \bLambda$, where
    $\bLambda = \sqrt{n} \Diag(\check{\bH}^+ \ba) = \sqrt{n} \Diag(\check{\bH}^- \ba) = \sqrt{n} \Diag(\check{\bH}_c \ba)$.
\end{theorem}
\begin{proof}
    The fact that $\check{\bH}^+ \bA \check{\bH}^+ = \sqrt{n} \Diag(\check{\bH}^+ \ba) = \bLambda = \check{\bH}_c \bA \check{\bH}_c + \check{\bH}_s \bA \check{\bH}_s$ follows from \cite[Theorem~1]{bini_matrix_1993}, since symmetric, circulant matrices belong to the larger class of matrices considered in \cite{bini_matrix_1993}.
    Then, from the definition of $\check{\bH}^-$ and \cref{lem:prop_hartley_circ}, we have 
    \begin{equation}
        \check{\bH}^- \bA \check{\bH}^-
        = (\check{\bH}_c \bA \check{\bH}_c + \check{\bH}_s \bA \check{\bH}_s) 
        - (\check{\bH}_c \bA \check{\bH}_s + \check{\bH}_s \bA \check{\bH}_c)
        = \check{\bH}_c \bA \check{\bH}_c + \check{\bH}_s \bA \check{\bH}_s 
        = \bLambda.
    \end{equation}
    Finally, from \cref{lem:prop_hartley} and \cref{coro:prop_hartley},
    \begin{equation}
        \diag(\bLambda) 
        = \bLambda \bone_n 
        = \check{\bH}^- \bA \check{\bH}^- \bone_n 
        = \check{\bH}_c \bA \check{\bH}_c \bone_n + \check{\bH}_s \bA \check{\bH}_s \bone_n
        = \sqrt{n} \check{\bH}^- \bA \be_1
        = \sqrt{n} \check{\bH}_c \bA \be_1,
    \end{equation}
    which concludes the proof, since $\bA \be_1 = \ba$.
\end{proof}


\begin{corollary}\label{coro:mixed_hartley_pm}
    Let $\bA = \Circ(\ba) \in \R^{n \times n}$ be a symmetric, circulant matrix and let $\bLambda$ be defined as in \cref{thm:cymcirc_diag_hartley_node}.
    Then $\check{\bH}^+\bA\check{\bH}^- = \check{\bH}^-\bA\check{\bH}^+ = \bLambda\check{\bF}^2 = \check{\bF}^2\bLambda$.
\end{corollary}
\begin{proof}
    From \cref{thm:cymcirc_diag_hartley_node} and \cref{coro:hartley_ortho_mixed}, we have
    \begin{align}
        \check{\bH}^+\bA\check{\bH}^- 
        & = \check{\bH}^+\bA\check{\bH}^+\check{\bH}^+\check{\bH}^- 
        = \check{\bH}^+\bA\check{\bH}^+ \check{\bF}^2 = \bLambda \check{\bF}^2,\\
        \check{\bH}^-\bA\check{\bH}^+
        & = \check{\bH}^-\bA\check{\bH}^-\check{\bH}^-\check{\bH}^+
        =  \check{\bH}^-\bA\check{\bH}^- \check{\bF}^2 = \bLambda \check{\bF}^2,\\
        (\check{\bH}^+\bA\check{\bH}^-)^{\transp}
        & = (\bLambda \check{\bF}^2)^{\transp} = \check{\bF}^2 \bLambda = \check{\bH}^-\bA\check{\bH}^+ = \bLambda \check{\bF}^2,
    \end{align}
    where the last row follows from the symmetry of $\bA$, $\check{\bH}^\pm$ and $\check{\bF}^2$.
\end{proof}

\begin{lemma}\label{lem:hartley_shift}
    $\bH = \sqrt{n}(\check{\bH}^+ \bC + \check{\bH}^- \bS)$, 
    with 
    $\bC \bydef \Diag(\{ \cos\frac{k\pi}{n} \}_{k=0}^{n-1})$
    and
    $\bS \bydef \Diag(\{ \sin\frac{k\pi}{n} \}_{k=0}^{n-1})$.
\end{lemma}
\begin{proof}
    From elementary trigonometric identities, we have
    \begin{align}
        \cos\frac{(2j+1)k\pi}{n} &= \cos\frac{2jk\pi}{n}\cos\frac{k\pi}{n} - \sin\frac{2jk\pi}{n}\sin\frac{k\pi}{n}, \\
        \sin\frac{(2j+1)k\pi}{n} &= \sin\frac{2jk\pi}{n}\cos\frac{k\pi}{n} + \cos\frac{2jk\pi}{n}\sin\frac{k\pi}{n},
    \end{align}
    then \cref{lem:hartley_shift} follows.
\end{proof}

\begin{lemma}\label{lem:C_F2_S_plus_S_F2_C}
    $\bC \check{\bF}^2 \bS + \bS \check{\bF}^2 \bC = \bzero_n$, where $\bC$ and $\bS$ are defined as in \cref{lem:hartley_shift}.
\end{lemma}
\begin{proof}
    From \cite[Lemma~1]{bini_matrix_1993}, we have that, for $j,k=0,\dots,n-1$, $(\check{\bF}^2)_{j,k} = 1$ if $j=k=0$ or $j+k = n$, and $(\check{\bF}^2)_{j,k} = 0$ otherwise.
    Moreover, $(\bC \check{\bF}^2 \bS)_{j,k} = (\check{\bF}^2)_{j,k} \cos\frac{j\pi}{n} \sin\frac{k\pi}{n}$, and, by symmetry, $\bS \check{\bF}^2 \bC = (\bC \check{\bF}^2 \bS)^{\transp}$.
    Now, if $j=k=0$, $(\bC \check{\bF}^2 \bS)_{j,k} = (\bS \check{\bF}^2 \bC)_{j,k} = 0$ trivially.
    Otherwise, if $j \ne 0$, $k \ne 0$, and $j+k \ne n$, $(\bC \check{\bF}^2 \bS)_{j,k} = (\bS \check{\bF}^2 \bC)_{j,k} = 0$.
    Finally, if $j+k=n$, elementary trigonometric identities induce
    \begin{align}
        (\bC \check{\bF}^2 \bS)_{j,k} 
        & =  \cos\frac{j\pi}{n} \sin\frac{(n-j)\pi}{n} 
        = \cos\frac{j\pi}{n} \sin\mleft(\pi - \frac{j\pi}{n}\mright)
        = \cos\frac{j\pi}{n} \sin\frac{j\pi}{n}, \\
        (\bC \check{\bF}^2 \bS)_{k,j} 
        & = \sin\frac{j\pi}{n} \cos\frac{(n-j)\pi}{n} 
        = \sin\frac{j\pi}{n} \cos\mleft(\pi - \frac{j\pi}{n}\mright) 
        = - \cos\frac{j\pi}{n} \sin\frac{j\pi}{n},
    \end{align}
    so that $(\bC \check{\bF}^2 \bS + \bS \check{\bF}^2 \bC)_{j,k} = (\bC \check{\bF}^2 \bS)_{j,k} + (\bC \check{\bF}^2 \bS)_{k,j} = 0$, thus completing the proof.
\end{proof}

We are now ready to state and prove \cref{thm:cymcirc_diag_hartley_cell}.
\begin{theorem}\label{thm:cymcirc_diag_hartley_cell}
    Let $\bA = \Circ(\ba) \in \R^{n \times n}$ be a symmetric, circulant matrix and let $\bLambda$ be defined as in \cref{thm:cymcirc_diag_hartley_node}.
    Then ${\bH^{\transp} \bA \bH = n\bLambda}$.
\end{theorem}
\begin{proof}
    Starting from the expression of \cref{lem:hartley_shift}, and applying \cref{thm:cymcirc_diag_hartley_node}, \cref{coro:mixed_hartley_pm} and \cref{lem:C_F2_S_plus_S_F2_C}, we obtain
    \begin{align}
        n^{-1}\bH^{\transp} \bA \bH
        & = \bC \check{\bH}^+ \bA \check{\bH}^+ \bC + \bS \check{\bH}^- \bA \check{\bH}^- \bS
        + \bC \check{\bH}^+ \bA \check{\bH}^- \bS + \bS \check{\bH}^- \bA \check{\bH}^+ \bC\\
        & = \bC \bLambda \bC + \bS \bLambda \bS + \bC \bLambda \check{\bF}^2 \bS + \bS \bLambda \check{\bF}^2 \bC
        = \bLambda( \bC^2 + \bS^2 + \bC \check{\bF}^2 \bS + \bS \check{\bF}^2 \bC),
    \end{align}
    so that $n^{-1}\bH^{\transp} \bA \bH = \bLambda( \bC^2 + \bS^2) = \bLambda$ since $\bC^2 + \bS^2 = \bI_n$ by elementary trigonometry.
\end{proof}

\section{Optimality of the Galerkin operator}%
\label{sec:annex_galerkin}%
Let $n_1, n_0 \in \mathbb{N}$ be such that $n_1 > n_0 > 0$,
and let $\bA_1 \in \R^{n_1 \times n_1}$ and $\bA_0 \in \R^{n_0 \times n_0}$ be two linear operators.
Let $\bR \in \R^{n_0 \times n_1}$ and $\bP \in \R^{n_1 \times n_0}$ be a restriction operator and a prolongation operator, respectively, such that 
$\bR = \bP^{\transp}$ and $\bP^{\transp} \bW_1 \bP = \bW_0$ (see \cref{eq:prop_restr_prolong}).
We seek the operator $\bA_0^*$ that minimizes $\| \bA_1 - \bP \bA_0 \bR \|^2_{\textrm{F}, \bW_1}$ among the linear operators $\bA_0 \in \R^{n_0 \times n_0}$.
We restrict ourselves to the case where, for $\ell \in \{0,1\}$, $\bW_{\ell} = n_{\ell}^{-1} \bI_{n_{\ell}}$, so that $\bP^{\transp} \bP = \bR \bR^{\transp} = (n_1/n_0)\bI_{n_0}$ and
\begin{equation}\label{eq:Galerkin_optim}
    \bA_0^* 
    \bydef 
    \argmin_{\bA_0  \in \R^{n_0 \times n_0}} \| \bA_1 - \bP \bA_0 \bR \|^2_{\textrm{F}, \bW_1}
    =
    \argmin_{\bA_0  \in \R^{n_0 \times n_0}} \| \bA_1 - \bP \bA_0 \bR \|^2_{\textrm{F}},
\end{equation}
where $\|\cdot\|_{\textrm{F}}$ denotes the classical (unweighted) Frobenius norm.
We start by writing
\begin{align}
    \| \bA_1 - \bP \bA_0 \bR \|^2_{\textrm{F}}
    &= \| \bA_1 - \bP \bA_0 \bR \|^2_{\textrm{F}}
    = \| \vect(\bA_1 - \bP \bA_0 \bR) \|^2_2 \\
    &= \| \vect(\bA_1) - \vect(\bP \bA_0 \bR) \|^2_2 
    = \| \vect(\bA_1) -  (\bR^{\transp} \otimes \bP) \vect(\bA_0) \|^2_2,
\end{align}
so that \cref{eq:Galerkin_optim} is recast as an ordinary linear least squares problem
whose solution is
\begin{align}
    \vect(\bA_0^*)
    & = \left( (\bR^{\transp} \otimes \bP)^{\transp} (\bR^{\transp} \otimes \bP) \right)^{-1} (\bR^{\transp} \otimes \bP)^{\transp} \vect(\bA_1) \\
    &= \left( (\bR \otimes \bP^{\transp}) (\bR^{\transp} \otimes \bP) \right)^{-1} (\bR \otimes \bP^{\transp}) \vect(\bA_1) \\
    &= ( \bR \bR^{\transp} \otimes \bP^{\transp} \bP )^{-1} \vect(\bP^{\transp} \bA_1 \bR^{\transp})
    = \left( (n_1/n_0)\bI_{n_0} \otimes (n_1/n_0) \bI_{n_0} \right)^{-1} \vect(\bR \bA_1 \bP) \\
    &=  (n_0/n_1)^2  (\bI_{n_0} \otimes \bI_{n_0}) \vect(\bR \bA_1 \bP) 
    =  (n_0/n_1)^2  \vect(\bR \bA_1 \bP),
\end{align}
thus proving that
$    \bA_0^* 
    = (n_0/n_1)^2 \bR \bA_1 \bP
$.

\section{Filtered restriction of Hartley vectors}%
\label{app:proof_filtered_prol_restr_hartley}%
We start by looking at the effect of the Shapiro filter on the fine Hartley basis vectors,
\begin{equation}
    \forall j,k=0,\ldots,n_1-1,
    \quad
    (\bS_1 \bh^1_k)_j
    = \frac{1}{4}
    \left[
        (\bh^1_k)_{j-1} + 2 (\bh^1_k)_{j} + (\bh^1_k)_{j+1}
    \right].   
\end{equation}
Elementary trigonometric identities imply that, for $j,k=0,\ldots,n_1-1$,
\begin{align}
    \cos\frac{(2j-1)k\pi}{n_1} + \cos\frac{(2j+1)k\pi}{n_1} &= 2c_k\cos\frac{2jk\pi}{n_1}, 
    \label{eq:app_trig_Sh_1}\\
    \cos\frac{(2j+1)k\pi}{n_1} + \cos\frac{(2j+3)k\pi}{n_1} &= 2c_k\cos\frac{(2j+2)k\pi}{n_1}, \\
    \sin\frac{(2j-1)k\pi}{n_1} + \sin\frac{(2j+1)k\pi}{n_1} &= 2c_k\sin\frac{2jk\pi}{n_1}, \\
    \sin\frac{(2j+1)k\pi}{n_1} + \sin\frac{(2j+3)k\pi}{n_1} &= 2c_k\sin\frac{(2j+2)k\pi}{n_1},
    \label{eq:app_trig_Sh_4}
\end{align}
leading to
\begin{equation}
    (\bS_1 \bh^1_k)_j
    = \frac{c_k}{2}
    \left[
    \cos\frac{2jk\pi}{n_1} + \cos\frac{(2j+2)k\pi}{n_1}
    + \sin\frac{2jk\pi}{n_1} + \sin\frac{(2j+2)k\pi}{n_1}
    \right]
    =
    c_k^2 (\bh^1_k)_j,
\end{equation}
where the last identity is obtained by applying the same trigonometric identities as for \crefrange{eq:app_trig_Sh_1}{eq:app_trig_Sh_4}, and thus showing that $\bS_1 \bh^1_k = c_k^2 \bh^1_k$, for $k=0,\ldots,n_1-1$.
Then, \cref{eq:prolong_hartley} implies that, for $k=0,\ldots,n_0-1$,
$
    \bar{\bP} \bh^0_k = \bS_1 \bP \bh^0_k 
    =
    c_k \bS_1 \bh^1_k - c_{n_0+k} \bS_1 \bh^1_{n_0+k}
$,
from which \cref{eq:filtered_prolong_hartley} follows.
Furthermore, \cref{eq:filtered_restrict_hartley} follows immediately from the direct application of \cref{eq:restrict_hartley} to $\bar{\bR} \bh^1_k = \bR \bS_1 \bh^1_k = c_k^2 \bR\bh^1_k$.

\section{Factorization of the diffusion-based covariance matrix}%
\label{app:facto_L}%

Let $\bDelta$ and $\bW$ as defined in \cref{sec:experiments}.
Because $\bDelta$ is self-adjoint with respect to $\langle \cdot, \cdot \rangle_{\bW}$, 
we have
\begin{equation}
    \bW(\bI - \bDelta) 
    =
    \bW - \bW \bDelta = \bW - \bDelta^{\transp}\bW = (\bI-\bDelta)^{\transp} \bW,
\end{equation}
and thus, for $q\in\N$,
\begin{equation}
    \bW(\bI - \bDelta)^q
    = \bW(\bI - \bDelta)(\bI - \bDelta)^{q-1}
    = (\bI-\bDelta)^{\transp} \bW (\bI - \bDelta)^{q-1} 
    = \cdots 
    = [(\bI-\bDelta)^{\transp}]^q \bW. 
\end{equation}
It follows that
\begin{equation}
    \bA 
    \bydef
    (\bI - \bDelta)^{-q} \bW^{-1}
    = ( \bW(\bI - \bDelta)^q )^{-1}
    = ( [(\bI-\bDelta)^q]^{\transp} \bW )^{-1}
    = \bW^{-1} [(\bI-\bDelta)^{-q}]^{\transp},
\end{equation}
and thus
\begin{equation}
    \bL 
    \bydef
    (\bI - \bDelta)^{-2q} \bW^{-1}
    = (\bI - \bDelta)^{-q} \bA
    = (\bI - \bDelta)^{-q} \bW^{-1} [(\bI-\bDelta)^{-q}]^{\transp}
    = \bA \bW \bA^{\transp}.
\end{equation}

\printbibliography

\end{document}